\documentclass[11pt,oneside]{article}
\usepackage[utf8]{inputenc}

\usepackage[english]{babel}
\usepackage[a4paper]{geometry}
\usepackage{indentfirst}

\usepackage{amsmath,amssymb,amsfonts,amsthm,mathrsfs,stmaryrd}
\usepackage{graphicx,caption,float}
\usepackage[usenames, dvipsnames]{xcolor}
\usepackage{yhmath}
\usepackage[linesnumbered, ruled, vlined]{algorithm2e}
\usepackage{esvect}
\usepackage{longtable}

\usepackage[hang,flushmargin]{footmisc} 

\usepackage{enumitem}
\usepackage{multirow}

\usepackage{caption}
\usepackage{yhmath}

\usepackage[colorlinks=true, pdfstartview=FitV,linkcolor=ForestGreen,citecolor=ForestGreen, urlcolor=blue]{hyperref}

\author{
Ramona Anton$^{\small 1}$,
Nicolae Mihalache$^{\small 2}$
and
Fran\c{c}ois Vigneron$^{\small 3}$
}

\newcommand{\authorramona}{{\small
Sorbonne Univ, IMJ-PRG, CNRS UMR 7586, 75252 Paris, France
}}

\newcommand{\authornicu}{{\small
Univ Paris Est Creteil, CNRS, LAMA, F-94010 Creteil, France and\\
Univ Gustave Eiffel, LAMA, F-77447 Marne-la-Vallée, France
}}

\newcommand{\authorfv}{{\small
\small Universit\'{e} de Reims Champagne-Ardenne,
Laboratoire de Mathématiques de Reims,
UMR 9008 CNRS, Moulin de la Housse, BP 1039,
F-51687 Reims --
\texttt{francois.vigneron@univ-reims.fr}
}}

\newtheorem{thm}{Theorem}
\newtheorem{proposition}{Proposition}
\newtheorem{definition}[proposition]{Definition}
\newtheorem{lemma}[proposition]{Lemma}
\newtheorem{remark}[proposition]{Remark}
\newtheorem{example}[proposition]{Example}

\newcommand{\N}{\mathbb{N}}
\newcommand{\Z}{\mathbb{Z}}

\newcommand{\R}{\mathbb{R}}
\newcommand{\C}{\mathbb{C}}
\newcommand{\CC}{\overline{\mathbb C}}
\newcommand{\RR}{\overline{\R}}

\DeclareMathOperator{\avg}{avg}

\DeclareMathOperator{\argmax}{argmax}

\newcommand\SetDef[2]{\ensuremath{\left\{{#1}\, : \, {#2}\right\}}}
\newcommand\ii[2]{\ensuremath{\llbracket{#1}, {#2} \rrbracket}}

\DeclareMathOperator{\ulp}{ulp}

\renewcommand{\Re}{\operatorname{Re}}
\renewcommand{\Im}{\operatorname{Im}}
\DeclareMathOperator{\eval}{V}
\newcommand{\defequal}{=}

\newcommand{\eg}[1][~]{\textit{e.g.}#1}
\newcommand{\ie}[1][~]{\textit{i.e.}#1}

\newcommand{\F}{\mathcal{C}} 
\newcommand{\A}{\textrm{FPE}}

\def\cod#1{\texttt{#1}}

\begin{document}
\title{Fast Evaluation of Real and Complex Polynomials}
\maketitle

\begin{abstract}  
We propose an algorithm for quickly evaluating polynomials.
It pre-conditions a complex polynomial $P$ of degree $d$ in time $O(d\log d)$, with a low multiplicative constant independent of the precision.
Subsequent evaluations of $P$ computed with a fixed precision of $p$ bits are performed in \textit{average} 
arithmetic complexity $O\big(\sqrt{d(p+\log d)}\big)$ and memory $O(dp)$.
The average complexity is computed with respect to points $z \in \C$,
weighted by the spherical area of~$\CC$.
The worst case does not exceed the complexity of H\"orner's scheme.

In particular, our algorithm performs asymptotically as $O(\sqrt{d\log d})$ per evaluation.
For many classes of polynomials, in particular those with random coefficients in a bounded region of $\C$,
or for sparse polynomials, our algorithm performs much better than this upper bound, without any modification or parameterization.

The article contains a detailed analysis of the complexity and a full error analysis,
which guarantees that the algorithm performs as well as H\"orner's scheme, only faster.
Our algorithm is implemented in a companion library, written in standard \texttt{C} and released as an open-source project~\cite{FPELib}.
Our claims regarding complexity and accuracy are confirmed in practice by a set of comprehensive benchmarks.

\textbf{MSC primary:} 68W40\\  
\textbf{MSC secondary:}
03D15, 
68Q25, 
68-04, 
68W01, 
12Y05. 
\end{abstract}

\tableofcontents

\newpage
\section{Introduction}
\label{sect:intro}

The study of polynomials has sparked the interest of many generations of mathematicians and inspired major theoretical developments.
In modern algebra, the notion of group stemmed from the impossibility of solving polynomials
with radicals; abstract rings generalize the properties of $\Z$ and $\Z[X]$. Modern number theory
is indissociable from polynomials and algebraic curves.

Modern analysis too evolved from the prototype of a function space given by polynomials.
A few obvious testimonies to this heritage are
Descartes's notation of $x$, $y$,\ldots{} for the variables of functions, 
the fact that successive approximations of a real number in base $b$ are polynomials in $b^{-1}$,
or the fact that polynomials in $e^{ix_1}, \ldots, e^{i x_n}$ (\ie trigonometric polynomials) are the archetype of periodic functions over $\R^n$.
Smooth functions can be approximated locally by Taylor's polynomial expansions or globally thanks to
the Weierstrass approximation theorem.

Polynomials are also ubiquitous due to their practical interest. 
Greeks and Babylonians used quadratic equations circa 2000~BC to compute the boundaries of their agricultural fields
in order to define fair taxes and trade rules. About 4000 years later, we still handle polynomials and solve polynomial equations,
not just on school benches, but also in real life to find the natural modes of oscillations of engineering structures or the
rate of spread of a virus. Polynomials are at the heart of numerical analysis and appear in particular
as approximations and interpolations of other functions in finite-element methods or through quadrature formulas~\cite{BERMAD},
or as a unifying frame for Fast-Fourier transforms~\cite{Nus82}.
Polynomials are found at the crossroads of science: computer-aided design relies heavily on geometric splines,
polynomials arise naturally in finance~\cite{ACK17}, in biology~\cite{MY20}, etc.
It is actually easy to find more than 50 different families of polynomials named after mathematicians
and that play a central role in various applications.

In most of the applications, having the fastest evaluation algorithm for a given level of accuracy is of the utmost importance.

\medskip
In this article, we propose a novel approach to evaluating complex polynomials in the case of fixed precision
floating-point arithmetic. Our algorithm is designed for better speed without compromising precision,
not directly for improving the precision of the results (though, for a given cost of computations,
 it may be used to achieve a higher precision than can be reached with the current, more costly, algorithms).
The algorithm is designed for repeated single-point evaluations. A typical application
is Newton's method to find one single root, where the sequence of evaluation points is not initially known.
However, the algorithm can also be used as an embarrassingly
parallel multi-point evaluator, which makes it highly versatile.

\subsection{Existing evaluation schemes}

Let us review briefly the state of the art regarding polynomial evaluation.\newpage
\begin{definition}[Complexity]\label{def:complexity}\quad
\begin{itemize}
\item Let us denote by $\eval_d$ the arithmetic complexity (number of arithmetic operations,
with the convention that 1 operation is a multiplication followed by an addition$^\ast$\footnote{$^\ast$ In most
hardware multiply-accumulate operations are implemented in one cycle, as per IEEE 754 \cite{IEEE}.})
of evaluating a polynomial of degree $d$. We will denote by  $\eval_d(k)$
the \textit{arithmetic complexity} of simultaneously computing $k$ values of a polynomial of degree~$d$.
\item
When all the computations are performed with a fixed precision of $p$ bits$^{\ast\ast}$\footnote{$^{\ast\ast}$ Note that, in this case, the number of exact digits in the result may be significantly smaller than~$p$. Performing computations to ensure $p$ exact digits could require  intermediary computations with arbitrarily high precision if the evaluation point is near a zero of the polynomial or a zero of some arbitrary subexpression that will cancel itself out (see Figure~\ref{fig:cancel} and Remark~\ref{rmk:cancel} below).}, we denote by $\eval_d(k,p)$
the corresponding \textit{bit complexity}$^{\ast\ast\ast}$\footnote{$^{\ast\ast\ast}$ Equivalent to the computation time, up to compiler and hardware optimizations or limitations.}
(number of bit operations).
\end{itemize}
\end{definition}
\noindent
One has $\eval_d(k,p) = M(p) \eval_d(k)$ where $M(p)$ is the bit complexity of the multiply-accumulate
of two floating-point numbers with precision $p$.
Typically, one has:
\begin{equation}\label{Mp}
M(p) = \begin{cases}
O(p^{1.585}) & \text{Karatsuba},\\
O(p^{1.465}) & \text{Toom-Cook},\\
O(p \log p \log\log p) & \text{Sch\"onhage-Strassen}.
\end{cases}
\end{equation}
For example, $M(128) \simeq 10^3$ bit instructions for a Toom-Cook multiplication.
The choice between these methods is usually driven by the competition between the value of $p$ and
the size of the hidden prefactor. The acceptable level of technicality in the code can also be taken
into consideration.

\subsubsection{Single point evaluation}

H\"orner's method ensures that, in general, $\eval_d = d$ and $\eval_d(1,p)=M(p) d$ when the polynomial is defined by its coefficients.
Since Ostrowski \cite{OST54} and Pan \cite{PAN66}, it has been well known that for evaluating a complex polynomial $P \in \C[X]$ of degree $d$ at a given point $z \in \C$, $d$ multiplications and $d$ additions are both necessary and sufficient. That is, H\"orner's scheme is optimal for
one point-wise evaluation of a general polynomial.

Some classical evaluation schemes offer a similar order of complexity with more balanced 
and better parallelizable intermediary computations to improve numerical stability and take advantage of modern hardware,
like Estrin's divide and conquer method~\cite{EST60},~\cite{Mor13}, which is \eg implemented in the Flint library~\cite{FLINT}.
As a side note, H\"orner's method is at the heart of a beautiful graphical construction for finding the real roots
of a polynomial, known as Lill's method~\cite{KAL08}.

\medskip
For iteratively defined polynomials, \ie a family $P_{n+1}(z)=Q(P_n(z))$, evaluation is obviously
more efficient: in this case,  one gets $\eval_d=O(r \log_r d)$ where $r=\deg Q$.
Similarly, any intermediary power $(z^k)_{0\leq k\leq d}$ of $z$ can be computed recursively in $O(\log_2 k)$.
In particular, sparse polynomials that contain only $\sigma$ non-zero coefficients can be evaluated
in this fashion in $O(\sigma\log_2 d)$ operations.

\subsubsection{Multi-point evaluation}

If the same polynomial has to be evaluated repeatedly,  one obviously seeks to obtain $\eval_d(k) \ll k \eval_d$
and there are better strategies to reduce the average computing time.
Knuth~\cite{KNU62} proposed a preprocesing based on finding all the zeros of the odd part of the polynomial (with Eve's variant \cite{EVE64} in the general case) that gains a factor of $2$ for the number of multiplications.
It then brings down the cost of subsequent evaluations to~$\eval_d=[\frac{1}{2}(d+4)]$.

\medskip
A common case where simultaneous evaluation brings a substantial benefit is the evaluation of trigonometric polynomials
along a regular mesh on the unit circle.
Computing the values $P\left(e^{2i\pi k/(d+1)}\right)$ for $0\leq k\leq d$ where $P(z)=\sum a_j z^j$  can be performed by Fast-Fourier
Transform (FFT) algorithms \cite{DL42}, \cite{CT65}, \cite{R00} in $\eval_d(d) = O(d\log d)$ operations by taking advantage of
(\ie factoring) the matrix structure
\[
\begin{pmatrix}
\hat{a}_0 \\\hat{a}_1\\ \vdots\\\hat{a}_{d}
\end{pmatrix}
\defequal
\begin{pmatrix}
1 &1 & \cdots & 1\\
1 & w & \cdots & w^{d}\\
\vdots & \vdots & & \vdots \\
1 & w^{d} & \cdots & w^{d^2}
\end{pmatrix}\begin{pmatrix}
a_0 \\a_1\\ \vdots\\a_{d}
\end{pmatrix}
=
\begin{pmatrix}
P(1)\\
P(w)\\
\vdots\\
P(w^{d})
\end{pmatrix}
\]
with $w=e^{2i\pi/(d+1)}$.
As an evaluation algorithm, this method brings down the average cost per computed value
to $O(\log d)$, the price to pay being that one single evaluation
cannot be (efficiently) performed alone.
The fact that the FFT is numerically well behaved \cite{PST02} and essentially
involutive brings evaluation and interpolation to an equal footing
and is the key that unlocks most of its applications.

Note that there are variants of the FFT method for computing approximations of the values of $P$
along a non-uniform mesh \cite{PST01}. The evaluation points are however constrained to remain
on a circle. Anticipating on our algorithm (see Section~\ref{sect:algo}), let us point out that there is indeed a benefit
in sorting the evaluation points according to the size of $|z|$ and that
we are  able to discard the geometric restriction of cocyclicity.

\medskip
Fast multipoint evaluation on a general set of points is  possible and
relies on a few standard tricks in polynomial arithmetic, which we first recall briefly.
Polynomial multiplication can be computed in $O(d^{1.585})$ with Karatsuba's algorithm. It is based on the identity
\[
(P z^k+ R) (Q z^k + M) = RM+ \left( (P+R)(Q+M)-RM-PQ \right) z^k+PQ z^{2k}\, ,
\]
which boils down to 3 multiplications of smaller polynomials (recursively optimized with $k\simeq d/2$) and coefficient shifts.
Over~$\C[X]$ or more generally if the field admits a discrete Fourier transform, one can use the FFT to conjugate the multiplication of polynomials to a pointwise multiplication of enough interpolation
points on the unit circle, with an overall cost of $O(d \log d)$. Next, fast division is based on the reversal of the order of the coefficients of the polynomial $P$,
\ie $\widetilde{P}(z) \defequal  z^{\operatorname{deg} P} \, P(1/z)$ and the identity
\[
P(z) = Q(z)M(z)+R(z) \quad\Longleftrightarrow\quad \widetilde{P}(z) =
\widetilde{Q}(z)\widetilde{M}(z) + z^{\operatorname{deg} P-\operatorname{deg} R} \,\widetilde{R}\left(z\right).
\]
The quotient $\widetilde{M}(z)$ can then be computed as $\widetilde{P}(z) \cdot \widetilde{Q}(z)^{-1} 
\operatorname{mod} z^{\operatorname{deg} P-\operatorname{deg} Q+1}$.
The series expansion of the inverse is computed recursively with Newton's method
in $\C[[X]]$:
\[
J_0=\widetilde{Q}(0)^{-1} \quad\text{and}\quad J_{n+1}(z) = J_n(z)\cdot \left( 2-J_n(z)\widetilde{Q}(z) \right),
\]
which ensures that $J_n(z) = \widetilde{Q}(z)^{-1} \operatorname{mod} z^{2^n}$.
Ultimately, this algorithm brings the overall cost of computing the division of a polynomial~$P$ to~$O(d \log d)$
where $d=\operatorname{deg} P$.

\medskip
The fast multipoint evaluation algorithm allows us to compute simultaneously $k$ values with a cost of $\eval_d(k) = O((d+k) \log d \log kd)$.
The central idea is a divide and conquer recursion.
One splits the evaluation points in two families $\mathcal{Z}_1$, $\mathcal{Z}_2$ of size $k/2$.
With the previous fast division scheme, one computes remainders modulo polynomials that vanish either on $\mathcal{Z}_1$ or on $\mathcal{Z}_2$.
The problem is thus reduced to the evaluation of the two remainders on sets of points that are half-sized:
\[
P(z) = Q(z) \cdot \prod_{\zeta\in \mathcal{Z}_j} (z-\zeta) + R_j(z) \quad\Longrightarrow\quad \forall \zeta\in\mathcal{Z}_j, \quad P(\zeta) = R_j(\zeta)
\qquad (j=1,2) \,.
\]
Conversely, one can reuse the structure of the intermediary computations to interpolate  with a similar total cost,
\ie compute the coefficients of~$P$ from the values $P(z_j)$ at $k=d+1$ distinct points $(z_j)_{j=0,\ldots,d}$.
The method can be refined~\cite{P95},~\cite{R99} to improve the poly-logarithmic factor in the arithmetic complexity.

\medskip
For finite precision arithmetic, \ie approximations of order~$2^{-p}$, 
advanced algorithms reach a theoretical bit complexity of $\eval_d(d,p) =O(d(d+p+\omega)\log F)$
to compute the values of~$P(z)$ at~$d$ complex points. In this formula,
one takes $\omega\geq0$ such that $|z|+\max|a_j|<2^\omega$ and~$\log F$ denotes logarithmic factors;
see~\cite{S82},~\cite{KS16} for more details.

Of course, the comparison with the standard multi-point evaluation is not clear-cut because bit complexity depends on the
size of the data while arithmetic complexity does not.
Theoretically, these variants thrive when $d+p+\omega \ll M(p)$ \ie[,] roughly speaking, when for example $d \ll p^{1.585}$.
Even though the multiplicative constants and logarithmic factor are either large or hard to track, one can
expect these algorithms to be competitive for moderately large degrees ($d\lesssim 10^4$) and a substantial
fixed precision ($128\leq p \leq 300$).

\subsubsection{Practical considerations beyond arithmetic complexity}

While the fast multipoint methods optimize the overall cost of multiple evaluations to~$\eval_d(d)=O(d (\log d)^2)$,
the number of operations that are involved in the computation
of one single value (or one single coefficient in the case of interpolation) exceed the number of naive operations that would be required to compute that value alone. Mechanically, one can thus expect a loss of precision.

A very detailed analysis of the numerical instabilities \cite{KZ08} points out the Wilkinson-type \cite{W84} expansion
of $\prod_{j=1}^{\ell} (z-j)$ as the main culprit,  which leads to ill-conditioned input to the subsequent polynomial divisions.
On the other hand \cite{KS16} exploit the fact that these same divisors are monic to ensure stable divisions when the required precision (in number of bits) dominates the degree.

\medskip
In many situations, in order to guarantee the numerical stability of the the multipoint method,
the precision of the numbers has to be much larger than $d$.
For example, S.~K\"ohler and M.~Ziegler \cite{KZ08}  conclude that,
for a specified level of accuracy, it is usually necessary to increase the precision
of intermediary computations to the point where the benefit over H\"orner's~$\eval_d(d)=O(d^2)$ naive
scheme is not significant.
The conclusion of A.~Kobel and M.~Sagraloff \cite{KS16} is more nuanced,
as they insist on the fact that the extra precision is only required for intermediary computations.
However, as one may need more than $O(d\log d)$ bits of memory per coefficient,
it quickly becomes impractical as $d$ increases (see \cite[Corollary 8]{KS16}).

\medskip
On modern computing machines, even on super-computers, the workload is often dominated by
data movements (disk, memory and cache access) and not by the computing power (usually measured
in floating-point operations per second, or \texttt{FLOPS}).
When implementing the multipoint evaluation algorithms,
the memory size limitations impose rather tight bounds on the degree (say $d\lesssim 10^6$).
This raises the question of how to compute efficiently the values of a giga-polynomial.
Our algorithm (see Section~\ref{sect:algo}) does not present this limitation.

\subsubsection{Alternatives}

Let us close this tour of the literature by mentioning briefly some less common methods of evaluation,
which have their own niche of applications.

\medskip
If working with extended precision is not an option, various methods based on a compensation
of H\"orner's Algorithm~\cite{SW2004}, \cite{GLL2005}, \cite{S2007} can improve the precision
of the standard evaluation scheme for a moderate increase in complexity.

\medskip
Choosing another basis instead of the canonical monomial basis of $\C[X]$ may provide
better numerical stability. 
Evaluation algorithms on Newton's (interpolation) basis have similar complexity to the ones exposed above \cite{BS05}, sometimes even with better constants. We refer to \cite{F12}, \cite{F08}
for an in depth review of the benefits of the Bernstein basis and its industrial applications.
The complexity of evaluating the Bernstein basis functions has recently been improved in \cite{CW20} and is now~$O(d)$, which is still huge compared to evaluating $z^d$ but makes it a viable option for \eg $d\lesssim 10^3$.
In this article, we will not investigate further  the question of generalizing our algorithm to other bases.

\medskip
For the sake of completeness, let us also mention that better performances as low as~$O(\sqrt{d})$ can be achieved
for non-scalar evaluations \cite{PS73}, \cite{F18}, \ie if one computes polynomials of matrices where the complexity
of scalar operations is simply discarded. However, these algorithms do not bring any improvement
upon H\"orner's method, when they are applied to the evaluation of scalar polynomials.

\subsection{New evaluation scheme}

In this article we propose a simple algorithm and its practical implementation as a~\texttt{C} library~\cite{FPELib}
that brings down the average cost for the repeated evaluation of all polynomials
and never exceeds H\"orner's complexity in general.

\begin{figure}[H]
\captionsetup{width=.95\linewidth}
\begin{center}
\includegraphics[width=\textwidth]{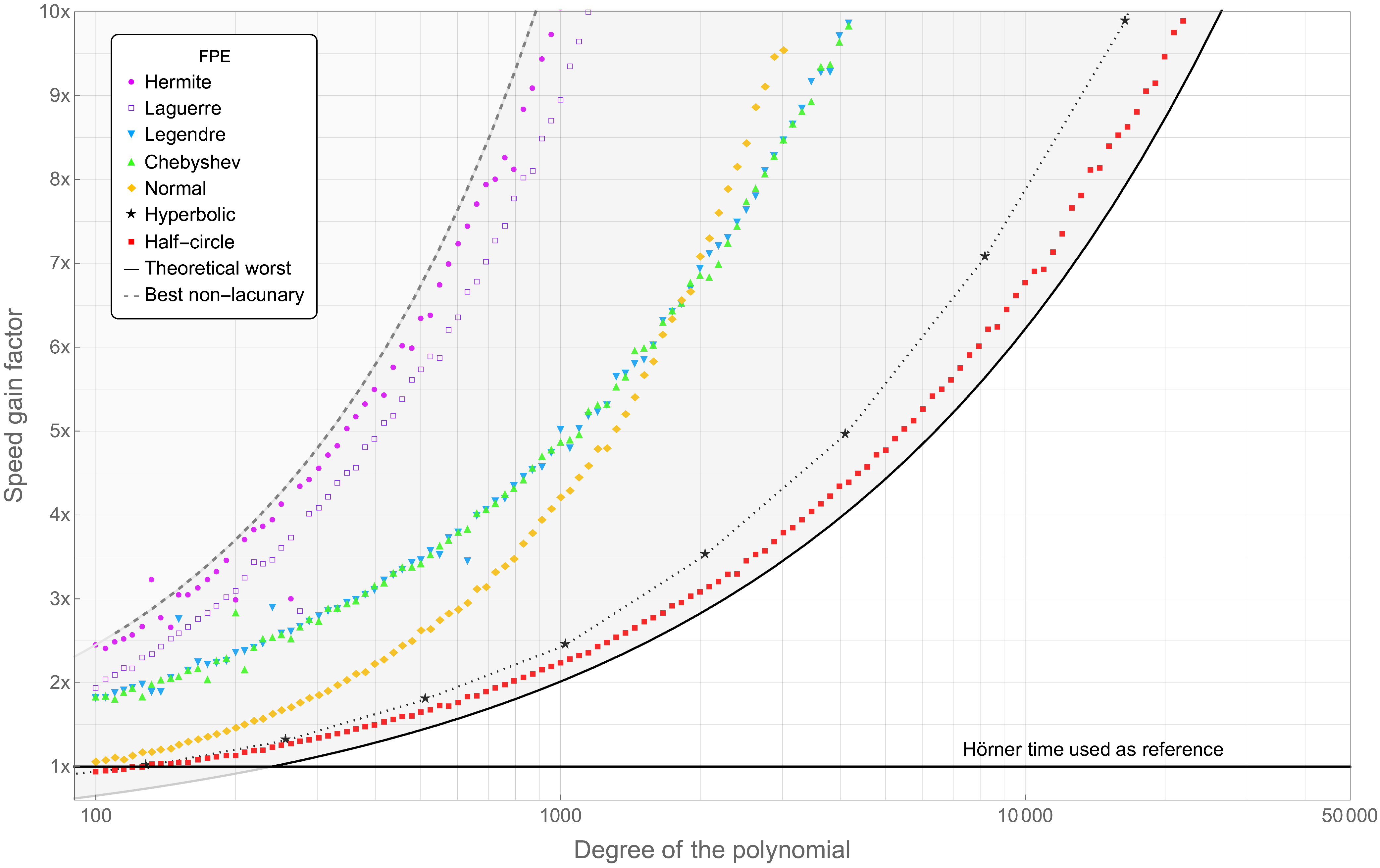}
\caption{\label{fig:HornerFPETheory}
The average speed gain of our FPE (\emph{Fast Polynomial Evaluator}) algorithm over H\"orner's method for computations with $p=53$ bits (using
\texttt{MPFR} numbers; \cite{MPFR}).
The solid curve corresponds to bound~\eqref{eq:asymptGain} that follows from Theorem~\ref{thm:algo}.
The dashed one corresponds to the example at the end of Section~\ref{par:algo} where one evaluation point has
the same complexity as H\"orner, but the average complexity is favorable. Note that lacunary polynomials may lead to even
higher gain factors. 
The data points are actual benchmarks (see Section~\ref{sec:bench}).
The case of the red squares (half-circle) is studied in detail in Section~\ref{par:optimal}.
}
\end{center}
\end{figure}

\bigskip
More precisely, our algorithm, called \A{} or \emph{Fast Polynomial Evaluator} (see Section~\ref{par:algo})
pre-conditions $P\in \C[X]$ in time $O(d\log d)$ with a low multiplicative constant that does not depend on
the precision. Subsequent evaluations of~$P(z)$ with a fixed precision of~$p$ bits are performed in \textit{average} 
arithmetic complexity
\begin{equation}\label{eqComplexity}
\avg_{\CC}\eval_d=O\big(\sqrt{d(p+\log d)}\big) \,.
\end{equation}
The constant is small and explicit and is given by~\eqref{eq:compAlgo} below.
The memory requirement is~$O(dp)$.
The average of the complexity used in~\eqref{eqComplexity} is taken
with respect to points $z \in \C$  weighted by the spherical area of $\CC$.
A similar estimate holds for real polynomials and a uniformly distributed
evaluation point along the circle $\R\cup\{\infty\}$.

\medskip
As illustrated in Sections~\ref{par:optimal} and~\ref{sec:bench},
for many particular classes of polynomials, in particular for sparse polynomials or those with random coefficients
confined in a bounded region of $\C$, our algorithm performs much better than the upper bound~\eqref{eqComplexity}.

\medskip
The \A{} algorithm has many interesting features. One can guarantee that the result is \emph{as precise} as H\"orner's method.
Pointwise, the complexity of \A{} does not exceed that of H\"orner (\ie $V_d \leq Cd$).
In case of equality for some $z_0\in\C$ (see Remark~\ref{rmk:betterIdea}),
one has
\[
\avg_{\CC}\eval_d = O\left( (p+\log d) \left(1+\left|\log\frac{d}{p}\right|\right)\right) \,,
\]
which is even more advantageous than~\eqref{eqComplexity}.
This radical difference between the pointwise and the average complexity is a strong incentive in favor of studying averages.
The \A{} algorithm is embarrassingly \emph{parallel} and can be implemented on any set of evaluation points, \emph{without constraints} of
size or of geometric structure. New evaluation points can be added on the fly.
These properties make the \A{} algorithm particularly well suited for a root finding scheme with Newton's method.

\medskip
For low-precision computations the theoretical speed factor of \A{} over H\"orner's method is illustrated in Figure~\ref{fig:HornerFPETheory}.
Benchmarks of our implementation will be presented in Section~\ref{sect:code} and, in particular, the analysis of the influence
of the preprocessing phase over the global cost (it remains minimal).

\medskip
The cornerstone idea at the foundation of the~\A{} algorithm is \emph{lazy polynomial evaluation} (see Section~\ref{par:lazy}): 
adding two finite precision numbers is only necessary if their orders of magnitude are close enough that their bits will interact.
Monomials tend to have extremely diverse orders of magnitude, which means that the value of a polynomial at a given point
is dictated by only a small subset of its monomials. The second ingredient is a \emph{geometric selection principle},
\ie the ability to identify this parsimonious representation
with geometric tools, which, in practice, boil down to the computation of the concave cover of a dataset (see Figure~\ref{fig:concave}).

\medskip
This article features a detailed analysis of the complexity (see Theorem~\ref{thm:algo}) and a precise error analysis (see Theorem~\ref{thm:equivAlgo}
and Figure~\ref{fig:accuracy}) of the~\A{} algorithm. Both are put to the test in systematic benchmarks presented in Section~\ref{sec:bench}.
The preview offered in Figure~\ref{fig:HornerFPETheory} illustrates the extent to which we have explored the theoretical 
and practical envelope of this new algorithm.

\subsection{Structure of the article}

The structure of the text is the following.

\bigskip
In Section~\ref{par:arithmetic} we detail the fundamentals of finite precision arithmetic for a general audience
and discuss the specificities of complex numbers.
Subsection~\ref{par:equivModP} is dedicated to the various notions of \textsl{closeness} in finite precision
(equivalence, adjacency and similarity modulo phase-shift), which play a central role in the proof of the correctness of our algorithm.

In Section~\ref{sect:analyse} we briefly introduce some geometric tools that will be needed
to state the \A{} algorithm and prove its complexity. 
This section contains only definitions and statements; the proofs of the corresponding theorems can be found in Appendix~\ref{par:geom_proof}.

In Section~\ref{sect:algo} we describe and analyze the~\A{} algorithm and state our two main results.
The correctness of the algorithm and the associated error analysis is Theorem~\ref{thm:equivAlgo}.
The result regarding complexity is Theorem~\ref{thm:algo}.
The proofs are done in  the next two sections:
Theorem~\ref{thm:algo} is proved in Section~\ref{sect:comp} and
Theorem~\ref{thm:equivAlgo} in Section~\ref{sect:err}.

Section~\ref{sect:applic} explores a few examples of possible applications of the \A{} algorithm and should be of general interest.
Section~\ref{sect:code} is dedicated to presenting our implementation in the \texttt{C} language, which we are publishing~\cite{FPELib} as an open source project.
Our implementation uses both machine floating-point numbers and MPFR arbitrary precision numbers (see \cite{MPFR}). 
Appendix~\ref{FPETasks} contains a listing and description of the tasks that can be performed with it.
Extensive numerical benchmarks are presented in subsection~\ref{sec:bench} and confirm the theoretical predictions regarding
the complexity and error analysis of the \A{} algorithm.

\bigskip
In order to keep this article accessible to the widest possible audience,
we provide comprehensive definitions of all notions and fully detailed proofs.
We also tried to keep the sections as independent as possible.
Overall, the key ideas are of a geometric nature. There are strong similarities between the
geometrical reasonings of~Section~\ref{sect:analyse} and Appendix~\ref{par:geom_proof}
and the presentation of the algorithm in~Section~\ref{sect:algo}
(for example, compare Figures~\ref{fig:strip} and~\ref{fig:concave}, or~\ref{fig:shear} and~\ref{fig:risingSun}).
It is our belief that one may enlighten the other.

However, a reader interested in understanding the algorithm, but not the proof of its correctness and complexity,
may safely read only the beginning of Section~\ref{par:arithmetic} and skip Subsection~\ref{par:equivModP},
then read Section~\ref{sect:analyse} before proceeding to Sections~\ref{sect:algo}, \ref{sect:applic} and~\ref{sect:code}.
The more theoretical Sections~\ref{sect:comp}, \ref{sect:err} and Appendix~\ref{par:geom_proof} may be skipped.

\medskip
Finally, as a convenience for all readers, Appendix~\ref{app:index} recapitulates the notations used throughout the article.

\section{A bit of finite precision arithmetic}
\label{par:arithmetic}

In this section we introduce notations and tools that will be useful in the statement of the algorithm
and in all subsequent analysis. We refer to,~\eg[,]~\cite{HFPA2018}, \cite{G1991}
or~\cite{IEEE} for further details on finite precision arithmetic.

\subsection{General considerations}

Let us start with a word of caution:
the exact evaluation of complex polynomials at arbitrary points is not possible in practice.
Attempting to evaluate an explicit polynomial (that is, given by its coefficients)
near one of its roots will produce a large cancelation of the digits. 
For example, $P(z)=z-z_0$ evaluated at $z$ such that $|z-z_0|<2^{-n}|z_0|$ produces a result of
order at most $2^{-n}|z_0|$, effectively losing $n$ leading bits from the precision that was used to
express~$z_0$ and~$z$. More generally, when $P(z)$ has many terms,
cancelations can occur not only at the roots but among all  polynomial subexpressions of~$P$ (see Figure~\ref{fig:cancel}).
To guarantee that the result has any number of significant exact digits,
unbounded precision would have to be used for intermediary computations.

As this is not practical, we will focus only on computations done with some \textit{fixed} precision $p$.
Our algorithm does not attempt to produce results that are more accurate than those of H\"orner's method: we
want to produce results of similar accuracy, only faster.
The error analysis of H\"orner's scheme is classical and we refer the reader to \cite{O1979}, \cite{M1983}.
For a more general analysis with recursive basis functions, see~\cite{BJS2013}.

\medskip
When using fixed precision numbers, additions and subtractions are the main
sources of errors because they can produce a cancelation of the most significant digits.
After such an occurence, the relative uncertainty is multiplied by $2^s$,
where $s$ is the number of canceled bits.
Fortunately, having $s$ bits canceled is conditioned by the fact that the numbers
have the same scale and then cancelation only occurs with a probability $2^{-s}$.

\medskip
Our algorithm exploits the limitations of finite precision arithmetic
to discard unnecessary computations and thus obtain significant gains on the computation time.


\begin{figure}[H]
\begin{center}
\setlength{\unitlength}{1cm}
\includegraphics[width=.45\textwidth]{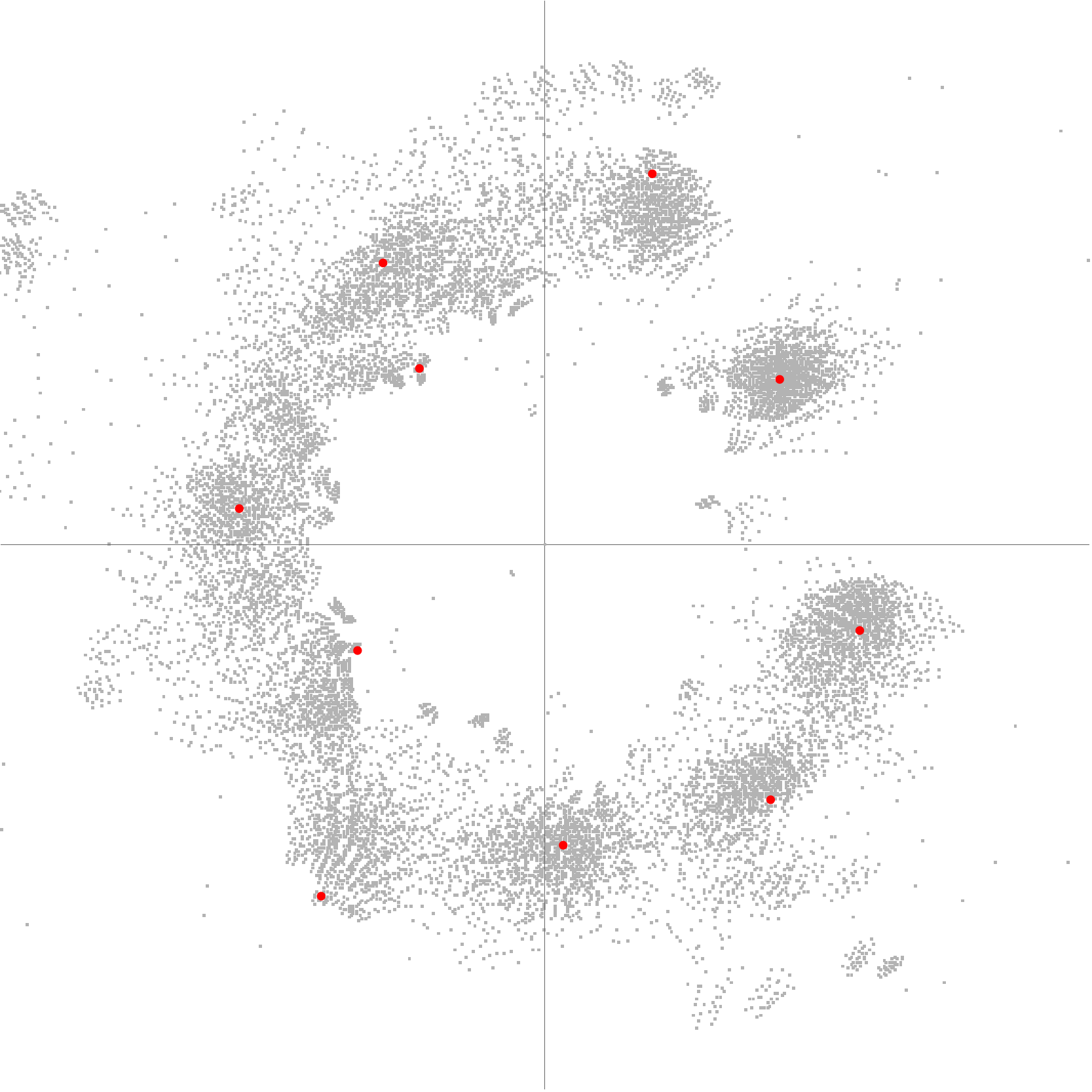}\qquad
\includegraphics[width=.45\textwidth]{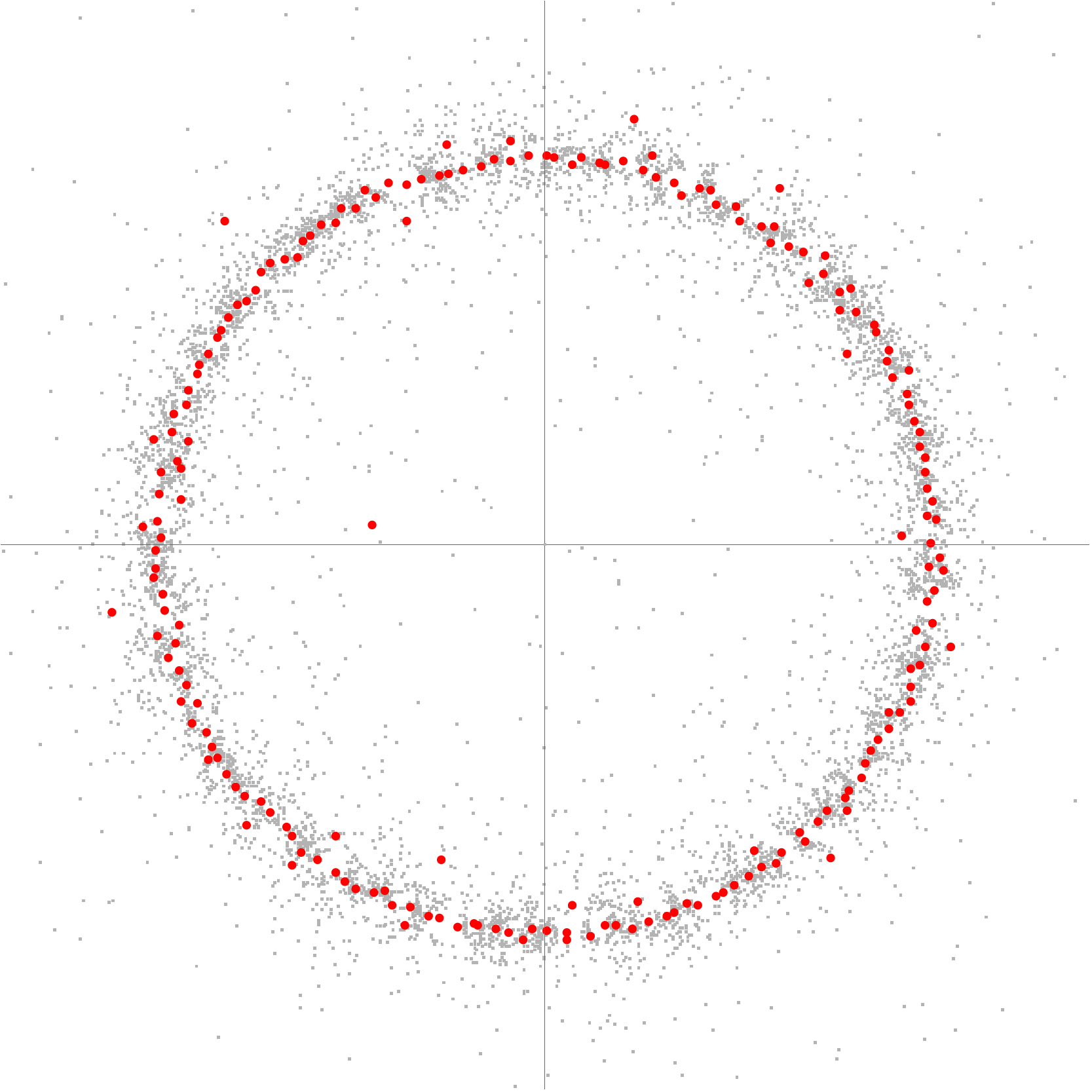}
\captionsetup{width=.95\linewidth}
\caption{\label{fig:cancel}%
The roots of a polynomial (red) and some polynomial subexpressions (gray).
The real and imaginary parts of the coefficients are independent normal distributions.
The graphics are zoomed in on the most significant part of $\C$.
\textbf{Left}: The roots of all $2^{11}-1$ non-trivial polynomial subexpressions of a polynomial of degree~10.
Note that cancelations may occur in a non-uniform way.
\textbf{Right}: The roots of polynomial subexpressions formed by consecutive $\frac{1}{4}\deg P$ monomials,
when $\deg P =200$. In both cases, the coalescence of the roots and cancelation points around the unit circle is expected for high degrees
because of Hammersley's theorem~\cite{Ham1956}, \cite{SZ2003}.}
\end{center}
\end{figure}

\subsection{Lazy addition in finite precision arithmetic}

A floating-point number $\xi$ represented with a precision$^\ast$ $p$ in base 2 is written
\footnote{$^\ast$ In the MPFR library, \eqref{equ:prec} is said to have precision $p+1$.}
\begin{equation} 
\label{equ:prec}
\xi = \pm 2^{n} \times 0.1\xi_1\xi_2 \ldots \xi_p  
\end{equation}
where the \textit{bits} $\xi_i \in \{0,1\}$ for $i \in \ii{1}{p} \defequal  [1, p] \cap \Z$.
The number $n\in \Z \cup \{-\infty\}$ is called the \textit{exponent} of~$\xi$.
By convention, $n=-\infty$ when $\xi=0$.
The smallest representable increment of $|\xi|$ is
\begin{equation}
\ulp(\xi) \defequal  2^{n} \times 0.00 \ldots 01 = 2^{n-p-1} \,.
\end{equation}
The name stands for \textit{unit in the last place}.
To ensure a unique representation of all real numbers we always assume a rounding to the
nearest representable number and choose a rounding away from zero at the tie.

\medskip
A key observation for additions in finite precision $p$ is that
\begin{equation}\label{equ:sum}
\xi+\eta = 
\begin{cases}
\xi & \text{if } n > m + p + 2,\\
\eta & \text{if } m > n + p + 2,
\end{cases}
\end{equation}
where $n$ and $m$ are the respective exponents of $\xi$ and $\eta$.
This means that by simply reading the values of $n$ and $m$ and comparing them to $p$,
we can avoid costly operations, especially if computing one of the terms $\xi$ or $\eta$ 
requires additional steps as is the case when they are monomials $a_k z^k$.

\subsection{Scale of a complex number}\label{par:scale}

Let us define the \textit{scale} of a number $z \in \C^*$ by
\begin{equation}
\label{equ:dsz}
s(z) \defequal  1+ \lfloor \log_2 |z| \rfloor \in\Z \,,
\end{equation}
where $\lfloor \alpha \rfloor$ is the floor of $\alpha$.
By convention,~$s(0) \defequal  -\infty$.
The scale is a logarithmic representation of the order of magnitude of $z$.
For example, with the notations of~\eqref{equ:prec}, one has $s(\xi) = n$, \ie 
for floating-point numbers, the scale coincides with the exponent;
moreover
\begin{equation}\label{def:ulp}
\ulp(\xi)=2^{s(\xi)-p-1} \,.
\end{equation}
In general, $s(z)=\sigma$ if and only if $\sigma\in\Z$ and
\begin{equation}\label{eq:sigma}
2^{\sigma-1}\leq |z|<2^\sigma \,.
\end{equation} 
In particular, for $z=x+iy\in\C^\ast$, one has $\frac{1}{2}\leq |z|2^{-\max(s(x),s(y))} < \sqrt{2}<2$,
thus
\begin{equation}\label{equ:sz}
s(z) \in \max\{s(\Re z),s(\Im z)\} +\{0,1\} \,.
\end{equation}
For example $s(3+2i)=2=s(3)$ and $s(3+3i)=3=s(3)+1$.

\bigskip
As $\lfloor\alpha+\beta\rfloor\in\lfloor\alpha\rfloor+\lfloor\beta\rfloor+\{0,1\}$ and $1+\lfloor\alpha\rfloor-\alpha\in(0,1]$,
we claim the following bounds.
\begin{lemma}
For any  $z, z' \in \C$ and $n\in\Z$, one has~:
\begin{eqnarray}
s(2^n z) & = & n+s(z) \label{eq:exact2pow}\\
\label{equ:sp}
s(zz') - s(z) - s(z') & \in & \{-1, 0\}, \\
\label{equ:sn}
s(z^n) - n\log_2|z|  & \in & (0,1], \\
\label{equ:ss}
s(z \pm z') & \leq &\max\{s(z),s(z')\} + 1
\end{eqnarray}
and, for a family $z_1,\ldots,z_N\in\C$~:
\begin{equation}\label{equ:bs}
s\left(\sum_{j=1}^N z_j \right) \leq \max_{1\leq j\leq N} s(z_j) + s(N) \,.
\end{equation}
\end{lemma}
\noindent%
Due to the cancelation of significant bits, the scale of a sum may be much smaller (even~$-\infty$) than
the largest scale of the terms involved. 

\proof
The identity~\eqref{eq:exact2pow} is immediate.
For $s(z)=\sigma$, $s(z')=\sigma'$, one has $2^{\sigma+\sigma'-2}\leq |zz'|<2^{\sigma+\sigma'}$
and~\eqref{equ:sp} follows from~\eqref{eq:sigma}.
The estimate~\eqref{equ:sn} follows from $|z^n| = 2^{n\log_2|z|}$ \ie
$s(z^n) = 1+\lfloor n\log_2|z|\rfloor$.
For sums, we write $|z\pm z'|\leq 2\max\{|z|,|z'|\}<2^{1+\max\{s(z),s(z')\}}$ hence~\eqref{equ:ss}.
For $N$ terms, \eqref{equ:bs} follows from:
\[
\left|\sum z_j\right|\leq N \max |z_j| < 2^{\max s(z_j) + \log_2 N} < 2^{\max s(z_j) + \lfloor\log_2 N\rfloor+1} \,.
\]
For the last inequality, note that we could use the ceiling function $\lceil \log_2 N\rceil$ instead of~$s(N)$ for a slightly tighter estimate
when $N$ is a power of 2.
\qed

\subsection{Equivalence and adjacency modulo finite precision}\label{par:equivModP}

The topology induced by finite precision arithmetic is surprisingly subtle.
In this subsection, we introduce three distinct binary relations that express the \emph{proximity}
between a floating-point number and a real or complex number. The complex case is even more subtle
and is dealt with last.

\begin{definition}
For $x,y \in \R$ we use the notation $x =_p y$ to say that $x$ and $y$ have the same representation
as floating-point numbers with precision $p \in \N^*$. 
\end{definition}

\begin{figure}[H]
\begin{center}
\setlength{\unitlength}{1cm}
\includegraphics[width=12cm]{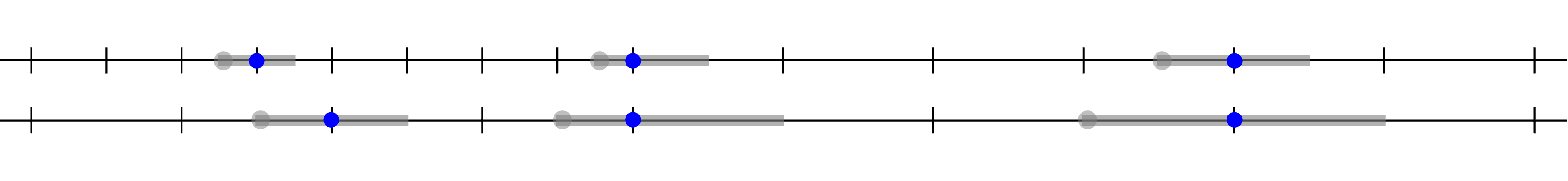}
\begin{picture}(0, 0)(12.1,0)
	\put(4.1,1.1){\footnotesize $a$}
	\put(4.76,1.1){\footnotesize $b$}
	\put(5.9,1.1){\footnotesize $c$}
	\put(12.25,0.84){\footnotesize $=_p$}
	\put(12.25,0.39){\footnotesize $=_{p-1}$}
\end{picture}
\captionsetup{width=.9\linewidth}
\caption{\label{fig:ulp}%
On the first line, three examples of equivalence classes for~$=_p$ are grayed out.
On the second line, examples of classes for~$=_{p-1}$ illustrate how the grid gets thinned out when
one bit of precision is dropped. Note in particular how the equivalence class of odd numbers (left) gets split.
}\end{center}
\end{figure}
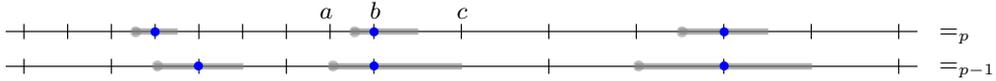

\noindent
It is an equivalence relation; the equivalence class of
a floating-point number $\xi>0$ is
\begin{equation}\label{eq:class_modp}
\{y\in\R \,;\, y=_p \xi\} = \left[ \xi-\frac{1}{2}\ulp(\xi-\ulp(\xi)); \xi+\frac{1}{2}\ulp(\xi) \right).
\end{equation}
In general, the equivalence class~\eqref{eq:class_modp} of a floating-point $\xi>0$ is the interval
\[\textstyle
\left[\xi-\frac{1}{2}\ulp(\xi), \xi+\frac{1}{2}\ulp(\xi) \right).
\]
An exception occurs at scale turnover where the class is asymmetric. For example,
the consecutive numbers $a=2^n(1-2^{-p-1})=2^n\times 0.111\ldots11$,
$b=2^n=2^{n+1}\times 0.100\ldots00$ and $c=2^n(1+2^{-p})=2^{n+1}\times 0.100\ldots01$ satisfy
$\ulp(b)=c-b$ and $\ulp(a)=b-a=\frac{1}{2}\ulp(b)$
thus the equivalence class of $b$ is $[b-\frac{1}{4}\ulp(b),b+\frac{1}{2}\ulp(b))$.
Note that $b-\ulp(b)=2^n(1-2^{-p})$ and $\ulp(b-\ulp(b))=\ulp(a)$.
See Figure~\ref{fig:ulp}.
When $\xi<0$, the usual convention \textit{rounding away from zero at the tie}
implies that the interval~\eqref{eq:class_modp}  is flipped over and the $\pm$ signs must be reversed.

\bigskip
We say that equivalence classes for $=_p$ are \textit{adjacent} if their respective
closures in~$\R$ have a non-empty intersection. The corresponding floating-point numbers
are called adjacent too.
\begin{definition}
For $x,y \in \R$, we denote by $x \simeq_p y$ if the floating-point representations of~$x$ and~$y$ with $p$~bits
are identical or adjacent.
\end{definition}
\noindent
Though not transitive because of the obvious overlap between the sets of real numbers that are adjacent
to consecutive floating-point numbers, this relation is symmetric and simplifies the handling of scale turnover.
For example, one can always find real points that are arbitrarily close to one another
but whose floating-point representations are distinct; these points are however adjacent. See Figure~\ref{fig:ulp2}.

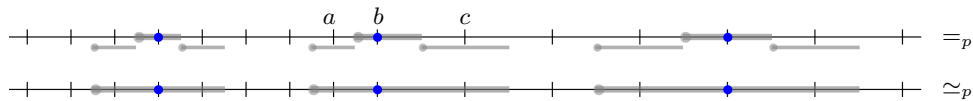
\begin{figure}[H]
\begin{center}
\setlength{\unitlength}{1cm}
\includegraphics[width=12cm]{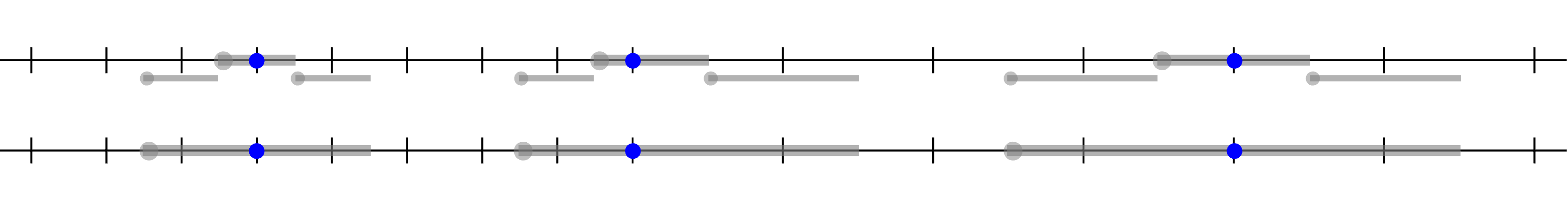}
\begin{picture}(0, 0)(12.1,0)
	\put(4.1,1.31){\footnotesize $a$}
	\put(4.76,1.31){\footnotesize $b$}
	\put(5.9,1.31){\footnotesize $c$}
	\put(12.25,1.065){\footnotesize $=_p$}
	\put(12.25,0.39){\footnotesize $\simeq_p$}
\end{picture}
\captionsetup{width=.9\linewidth}
\caption{\label{fig:ulp2}%
On the first line, examples of triplets of adjacent $=_p $ classes.
On the second line, examples of numbers $x$ such that $x\simeq_p \xi$ with respect to the marked (blue) floating-point $\xi$.
The relation $\simeq_p $ is not transitive: $a\simeq_p b\simeq_p c$ but $a\not\simeq_p c$.
}\end{center}
\end{figure}

\medskip
When $x$ is a real number and $\xi$  is its floating $p$-bit representation, it
satisfies $|x-\xi|\leq \frac{1}{2}\ulp(\xi)$ and $s(x)\leq s(\xi) \leq s(x)+1$.
Moreover, $s(x)\neq s(\xi)$ occurs only at scale turn-over, when $2^n(1-2^{-p-2})\leq x<\xi=2^n$ for some~$n\in\Z$.
Conversely, if $\xi$ is a floating-point number such that $|x-\xi|\leq \frac{1}{4}\ulp(\xi)$ then $\xi$ is the $p$-bit representation of $x$.

\medskip
Along the real line, the characterization of adjacency in terms of scale is the following.
\begin{lemma}\label{lemma:eqp}
Let $x$, $y\in\R$ and $p\in\N^\ast$. One has
\begin{equation}\label{equ:eqp3} 
s(x - y) \leq \max\{s(x),s(y)\} - p - 2 \quad\Longrightarrow\quad 
x \simeq_p y
\end{equation}
and, conversely,
\begin{equation}\label{equ:eqp3_converse}
x \simeq_p y
\quad\Longrightarrow\quad
s(x - y) \leq \max\{s(x),s(y)\} - p \,.
\end{equation}
If $x$ and $y$ are $p$-bit floating-point numbers, then 
\begin{equation}\label{equ:eqp_float}
s(x - y) \leq \max\{s(x),s(y)\} - p - 1 \quad\Longrightarrow\quad 
x \simeq_p y \,.
\end{equation}
\end{lemma}
\proof
If $x$ and $y$ are not adjacent real numbers, then $x$ and $y$ are separated by a whole $p$-bit equivalence class of some
floating-point number $\xi$ (see Figure~\ref{fig:ulp2}), \ie
\begin{equation}\label{eq:not_eqp}
|x-y|> \begin{cases}
\frac{3}{4}\ulp(\xi) & \text{if }\xi=2^n, \quad n\in\Z\\
\ulp(\xi) &\text{otherwise}.
\end{cases}
\end{equation}
In general, choosing the largest $\xi$ possible ensures $s(\xi)=\max\{s(x),s(y)\}$ and
\[
2^{s(x-y)} > |x-y|>\frac{1}{2}\ulp(\xi) = 2^{s(\xi)-p-2} = 2^{\max\{s(x),s(y)\}-p-2} \,.
\]
The only exception is $s(\xi)=\max\{s(x),s(y)\}-1$ when $\xi=2^{n}(1-2^{-p-1})$;
in that case
\[
2^{s(x-y)} > |x-y|>\ulp(\xi) = 2^{s(\xi)-p-1} = 2^{\max\{s(x),s(y)\}-p-2} \,.
\]
In both cases, we have~\eqref{equ:eqp3}.
If  $x$ and $y$ are non-adjacent $p$-bit floating-point numbers, then~\eqref{eq:not_eqp} can be improved by a factor 2,
hence~\eqref{equ:eqp_float}.

\medskip
Conversely, if $x \simeq_p y$ and $\tilde{x}$, $\tilde{y}$ are respectively the $p$-bit representations of $x$ and~$y$,
then $|x-y|<\ulp(\tilde{x})+\ulp(\tilde{y})\leq 2\max\{\ulp(\tilde{x}),\ulp(\tilde{y})\}=2^{\max\{s(\tilde{x}),s(\tilde{y})\}-p}$
by~\eqref{def:ulp}, thus, according to~\eqref{eq:sigma} :
\[
s(x-y) \leq \max\{s(\tilde{x}),s(\tilde{y})\}-p \,.
\]
If $\max\{s(\tilde{x}),s(\tilde{y})\}= \max\{s(x),s(y)\}$, which is the case in general, then~\eqref{equ:eqp3_converse} holds.
At scale turnover, one may also have $\max\{s(\tilde{x}),s(\tilde{y})\} = \max\{s(x),s(y)\}+1$; then $x,y<\max\{\tilde{x},\tilde{y}\}=2^n$ for
some $n\in\Z$. In this case, the previous estimate improves to $|x-y|<\max\{\ulp(\tilde{x}),\ulp(\tilde{y})\}$,
which ensures~\eqref{equ:eqp3_converse}.
\qed

\bigskip
For complex numbers, the situation is more complicated because of phase shifts. 
In the real case, the only phase shift possible is a sign change, which does not affect precision;
in particular, there are only $2^{p+1}$ finite precision numbers for a given scale.
The direct extension of the adjacency relation $z\simeq_p z$ as the conjunction of $\Re z\simeq_p \Re z'$ and $\Im z\simeq_p \Im z'$
can lead to an extreme scale imbalance between the real and imaginary parts,
which is not compatible with phase shifts (\ie complex rotations).
Complex numbers whose argument is close to~$k \pi/2$ (with $k\in\Z$, \ie near the axes) have one component artificially over-resolved
compared to the other one; consequently, there are infinitely many$^\ast$\footnote{$^\ast$
In practice, the number of finite precision complex numbers
at a given scale is limited by the extreme negative exponent value authorized in the implementation.}
finite precision numbers with a given scale  (see Figure~\ref{fig:fpComplex}).

A similar instance of the same issue occurs if one choses two complex numbers close to the diagonal whose real and imaginary parts
are $p$-bit adjacent, \eg:
\[
(1+i)(1+ i 2^{-p-r})= 1-2^{-p-r}+i (1+2^{-p-r}) \simeq_p 1+i
\]
for any $r\geq1$ (recall that $\ulp(1)=2^{-p}$ so $1\pm 2^{-p-r} \simeq_p 1$).
Once we rotate them to bring them along the real axis (or, equivalently, if we multiply by $(1+i)^{-1}=\frac{1}{2}(1-i)$, which is an exact 0-bit number),
their imaginary parts will not be adjacent anymore and will
instead be separated by infinitely many scales: $1+i2^{-p-r} \not\simeq_p 1$.

\bigskip
To address this issue, we introduce a looser version of~$\simeq_p$, which is based on the scale
and inspired by the property~\eqref{equ:eqp3} above.
 \begin{definition}
For a pair of complex numbers $z,z'$, we note $z \approx_p z'$ if and only if
\begin{equation}\label{equ:eqpComplex}
s(z-z')\leq \max\{s(z),s(z')\} - p -2 \,.
\end{equation}
We say that $z$ and $z'$ have similar $p$-bit floating-point representations,
modulo approximate phase-shift invariance.
\end{definition}
Let us point out that, because $s$ is an increasing function, the following criterion holds~:
\begin{equation}\label{eq:eqp_crit}
|z-z'| \leq 2^{-p-2}|z| \quad\Longrightarrow\quad z\approx_p z' \,.
\end{equation}
Conversely, according to~\eqref{eq:sigma}, $ z\approx_p z'$ implies  $|z-z'| < 2^{-p-1}\max\{|z|,|z'|\}$.

\begin{figure}[H]
\begin{center}
\setlength{\unitlength}{1cm}
\includegraphics[width=.45\textwidth]{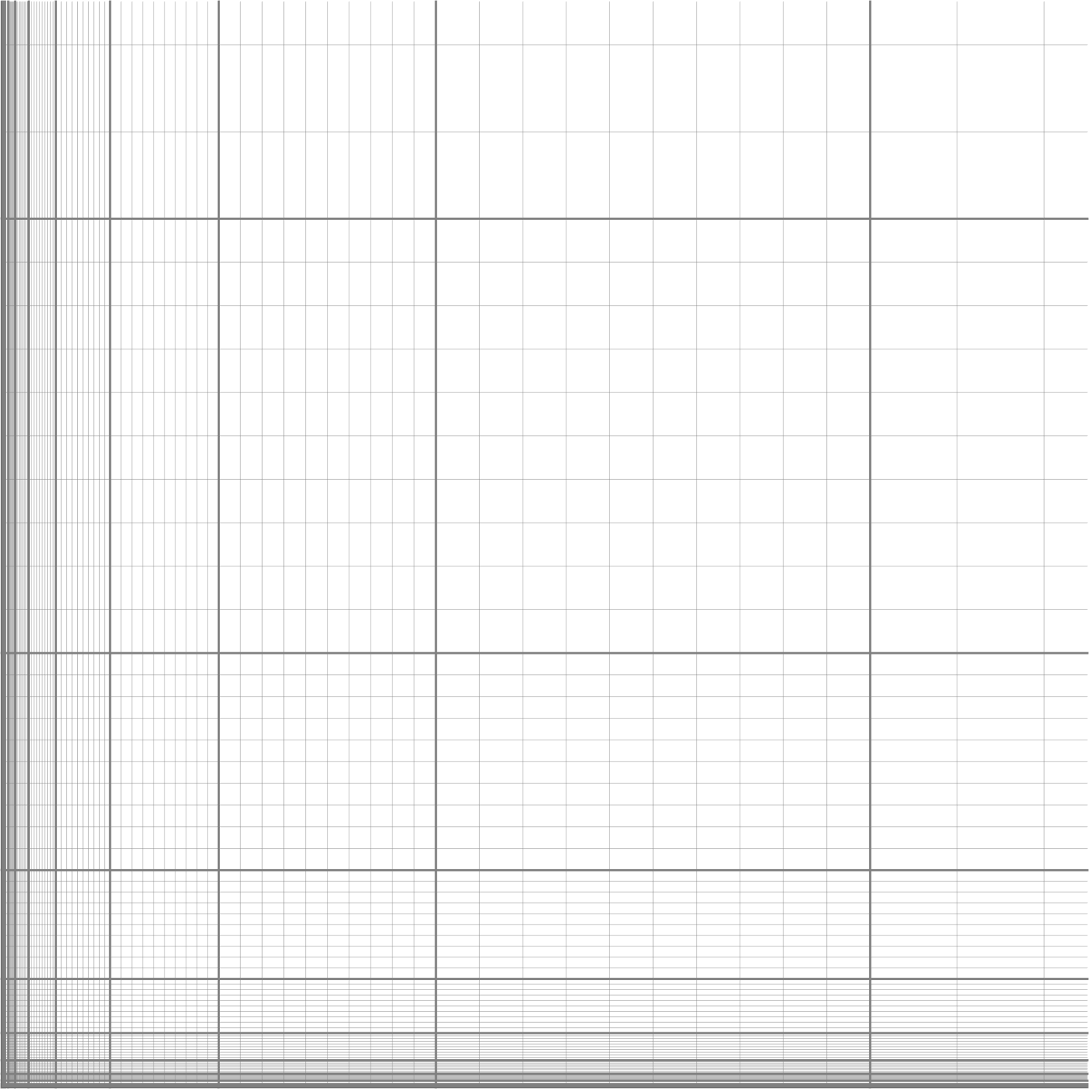}\qquad\quad
\includegraphics[width=.45\textwidth]{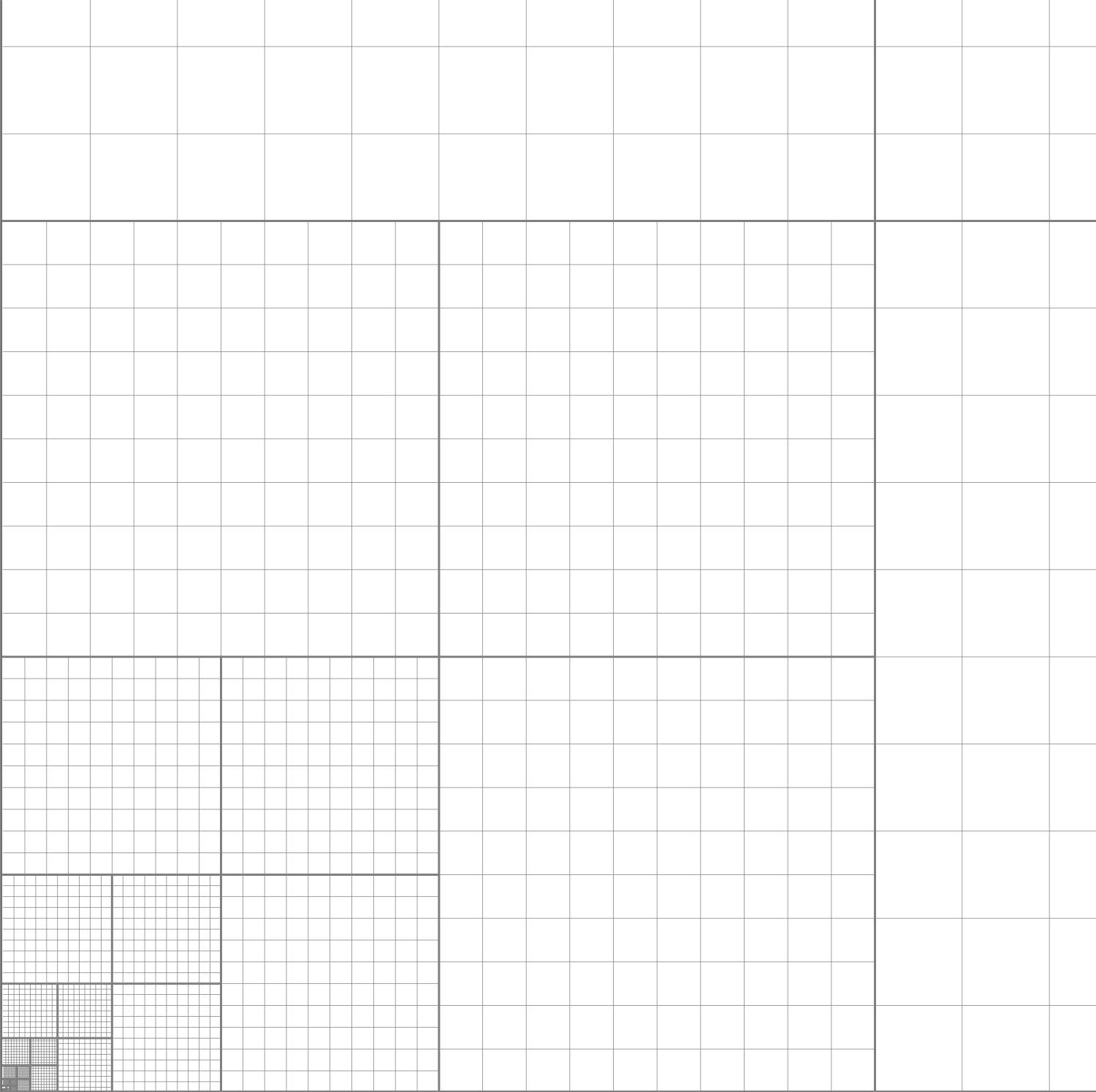}
\captionsetup{width=.95\linewidth}
\caption{\label{fig:fpComplex}%
The left grid represents the coordinates of floating-point complex numbers in the first quadrant. 
On this grid, direct neighbors correspond to a simultaneous adjacency~$\simeq_p$ of both the real and the imaginary parts.
The right grid illustrates the courser mesh associated with the~$\approx_p$ relation (similarity moduluo phase-shift), where the maximum of the scales of
the real and imaginary parts dictates the overall precision.}
\end{center}
\end{figure}

Thanks to Lemma~\ref{lemma:eqp}, the relation $z\approx_p z'$ implies adjacency when~$z,z'\in\R$.
If~$\zeta$ is a finite precision number close to the $x$-axis, then
\[
\{z \text{ floating-point such that } z\approx_p \zeta\}
\subset\{ x+iy \,;\, x\simeq_{p} \Re \zeta \text{ and }y\simeq_{p-\delta} \Im \zeta \}
\]
with $\delta=s(\Re \zeta)-s(\Im \zeta)\in\N$.
Indeed, denoting $\zeta=\xi+i\eta$ and $z=x+iy$,~\eqref{equ:sz} ensures in this case $s(\zeta)\leq s(\xi)+1$ and
$s(x-\xi)\leq s(z-\zeta) \leq \max\{s(z),s(\zeta)\} - p -2 \leq \max\{s(x),s(\xi)\}- p-1$ allows us to invoke~\eqref{equ:eqp_float}.
Similarly, $s(y-\eta)\leq s(z-\zeta) \leq \max\{s(x),s(\xi)\}- p-1\leq\max\{s(y),s(\eta)\}+\delta- p-1$.
This configuration is illustrated in Figure~\ref{fig:fpComplex}.

\begin{remark}
To build a finite arithmetic theory that is truly rotation invariant, one should use ball arithmetic.
Its superiority is demonstrated in~\cite{MV21}, to provide estimates that remain significant after many
iterations of a conformal map.
\end{remark}

\begin{remark}[\bf on cancelations]\label{rmk:cancel}
If $x,y$ are two $p$-bit floating-point numbers such that $y=_q -x$ for some $1\leq q<p$,
then~$|s(x)-s(y)|\leq 1$ and $|x+y|< 2^{\max\{s(x),s(y)\}-q-1}$, \ie
\begin{equation}\label{eq:scale_cancel}
\max\{s(x),s(y)\} - s(x+y) \geq q+1 \,.
\end{equation}
Conversely, if~\eqref{eq:scale_cancel} holds, then
\[
|x+y| <2^{-q} \times 2^{\max\{s(x),s(y)\}-1} \leq 2^{-q} \max\{|x|,|y|\} \,.
\]
In particular, if~$y =_p -x$, all bits cancel out and~$s(x+y)=-\infty$.
The inequality~\eqref{eq:scale_cancel} expresses that at least $q+1$ leading bits, including the implicit leading $\xi_0=1$ in~\eqref{equ:prec},
cancel each other in the addition of~$x$ and~$y$.
It is therefore possible to estimate the loss of precision by comparing the scales of the operands with that of the result: when
\[
0\leq \max\{s(x),s(y)\} - s(x+y)<\infty \,,
\] this value is the exact number of leading bits lost in the operation.
\end{remark}

\section{A bit of geometry}
\label{sect:analyse}

In this brief section we introduce some geometric tools that will be useful in the proof of Theorem~\ref{thm:algo}.
For any real valued map $g$ defined on $[0,1]$, let us define the horizontal strip of height $\delta > 0$
under the graph of $g$ as follows:
\begin{equation}\label{equ:sfd}
S(g, \delta)\defequal  \SetDef{(x,y) \in [0,1] \times \R}{g(x) - \delta \leq y \leq g(x)} .
\end{equation}
The set $S(g, \delta)$ is the intersection of the \textit{subgraph} of $g$ with the \textit{epigraph} of $g-\delta$.

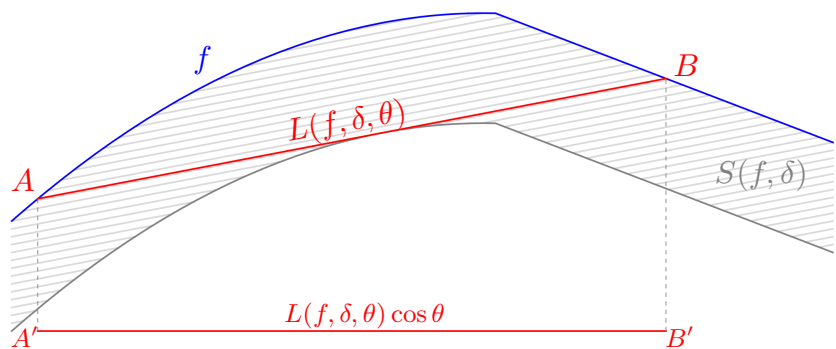
\begin{figure}[H]
\captionsetup{width=.85\linewidth}
\begin{center}
\resizebox{.77\linewidth}{!}{%
\setlength{\unitlength}{1cm}
\includegraphics[width=10cm]{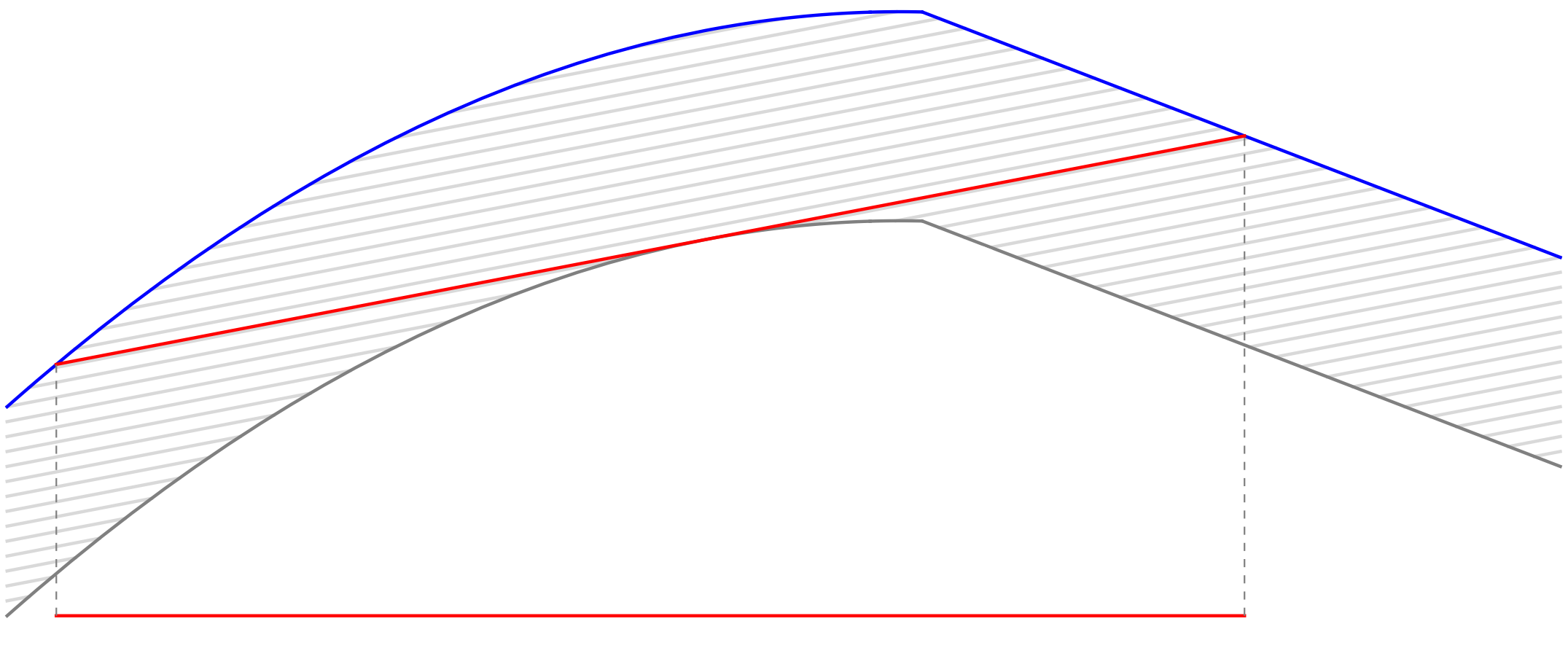}
\begin{picture}(0, 0)(10.1,0)
		\put(2.2, 3.4){\color{blue} $f$} 
		\put(8.5, 2.0){\color{gray} $S(f,\delta)$}
		\put(0,1.9){\color{red} $A$}
		\put(8,3.3){\color{red} $B$}
 		\put(3.3,2.5){\rotatebox{10.5}{\color{red} $L(f,\delta,\theta )$}}
		\put(0.0,0){\color{red}\footnotesize $A'$}
		\put(7.9,0){\color{red}\footnotesize $B'$}
		\put(3.3,0.3){\color{red}\footnotesize $L(f,\delta,\theta )\cos\theta $}
\end{picture}
}
\caption{\label{fig:strip}%
A concave function $f$, the strip $S(f,\delta)$,
the segment $[AB]$ of slope $\tan \theta $ and maximal length $L(f,\delta,\theta )$
and its horizontal projection~$[A'B']$, which is of length $L(f,\delta,\theta )\cos\theta $.
}
\end{center}
\end{figure}

\medskip
Let us denote by $\F$ the set of concave functions on $[0,1]$ \ie functions whose subgraph is a convex set.
For  all $f \in \F$ and $x,y,\lambda \in [0,1]$, one has
\[
f\left( (1 - \lambda) x + \lambda y \right) \geq (1 - \lambda) f(x)+ \lambda f(y) \,.
\]
For $f \in \F$, $\delta > 0$ and $\theta  \in (-\pi/2, \pi/2)$, we denote by $L(f, \delta, \theta )$
the maximal length of a segment of slope $\tan \theta $ contained in the strip $S(f, \delta)$, \ie
\begin{equation}\label{def_L}
L(f, \delta, \theta ) = \sup\left\{ |AB| \,;\, 
[AB]\subset  S(f,\delta) \enspace\text{and}\enspace \text{slope}(AB)=\tan\theta 
\right\},
\end{equation}
where the segment $[AB]=\{ (1-\lambda)A + \lambda B \,;\, \lambda\in[0,1] \}$.
These definitions are illustrated in Figure~\ref{fig:strip}.
One has $L(f,\delta,\theta )\leq \frac{1}{\cos\theta }$ because $|A'B'| = L(f,\delta,\theta )\cos\theta \leq 1$.

\medskip
The two following statements are key for estimating the complexity of the algorithm that
is presented in Section~\ref{sect:algo}.
\begin{thm}\label{thm:conc}
For all $f \in \F$ and all $\delta \in (0,1)$, one has
\begin{equation}\label{main_estim}
\frac{1}{\pi}\int_{ - \frac \pi 2}^{\frac \pi 2} L(f, \delta, \theta ) \cos \theta  \, d\theta
< 1.8644 \sqrt{\delta} \,.
\end{equation}
Moreover, there exists a function $f_0\in \F$ that satisfies a lower bound $>1.1128\sqrt{\delta}$
for $\delta$ small enough.
In the same conditions, one has
\begin{equation}\label{main_estim_bis}
0.91531\sqrt{\delta} < 
\sup_{f\in \F}\left(
\frac{1}{\pi}\int_{ - \frac \pi 2}^{\frac \pi 2} L(f, \delta, \theta ) \cos^2 \theta  \, d\theta 
\right)
<1.3505\sqrt{\delta} \,.
\end{equation}
\end{thm}
\noindent%
We use the following variant for computing averages on the Riemann sphere.
\begin{thm}\label{thm:conc2}
For any positive even weight $\omega \in L^\infty\left(-\frac{\pi}{2},\frac{\pi}{2}\right)$ decreasing on $[0,\pi/2)$
and such that $\omega\left(\frac{\pi}{2}-t\right)\leq C \left|\ln t\right|^{-\beta}$ with $\beta>1$ as $t\to0^+$, one has
\begin{equation}\label{main_estim_ter}
\int_{ - \frac \pi 2}^{\frac \pi 2} L(f, \delta, \theta ) \, \omega(\theta)  \, d\theta 
\leq C_\omega  \sqrt{\delta}
\end{equation}
with
\begin{equation}\label{main_estim_ter_techsupport}
C_\omega = 
\sqrt{2} \|\omega\|_{L^\infty\left(-\frac{\pi}{4},\frac{\pi}{4}\right)} + 4 \int_{\frac{\sqrt{2}}{4}}^\infty 
\frac{\omega\left(\arctan(\frac{1}{2}-\sqrt{2}x)\right)}{\sqrt{1+2\left(x-\frac{\sqrt 2}{4} \right)^2}} dx<\infty \,.
\end{equation}
\end{thm}
\noindent{Note that, in this second statement, any normalization factor is included within $\omega$.}

\medskip
Subsequently, we will apply these results to a function $f$ that is a renormalized concave cover of
the scales of the coefficients and a gap $\delta$ that depends on the precision of the computations
(see equation~\eqref{eq:def_f_from_Ep} and Figure~\ref{fig:avgcoscos2}).
We postpone the proof of Theorems~\ref{thm:conc} and~\ref{thm:conc2}
to Appendix~\ref{par:geom_proof}.

\section{The \A{} algorithm}
\label{sect:algo}

We now focus on a new \emph{Fast Polynomial Evaluator} algorithm, or~\A{} for short.

\subsection{Key idea: lazy polynomial evaluation}\label{par:lazy}

Consider a polynomial $P \in \C[X]$, which we identify to the entire function
\[
P(z) = a_0 + a_1 z + \ldots + a_d z^d \,,
\]
where $a_i \in \C$, $i \in \ii{0}{d}$ and $a_d \neq 0$. The degree of $P$ is $\deg(P)\defequal d$. 

\bigskip
The general idea of the \A{} algorithm is to perform as many lazy additions~\eqref{equ:sum} as possible.
We cut down the cost by not computing the monomials that will have no influence on the final result.
More precisely, with \textit{minimal overhead} (preprocessing), we identify the favorable cases where
it will be safe to perform lazy additions (this is the non-trivial and novel point, as the value $P(z)$ is not yet known), and in the remaining cases, we apply a variant of H\"orner's method.
All unnecessary monomials are thus left out.

\begin{example}
The simplest case study of the \A{} algorithm is the following.
Let $P(z)=1+z$ and assume we perform computations with precision $p$.
If $s(z)<-p - 1$ then $P(z) \simeq_p 1$ and  if $s(z)>p + 1$ then $P(z) \simeq_p z$.
So in these two cases, we get the result for free. An actual computation is only required
in the remaining case,~\ie when $|s(z)|\leq p+1$. If $z$ is uniformly distributed on the
Riemann sphere $\CC$, then on average  (see Section~\ref{sect:eval}), the result is computed in
\[
\int_{2^{-p-1}}^{2^{p+1}} \frac{2rdr}{(1+r^2)^2}
= \tanh\left((p+1)\ln 2\right) 
= 1 - 2^{-1-2p} + O(2^{-3-4p})
\]
operations instead of $1$. In this simplest case, the gain is negligible.
\end{example}

It turns out that the lazy evaluation method performs steadily better in terms of
arithmetic complexity (and speed) as the degree of~$P$ gets higher.
Compared to H\"orner's scheme, the gain is substantial (\eg $O(\sqrt{d\log d})$ instead of $O(d)$ for computations
in machine precision) and holds for \emph{every} polynomial
once we average out over all the possible scales of the evaluation point.
Before presenting the general case, we illustrate our algorithm and this phenomenon
below, on a polynomial of degree~$10$ (see Figure~\ref{fig:concave}).

Let us point out that the~\A{} algorithm thrives in the case of \textit{multiple evaluations}, because
the initial analysis needs not be repeated. We illustrate this fact on the example below too,
by showing how  the analysis at the points $z=\pm1$ can easily be transposed to an arbitrary value of~$z$
(see Section~\ref{par:ex2general}).
By construction, subsequent evaluations cannot exceed the
complexity of H\"orner's scheme and will, on average, be much better.

\smallskip
Contrary to the FFT or the Fast Multipoint Algorithm, the~\A{} algorithm is \textit{local} in the sense that
the evaluation at a given point is independent of the precise computations that are needed
to evaluate $P$ at another point. One can thus expect the algorithm to have the same numerical
stability as H\"orner's method.
The memory requirements for each evaluation are also minimal because the overhead storage is
negligible in comparison to that of the coefficients.
For example, our algorithm will thrive in implementations of Newton's method to find a single
root of a polynomial of large degree because the list of evaluation points is, obviously, not known in advance.

\smallskip
Finally, without increasing the arithmetic complexity,
it is possible to complement the result $P(z)$ with a \textit{confidence estimator}
that indicates how many of the $p$ bits may have suffered from cancelations
(for details, see Remarks~\ref{rmk:cancel} and~\ref{rmk:HornerAndBitLoss}).
This feature is part of our implementation \cite{FPELib} (see Section~\ref{sect:code}, in particular Figure~\ref{fig:estimPrecSphere}).
It may help if finding the proper value of $p$ is part of the problem,
\eg in the implementation of Newton's method with a dynamically adjusted precision.
Note that changing the value of $p$ will require a new preprocessing
of the polynomial (see Remark~\ref{rem:partialpreproc}).

\newcommand{\ep}{\wideparen{E}_P}
\medskip
In preparation for the general case, let us introduce the following notation.
\begin{definition}
The scales of the coefficients are modeled by the function $E_P:\ii{0}{d} \rightarrow \Z \cup \{-\infty\}$ defined by
\begin{equation}\label{def:EP}
\forall i \in \ii{0}{d}, \quad  E_P(i)\defequal s(a_i) \,.
\end{equation}
We will denote by $\ep$ the concave cover of~$E_P$, that is the minimal real concave
function on $[0,d]$ such that $E_P \leq \ep$.
Obviously, $\ep$ is piecewise linear.
\end{definition}

\subsection{A simple example detailed}

We analyze an example depicted in Figure \ref{fig:concave}, which represents $E_P$
and $\ep$ for a particular polynomial $P$ of degree $10$ with non-zero coefficients.
The scales of the coefficients are readable on the graphic; the actual values of the
phase of each coefficient are irrelevant to the discussion.

\[
\begin{array}{|c|ccccccccccc|}
\hline
\text{Coefficient} & a_0 & a_1 & a_{2} &  a_{3} &  a_{4} &  a_{5} &  a_{6} &  a_{7} &  a_{8} &  a_{9} &  a_{10}
\\\hline
|a_k|\rule{0pt}{13pt} & 2^{-3} & 2^{5} & 2^{-4} & 2^{15} & 2^{13} & 2^{-5} & 2^{26} & 2^{15} & 2^{29} & 2^{29} & 2^{17}
\\\hline
s(a_k) & -2 & 6 & -3 & 16 & 14 & -4 & 27 & 16 & 30 & 30  & 18
\\\hline
\end{array}
\]

\medskip
To keep this example simple, we compute $P(x)$ with a fixed precision $p=6$. 
The strip~$S(\ep,p)$ defined by~\eqref{equ:sfd} is a polygonal band
of vertical thickness $p$. For $\lambda=\tan\theta $, 
let  us also denote by $L_{-\lambda}$ the longest segment of slope $\lambda$ contained in~$S(\ep ,p)$,
that is
\begin{equation}
L_{-\lambda} = L(\ep,p,\theta )
\end{equation}
with $L$ defined by~\eqref{def_L}; see~Section~\ref{sect:analyse} if necessary.
The reason for the sign convention will appear in Section~\ref{par:ex2general}.

\definecolor{DarkGray}{RGB}{80, 80, 80}
\begin{figure}[p]
\begin{center}
\resizebox{0.853\linewidth}{!}{%
\setlength{\unitlength}{1cm}
\includegraphics[width=14cm]{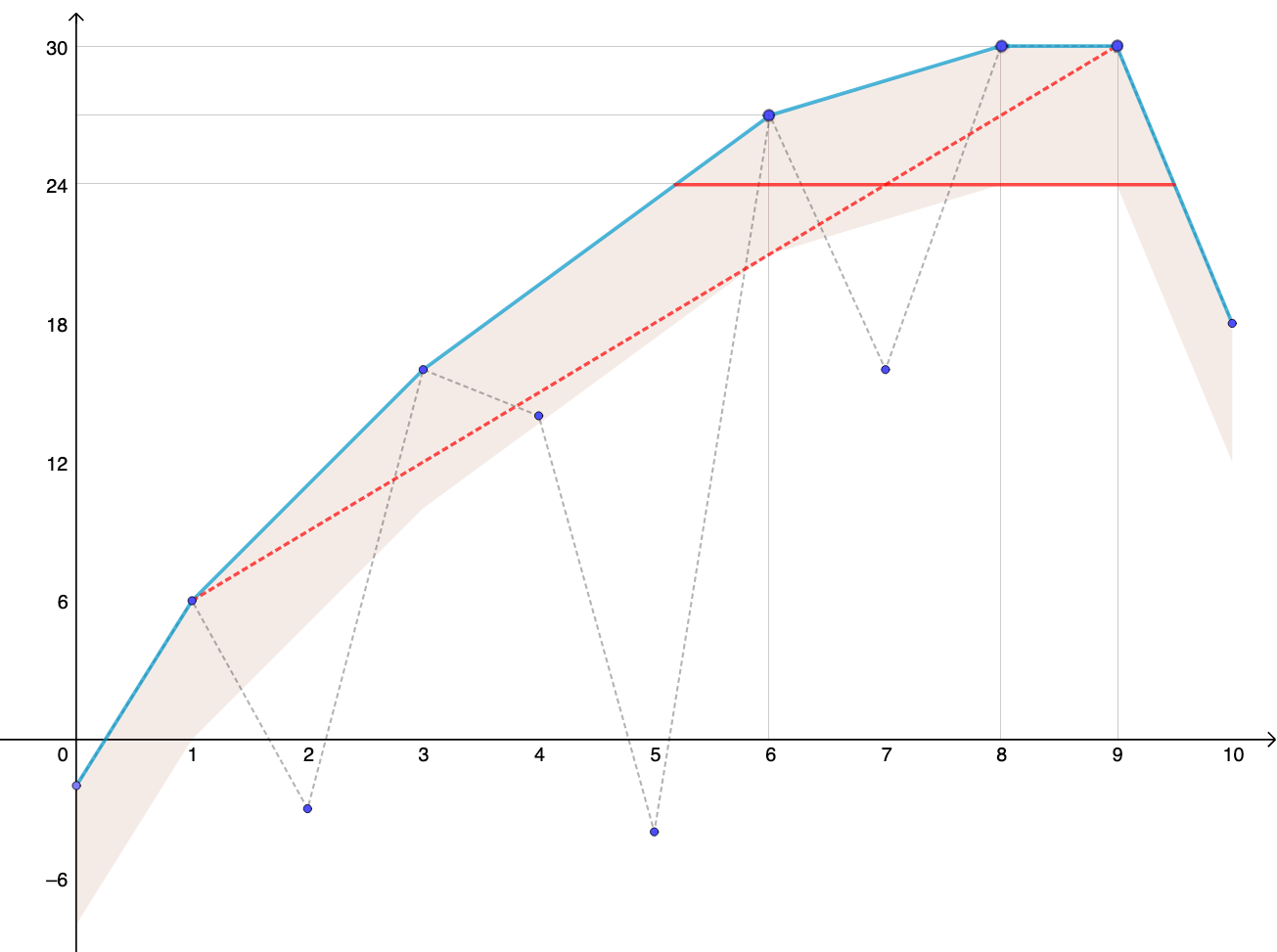}
\begin{picture}(0, 0)(14.1,0)
		\put(3,8.43){\color{Maroon} \vector(0,1){1.5}}
		\put(3.2,9.1){\color{Maroon} $p$}
		\put(7, 4.8){\color{DarkGray} $E_P$}
		\put(5.5, 7.8){\color{Cyan} $\ep$}
		\put(10.47,10.2){\color{blue} {\footnotesize $(k_0,s_0)$}}
		\put(12.7, 9.1){\color{Maroon} $S \big( \ep , p \big)$}
		\put(11.4, 8.58){\color{Red} $L_0$}
		\put(6.7, 6.9){\rotatebox{31}{\color{Red} $L_{-3}$}}
		\put(13.2, 2.6){$k$}
		\put(1.0, 10.2){$s(a_k)$}
\end{picture}
}
\caption{\label{fig:concave}
Example illustrating $E_P$, $ \ep$, $S \big( \ep ,p \big)$ for $p=6$, $L_0$ and $L_{-3}$.
The horizontal segment $L_0$ isolates the (three) coefficients of $Q_0(z)$, which is the suitable
reduction of $P(z)$ when the evaluation occurs on the unit circle.}
\end{center}
\end{figure}
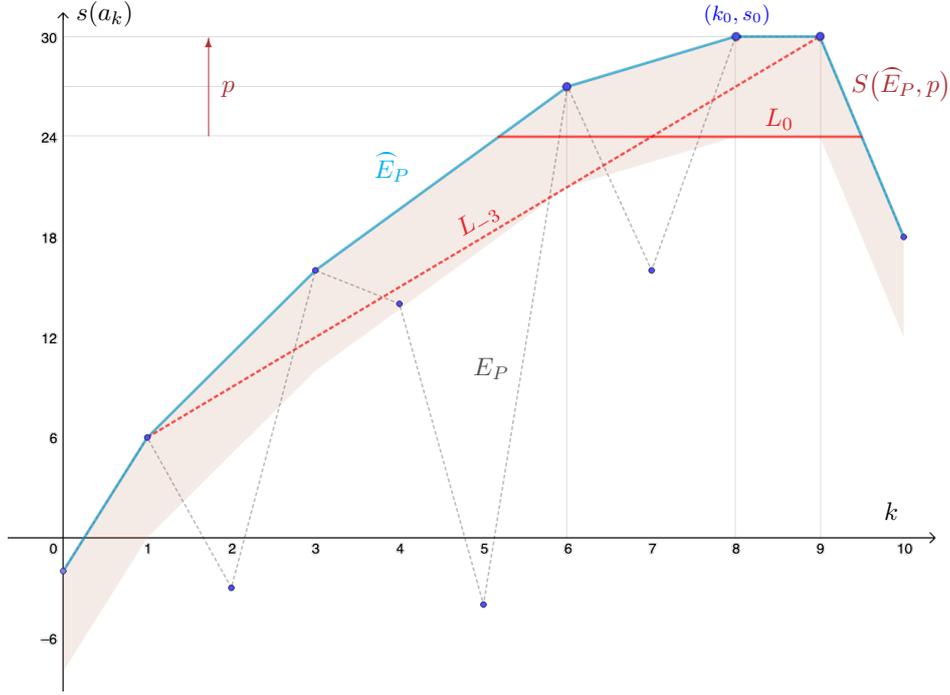

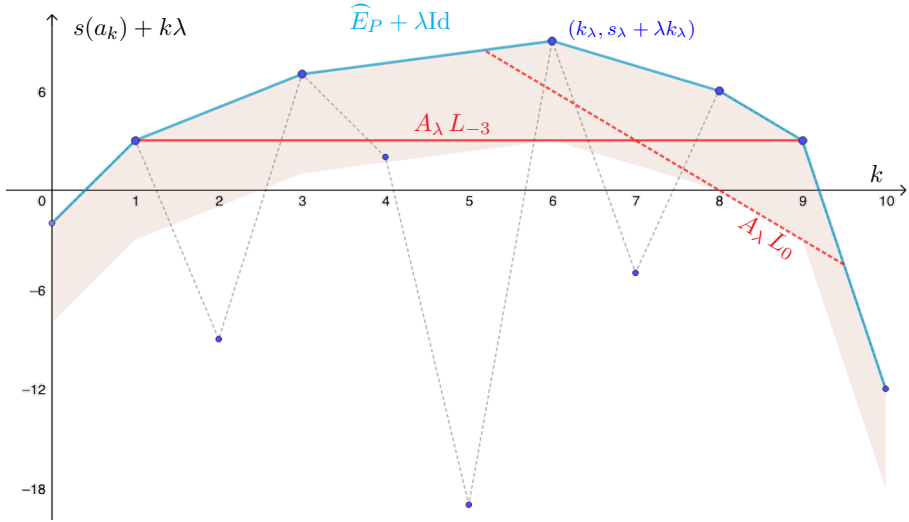
\begin{figure}[p]
\setlength{\unitlength}{1cm}
\begin{center}
\resizebox{0.82\linewidth}{!}{%
\includegraphics[width=14cm]{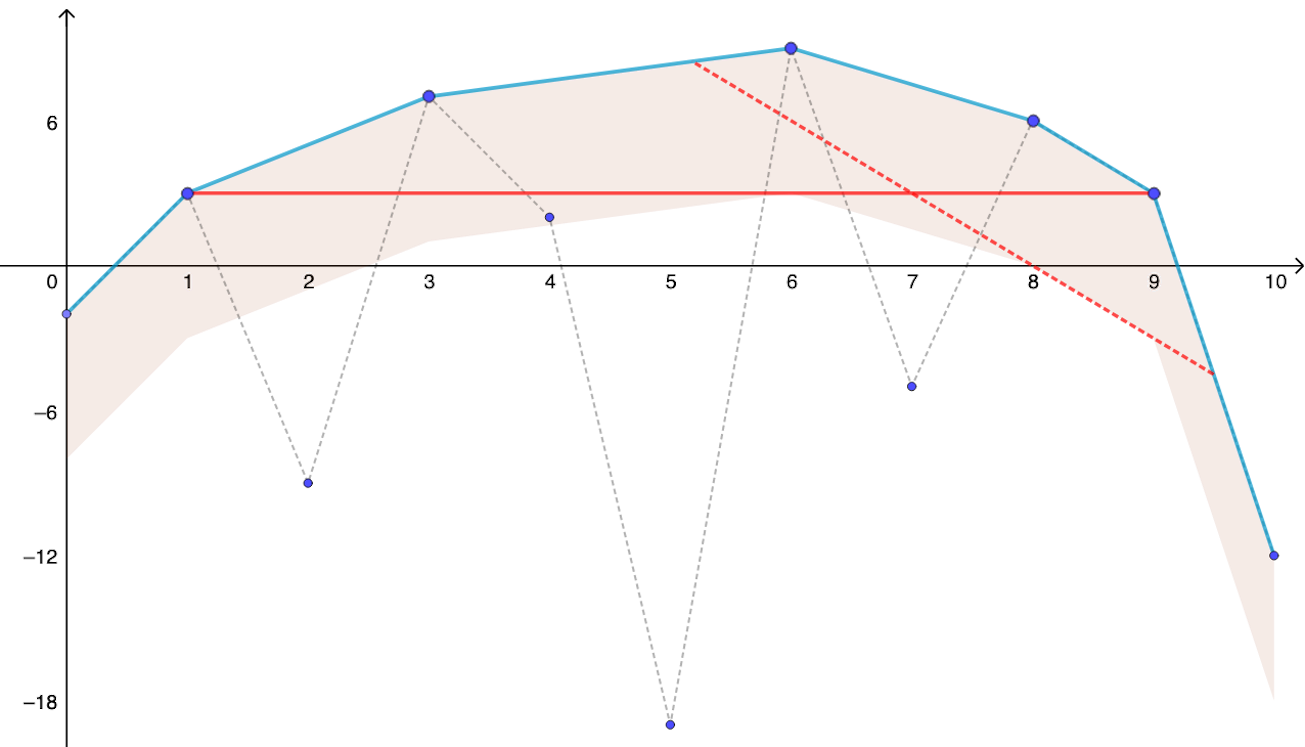}
\begin{picture}(0, 0)(14.1,0)
		\put(5.3, 7.7){\color{Cyan} $ \ep + \lambda \mathrm{Id}$}
		\put(11.3, 4.6){\rotatebox{-31}{\color{Red} $A_\lambda \, L_0$}}
		\put(6.3, 6.1){\color{Red} $A_\lambda\, L_{-3}$}
		\put(8.7, 7.6){\color{blue} {\footnotesize $(k_\lambda,s_\lambda+\lambda k_\lambda)$}}
		\put(13.4, 5.3){$k$}
		\put(1.0, 7.6){$s(a_k) + k\lambda$}
\end{picture}
}
\caption{\label{fig:shear}
Transformation of Figure \ref{fig:concave} by the affine map $A_\lambda$ (defined below) for $\lambda=-3$.
The image of the concave cover $\ep$ is $\ep + \lambda \mathrm{Id}$.
The horizontal segment $A_\lambda L_\lambda$ isolates the (five) coefficients of $Q_\lambda(z)$,
which is the reduction of $P(z)$ when $|z|= 2^\lambda$.}
\end{center}
\end{figure}

\subsubsection{Evaluation on the unit circle}

For the previous example,
let us first consider the case where the evaluation point~$z$ satisfies $|z|=1$.
For $k \in \ii{0}{10}$, we have $s(a_k z^k) = s(a_k)$.
To select the largest monomials in $P(z)$, let us consider the segment $L_0$ on Figure~\ref{fig:concave},
which is the horizontal segment situated at a scale $p$ below the scale of the monomial with maximal scale.
The monomials we keep are the ones above $L_0$.  In this example, we get
\[
Q_0(z) \defequal  a_6z^6  + a_8z^8 + a_9 z^9 \,.
\]
Let us now check that the floating-point value of~$P(z)$ coincides indeed with $Q_0(z)$.
If $k \notin \{6,8,9\}$, then $s(a_k z^k) \leq 18$ and the inequality \eqref{equ:bs} with $N=8$ implies
\[
s(P(z) - Q_0(z)) \leq 18+4 = 22
\quad\ie\quad |P(z) - Q_0(z)|<2^{22} \,.
\]
On the other hand, unless there is an exceptional cancelation of the leading term while computing $Q_0(z)$,
one has $s(Q_0(z)) \geq 30$.
In particular, one has $s(P(z)-Q_0(z))\leq 22\leq\max\{s(P(z)),s(Q_0(z))\}-p-2$
and \eqref{equ:eqpComplex} ensures that
\[
P(z) \approx_p  Q_0(z) \,,
\]
as long as the leading terms of $Q_0(z)$ do not cancel each other.

\bigskip
Let us now investigate the possible cancelations within $Q_0$.
As computations are restricted to $p=6$ bits, the last bit of both $a_8z^8$ and $a_9z^9$
represents a rounding interval of the real line of radius~$\frac{1}{2}\times 2^{30-6-1} = 2^{22}$.
This is the best error bound for $a_8z^8$ and $a_9z^9$ that we can hope for.
Thus, even if a cancelation of the most significant bits occurs,
the error when computing $a_8z^8 + a_9z^9$ cannot be bounded to less than~$2^{23}$.
Consequently, as the bound for the error can only increase with more terms in the sum, the uncertainty on
the value of $Q_0(z)$ will  exceed $2^{23}$. 
On the other hand, we have checked that $|P(z) - Q_0(z)|<2^{22}$,
so this difference is always smaller than the error bound in the computation of both~$Q_0(z)$ and~$P(z)$.
Therefore, one can claim that, when $|z|=1$,  the value~$Q_0(z)$
is always a good floating-point substitute for $P(z)$, \ie within the error bounds for 6 bits of precision throughout the
computation, even if significant bits are canceled out.

\medskip
To get a geometric feeling in this instance of our algorithm,
observe that in the computation of $Q_0(z)$ we have only used terms that are
above the line $L_0$, which is  is the longest horizontal segment contained in $S(\ep, p)$.
This line lies at the scale level $\max_{k}s(a_k z^k) - p$.
Note also that we did not use a precision larger than $p$ for any intermediary result.

\def\buff{3}  
\def\buffmo{2} 
\def\buffpo{4} 
\subsubsection{Towards the general case}\label{par:ex2general}

Let us continue the analysis of the previous example for a general evaluation point,
\ie when~$z \in \C$.
As $P(0)=a_0$ is immediately available, let us assume $z\neq0$ and consider
\begin{equation}\label{eq:defLambda}
\lambda \defequal  \log_2 |z| \,.
\end{equation}
The product formulas \eqref{equ:sp} and \eqref{equ:sn} imply that, for all $k \in \ii{0}{d}$, we have 
\begin{equation}
\label{equ:shear}
| s(a_k) + k \lambda - s(a_k z^k) | \leq 1\,.
\end{equation}
As before, we seek a simpler polynomial $Q_\lambda$ such that it is sufficient to
evaluate~$Q_\lambda$ at~$z$ instead of~$P$ when the computations
are done with $p=6$ bits.

In order to visualize the polynomial $Q_\lambda$, let us consider the image of Figure~\ref{fig:concave} by
the following affine map of $\R^2$
\[
A_\lambda: \begin{pmatrix}x\\y
\end{pmatrix}
 \mapsto
\begin{pmatrix}
  1 & 0 \\
  \lambda & 1
\end{pmatrix}
\begin{pmatrix}x\\y
\end{pmatrix}.
\]
By definition, $A_\lambda$ maps lines of slope $-\lambda$ to the horizontal.
For $\lambda=-3$, we consider the reduced polynomial
\[ 
Q_{-3}(z) \defequal   a_1z + a_3z^3 + a_6z^6  + a_8z^8 + a_9 z^9\,.
\]
In other words, we select the powers $k \in \ii{0}{d}$ such that $s(a_k) + k \lambda$ is above the horizontal line $A_{\lambda}\, L_{\lambda}$.

\[
\begin{array}{|c|ccccccccccc|}
\hline
\text{Coefficient} & a_0 & a_1 & a_{2} &  a_{3} &  a_{4} &  a_{5} &  a_{6} &  a_{7} &  a_{8} &  a_{9} &  a_{10}
\\\hline
|a_k|\rule{0pt}{13pt} & 2^{-3} & 2^{5} & 2^{-4} & 2^{15} & 2^{13} & 2^{-5} & 2^{26} & 2^{15} & 2^{29} & 2^{29} & 2^{17}
\\\hline
s(a_k z^k) & -2 & 3 & -9 & 7 & 2 & -19 & 9  & -5 & 6 & 3  & -12
\\\hline
\end{array}
\]

Figure~\ref{fig:shear} and the table above indicate that $s(Q_{-3}(z))=9$ so $\ulp(Q_{-3}(z))=2^2$.
Consequently,~\eqref{equ:bs} ensures that $s(P(z)-Q_{-3}(z))\leq 2+3=5$ \ie $|P(z)-Q_{-3}(z)|<2^5$. In particular, one has
\[
P(z) \approx_{p-4} Q_{-3}(z)\,.
\]
The precision loss (4 bits out of 6) may seem significant in this example. However, in general,
the loss is capped by $s(d) + \buff$, which means that an offset on the thickness of the $S(\ep,\cdot)$
strip will be enough to deal with the general case.

\begin{remark}
In the general statement of the~\A{} algorithm (see~\eqref{eq:threshold} in~Section~\ref{par:algo}),
we will use the scale threshold $\max_{k}s(a_kz^k) - p - s(d) - \buff$
instead of $\max_{k}s(a_kz^k) - p$ (used in Figures \ref{fig:concave} and \ref{fig:shear})
to prevent interaction between $Q_\lambda$ and $P-Q_\lambda$ and to secure upper bounds of $s(a_kz^k)$.
\end{remark}

We may now link the statement of Theorems~\ref{thm:conc} and~\ref{thm:conc2} to our algorithm.
As the sum of two concave maps is concave, the map $\kappa \mapsto \ep(\kappa) + \lambda \kappa$ is concave.
Therefore, there are at most $|A_{\lambda}\, L_{\lambda}| + 1$ terms in the reduced polynomial $Q_{\lambda}$,
where $|L|$ is the length of a segment~$L$. Observe that
\begin{equation}
\label{equ:cos}
| A_{\lambda}\, L_{\lambda} | = |L_\lambda| \cos \theta\,,
\end{equation}
where  $\theta  \in \big( -\pi/2, \pi/2 \big)$ satisfies $\lambda = \tan \theta $.
To estimate the average reduction in complexity of our algorithm over H\"orner's scheme, 
we are interested in averaging the number of monomials of $P(z)$ that are ultimately evaluated.
We will therefore compute the average value of $L(f,\delta,\theta )\, \cos\theta $
where $f \in \F$  is a renormalized version of~$\ep$, defined for $x \in [0,1]$ by
\begin{equation}\label{eq:def_f_from_Ep}
f(x) = \frac{\ep\left( d \, x \right)}{d} \quad\text{and}\quad
\delta \defequal  \frac{p +s(d) + \buff}{d}
\,\cdotp
\end{equation}
Depending on how the values $z$ are chosen in $\C$, various weights for $\theta  \in \left( -\pi/2, \pi/2 \right)$
are used (see~Section~\ref{sect:eval}).

\subsection{Statement of the algorithm}
\label{par:algo}

We are given a precision $p \geq 1$ and a polynomial expression
\[
P(z)=\sum_{j=0}^d a_j z^j
\]
in $\C[X]$ with $d = \deg P \geq 1$.
We will also assume that $a_0 \neq 0$, otherwise we reduce the problem to a lower
degree polynomial $z^{-k}P(z)$ for some~$k \geq 1$.

\medskip
We can formalize our evaluation algorithm $\A_p$ as follows. 
Each non-trivial operation has its time (bit) complexity marked as a comment on the right.
The Figures~\ref{fig:concave} and~\ref{fig:shear} illustrate the algorithm.

\bigskip
\SetKwFor{ForEach}{for~each}{do}{endfch}
\renewcommand{\algorithmcfname}{Algorithm $\A_p$}
\begin{algorithm}[H]
\caption{Fast evaluation of complex polynomials with precision $p$}
\label{algo:fast}
\DontPrintSemicolon
\SetKwRepeat{Do}{do}{while}
\smallskip
  \KwData{The list of \cod{coefficients} $a_0,\ldots,a_d$ and the precision $p \geq 1$}
\Begin(\textbf{preconditioning}){
  \cod{compute and sort $s(a_k)$, $k \in \ii{0}{d}$ \tcc*{$d\log_2 d$}}
  \cod{compute the concave map $\ep$ \tcc*{$d\log_2 d$}}
  \cod{list $G_p=\left\{k\in \ii{0}{d}\,;\,  s(a_k) \geq \ep(k) - p - s(d) - \buff\right\}$  \tcc*{$d$}}
}
\bigskip
\KwData{Pre-conditioned $P$ at precision $p$}
\KwData{Finite subset $Z$ of $\C^\ast$ (evaluation points)}
\Begin(\textbf{evaluation}){
 \ForEach{$z\in Z$}{
  \cod{let $\lambda \defequal  \log_2|z|$\tcc*{1}}
  \cod{compute $k_\lambda\defequal \argmax\big(\ep + \lambda \mathrm{Id}\big)$  \tcc*{$\log_2 d$}}
  \cod{let $N\defequal \ep(k_\lambda) + \lambda k_\lambda$\tcc*{1}}
  \cod{compute $\{\ell,r\}\defequal (\ep + \lambda \mathrm{Id}\big)^{-1}(N - p - s(d) - \buff)$ \tcc*{$2\log_2 d$}}
  \cod{compute and output $Q_\lambda(z) = \hspace*{-2ex}\sum\limits_{k\in G_p\cap\ii{\ell}{r} } \hspace*{-1.5ex}a_kz^k$
    \tcc*{avg $\lesssim 2.7 M(p) \sqrt{d p}$}}
 }
}
\end{algorithm}
\noindent
See Section~\ref{par:scale} for the definition of the scale functions $s$
and the end of Section~\ref{par:lazy} for that of  the concave cover~$\ep$
of~$s(a_j)$. 

\subsection{Statement of the main results}
\label{par:algoResults}

In the following subsections, we describe in detail each step of the algorithm.
Subsequently, we prove its correctness, \ie the statement of Theorem~\ref{thm:equivAlgo},
and we compute the time complexity of~$\A_p$ as stated in Theorem~\ref{thm:algo}
and equation~\eqref{eq:complexityAlgo}.

\begin{thm}\label{thm:equivAlgo}
Given $P$ as above and a precision $p\geq1$, for each $z\in\C^\ast$ and $\lambda=\log_2|z|$,
there exists a polynomial subexpression $Q_\lambda$
of $P$ such that (see Section~\ref{par:equivModP}):
\begin{equation}\label{eq:main}
P(z) \approx_{p-c} Q_\lambda(z) \,,
\end{equation}
where the number of canceled bits $c\geq0$ is defined by
\begin{equation}\label{eq:def_prec_loss}
c = \begin{cases}
0 & \text{if } |P(z)| \geq \mathcal{M},\\
s(\mathcal{M}) - s(P(z)) & \text{otherwise,}
\end{cases}
\end{equation}
and $\mathcal{M}\defequal  \max\limits_{j\in\ii{0}{d}}  |a_j z^j|$.
The reduced polynomial $Q_\lambda$ is given by
\begin{equation}\label{eq:reducedQlambda}
Q_\lambda(z) \defequal  \sum\limits_{k\in G_p\cap\ii{\ell}{r} } \hspace*{-1.5ex}\widetilde{a_k}z^k \,,
\end{equation}
where $\widetilde{a_k}$ is the $p$-bit floating-point representation of $a_k$
and where $\ell$, $r$ and $G_p$ are computed by the algorithm $\A_p$ described above.
The number of monomials of $Q_\lambda$ satisfies 
\begin{equation}
\avg_{\CC} \left(\#\ii{\ell}{r}\right) <  1+1.9046 \sqrt{d(p +s(d) + \buff)} \,,
\end{equation}
where the average is taken with respect to the uniform distribution of $z\in\CC$ on the Riemann sphere.
Moreover, for any $d$ and $p$, there exists a polynomial $P$ for which~$G_p = \ii{0}{d}$
and such that $\avg_{\CC} \left(\#\ii{\ell}{r}\right) > 1.3217 \sqrt{d(p +s(d) + \buff)}$ as $d\to\infty$.
\end{thm}
\begin{remark}\label{rmk:HornerAndBitLoss}
The proof of Theorem~\ref{thm:equivAlgo} ensures that
\begin{equation}
P(z) \approx_{p-c} \sum\limits_{k=0}^d \widetilde{a_k}z^k\,.
\end{equation}
The precision claimed by~\eqref{eq:main} is thus equivalent to that of H\"orner's scheme,
if all coefficients (limited to $p$ bits) had been kept.
Note also that, even though the number of cancelled bits $c$ is defined by~\eqref{eq:def_prec_loss}
and thus depends on the exact value of $P(z)$, it is possible to give a precise upper-bound of $c$ by
using Remark~\ref{rmk:cancel} for each addition that occurs in the computation of~\eqref{eq:reducedQlambda}.
\end{remark}

\medskip
Regarding complexity, the main result is as follows (see Figure~\ref{fig:HornerFPETheory}).

\begin{thm}\label{thm:algo}
Given a polynomial $P \in \C[X]$ of degree $d \geq 1$ and a bit precision $p \in \N^*$,
the preconditioning phase of algorithm $\A_p$ is performed on~$P$ in time~$d+2d\log_2 d$ and
requires $O(dp)$ in memory.
Subsequently, for $z$ uniformly distributed on the Riemann sphere~$\CC$,
the average evaluation time of $P(z)$ by the algorithm $\A_p$ is less than
\begin{equation}\label{eq:compAlgo}
2+3\log_2 d+\left(
2\log_2 d+ 1.9046 \sqrt{d(p+\log_2 d+\buffpo)}
\right)M(p) \,,
\end{equation}
where $M(p)$, recalled in~\eqref{Mp}, denotes
the time of one multiplication followed by an addition of two floating-point numbers with precision $p$.
If $P \in \R_d[X]$ and $x$ is uniformly distributed on the circle $\overline{\R}=\R\cup\{\infty\}$, the average evaluation time
of $P(x)$ by the algorithm $\A_p$ is less than
\begin{equation}\label{eq:compAlgoRe}
2+3\log_2 d+\left(
2\log_2 d+ 1.7673 \sqrt{d(p+\log_2 d+\buffpo)}
\right)M(p) \,.
\end{equation}
In both cases, the bit complexity of  the evaluator never exceeds $2+3\log_2 d + M(p)d$.
\end{thm}

\begin{remark}\label{rmk:complexityPractice}
In practice, $p\geq 52$ because double precision \texttt{FP64}
is implemented in most modern hardware, \ie $M(52)=1$.
One has $d\ll 2^{48}$ (see \eg\cite{MV21} for a record-breaking handling of a tera-polynomial).
In this case, $p +\log_2 d + \buffpo \leq 2p$ and~\eqref{eq:compAlgo} is bounded by
\[
2+3 \log_2 d+ (96+2.7  \sqrt{dp}) M(p)\,,
\]
when $d\gtrsim 100$ and~\eqref{eq:compAlgoRe} by $2+3 \log_2 d+ (96+2.5  \sqrt{dp}) M(p)$ in the real case.
\end{remark}

\begin{remark}
Let us point out that the uniform average over $\CC$ (or $\overline{\R}$) is unfavorable to our algorithm. Near the poles $z=0$ and $z=\infty$,
our algorithm will drop most terms and will thus be very quick. However, a uniform average does not favor those regions: the
area of the region $|z|>10$ (or $|z|<1/10$) represents about $1\%$ of the area of the sphere,
which is of the same order of magnitude as that of the annulus $||z|-1|<10^{-2}$.
Using the techniques exposed in Section~\ref{sect:comp}, one can compute the average complexity for any
particular distribution of evaluation points; for example, the case of a uniform distribution on $D(0,1)$
is treated in Remark~\ref{rmk:case_disk} below, estimate~\eqref{eq:case_disk}.
It is also possible to refine the estimate if the distribution of the coefficients of~$P$ is known (see \eg[~]Figure~\ref{fig:chebyshev}
for Chebyshev polynomials).
\end{remark}

\bigskip
A point is worth underlining: if the algorithm $\A_p$ encounters one ``bad'' case where one evaluation has the same complexity as H\"orner,
then, on average, it will perform much better than~\eqref{eq:compAlgo}.
More precisely, let us assume that one particular choice of $z_0$ with $\log_2 |z_0|=\lambda_0$ leads our algorithm to evaluate all the
monomials of $P(z_0)$, which is the worst case possible. 
Of course, for such a polynomial, our algorithm would not outperform H\"orner
if we were to evaluate only on $z$ in an annulus $|z|\simeq \lambda_0$.
However, from this shortcoming, we learn that the graph of the concave cover of $s(a_k)+k\lambda_0$
rescaled to $[0,1]$ (see \eqref{eq:def_f_from_Ep} and Figure~\ref{fig:shear})
is comprised between two horizontal lines $c$ and $c+\delta$,
\ie the modulus of the coefficients of $P$ are, roughly speaking, varying exponentially.
If we briefly anticipate the computations of Section~\ref{sect:eval}, 
the average number of terms when~$z$ is uniformly distributed on~$\CC$
can be estimated with a simple weight (see Figure~\ref{fig:avgcoscos2})~:
\[
\avg_{\CC} \left(\# \ii{\ell}{r} \right) < 0.46 d \int_{ - \frac \pi 2}^{\frac \pi 2} L(f, \delta, \theta ) \cos \theta  \, d\theta\,.
\]
The computation~\eqref{eq:ex1} from Example~1 then provides an explicit bound:
\begin{equation}\label{eq:dreamWorld}
\avg_{\CC} \left(\# \ii{\ell}{r} \right) < \frac{d \delta |\log\delta|}{1+\lambda_0^2} 
= \frac{p +s(d) + \buff}{1+\lambda_0^2}\left| \log \frac{d}{p +s(d) + \buff}\right|\,\cdotp
\end{equation}
This means that, if our evaluator performs \textsl{once} as poorly as H\"orner, then it will, on average,
perform as $O(p M(p) \log d)$ if the evaluation points are chosen uniformly over the Riemann sphere~$\CC$
and~$\log_2 d\leq p\ll d$ or instead as $O(M(p) \log^2 d )$ if $p\leq \log_2 d$.
This is a much better behavior than the one claimed by Theorem~\ref{thm:algo} in general
and it is the best that we have observed in practice (see Figure~\ref{fig:HornerFPETheory}
and, for details, Section~\ref{sec:bench}).

\begin{remark}\label{rmk:betterIdea}
More generally, if the coefficients of $P$ are a union of a few long geometric progressions (even possibly intertwined),
the graph of $\ep$ will be composed of only a few piecewise straight lines, say $N\ll d$. Each straight line will only be visible
on a finite range of values of $|z|$ and will contribute a logarithmic complexity bounded by~\eqref{eq:dreamWorld}.
The overall average complexity of the $\A_p$ evaluator will then be bounded by
\begin{equation}
O(N M(p) (p+\log d) \log d)
\end{equation}
if the evaluation points are chosen uniformly over the Riemann sphere~$\CC$ (or~$\RR$ in the real case).
\end{remark}

\bigskip
For further details and the construction of an example that saturates the upper bound~\eqref{eq:compAlgo},
see Section~\ref{par:optimal}.

\section{Complexity analysis and proof of Theorem~\ref{thm:algo}}
\label{sect:comp}

In this section, we describe the details of the algorithm $\A_p$ and prove Theorem~\ref{thm:algo} regarding complexity.

\subsection{Analysis of the preconditioning phase}
\label{sect:preCond}

We describe briefly the computation of the concave hull $\ep$.
The first step is standard and consists in obtaining an enumeration $(k_n)$ of $\ii{0}{d}$ to sort the values $s_n\defequal s(a_{k_n})$
in decreasing order, \ie such that for all $n \in \ii{0}{d-1}$,
\[
s_n \geq s_{n+1} \,.
\]
In case of equality, $k_n$ is chosen in increasing order (\ie $k_n\leq k_{n+1}$ if $s(a_{k_n})=s(a_{k_{n+1}})$).
This step can be performed in $d \log_2 d$ operations.
Observe that $k_0 = \argmax \big(\ep \big)$. 
\begin{lemma}
For $n \in \ii{0}{d}$,  we construct a sequence of concave maps $E_n:[\ell_n,r_n] \rightarrow \R$ such that, for all $n$,
$[\ell_n,r_n]$ is the convex hull of $\{ k_0,\ldots,k_n\}$ and
\[
\forall k\in\ii{\ell_n}{r_n}, \qquad s(a_k) \leq E_n(k) \leq \ep(k) \,.
\]
Constructing $E_{n+1}$ knowing $E_n$ is performed in $\log_2 n$ operations.
\end{lemma}
\noindent
Observe in particular that $E_d = \ep$ and that it is obtained in less than $d \log_2 d$ steps once $(k_n)$ is known.
An example of this construction is given on Figure~\ref{fig:pre}.

\bigskip
\begin{figure}[hbt!]
\begin{center}
\resizebox{0.82\textwidth}{!}{%
\setlength{\unitlength}{1cm}
\includegraphics[width=14cm]{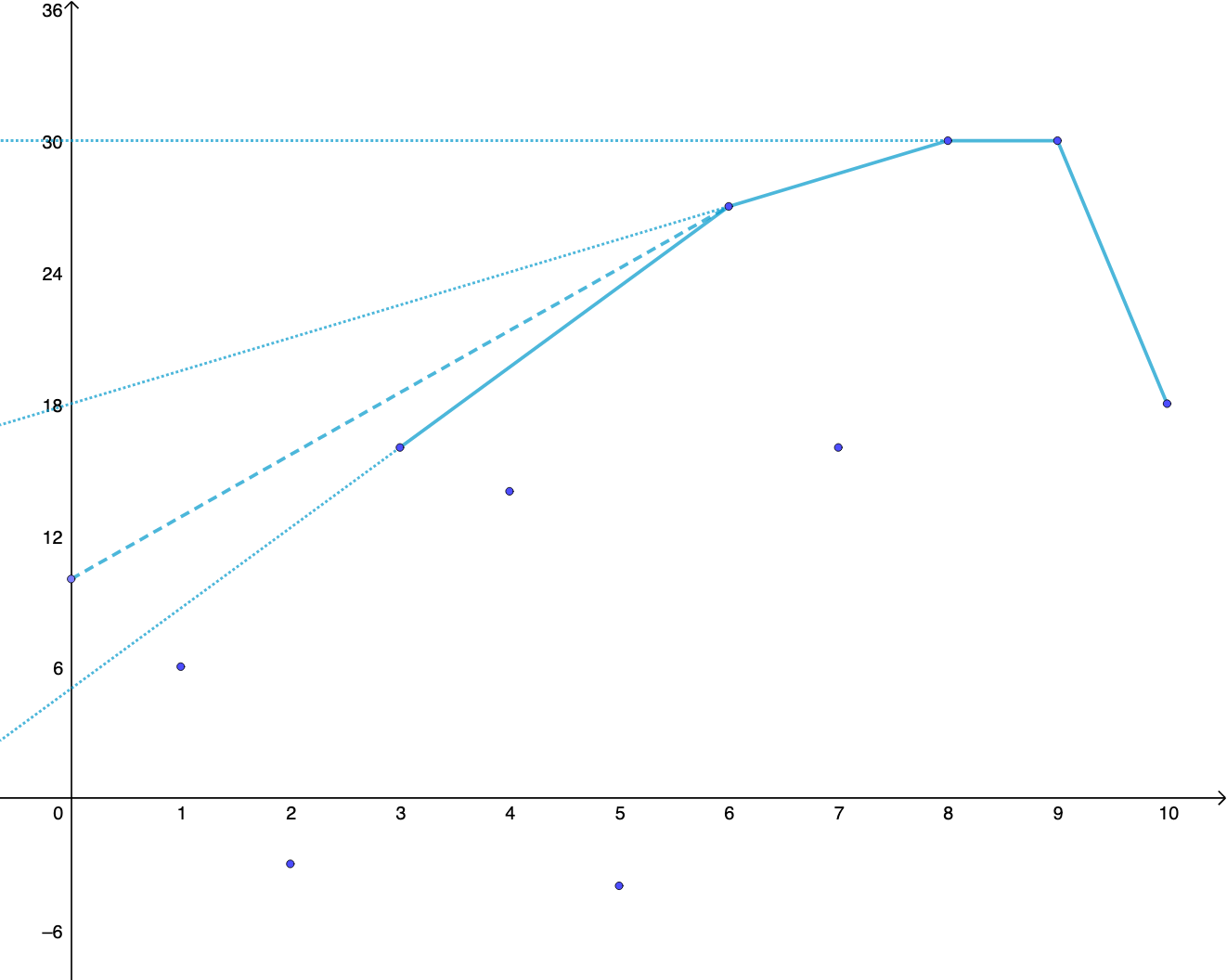}
\begin{picture}(0, 0)(14.1,0)
	\put(11.2, 9.2){\color{Cyan} $E_1$}
	\put(9.3, 8.8){\color{Cyan} $E_2$}
	\put(12.1, 8){\color{Cyan} $E_3$}
	\put(6.4, 6.8){\color{Cyan} $E_4=E_5=E_6$}
	\put(2.2, 5.8){\color{Cyan} $E_7$}
	\put(4.2, 7.9){\color{Cyan} $\widetilde{E}_4$}
	\put(3.3, 4.7){\color{Cyan} $\widetilde{E}_7$}
	\put(13.2, 2.3){$k$}
	\put(1, 10.3){\color{blue} $s(a_k)$}
	\put(5,10.5){\color{blue}\raisebox{.5pt}{\textcircled{\raisebox{-.5pt} {\footnotesize $i$}}} = \footnotesize $(k_i,s_i)$}
	\put(10.6, 9.8){\color{Gray}\raisebox{.5pt}{\textcircled{\raisebox{-.5pt} {\footnotesize 0}}}}
	\put(11.9, 9.8){\color{Gray}\raisebox{.5pt}{\textcircled{\raisebox{-.5pt} {\footnotesize 1}}}}
	\put(7.85, 8.95){\color{Gray}\raisebox{.5pt}{\textcircled{\raisebox{-.5pt} {\footnotesize 2}}}}
	\put(13.35, 6.75){\color{Gray}\raisebox{.5pt}{\textcircled{\raisebox{-.5pt} {\footnotesize 3}}}}
	\put(4.4, 5.65){\color{Gray}\raisebox{.5pt}{\textcircled{\raisebox{-.5pt} {\footnotesize 4}}}}
	\put(9.35, 5.65){\color{Gray}\raisebox{.5pt}{\textcircled{\raisebox{-.5pt} {\footnotesize 5}}}}
	\put(5.6, 5.15){\color{Gray}\raisebox{.5pt}{\textcircled{\raisebox{-.5pt} {\footnotesize 6}}}}
	\put(.85, 4.2){\color{blue}\raisebox{.5pt}{\textcircled{\raisebox{-.5pt} {\footnotesize 7}}}}
	\put(1.85, 3.15){\color{blue}\raisebox{.5pt}{\textcircled{\raisebox{-.5pt} {\footnotesize 8}}}}
	\put(3.1, 0.9){\color{blue}\raisebox{.5pt}{\textcircled{\raisebox{-.5pt} {\footnotesize 9}}}}
	\put(6.85, 0.65){\color{blue}\raisebox{.5pt}{\textcircled{\raisebox{+.2pt} {\tiny 10}}}}
	\put(0.85,3.1){\color{purple}$y_0^6$}
	\put(0.85,6.3){\color{purple}$y_1^6$}
	\put(0.85,9.2){\color{purple}$y_2^6$}
	\put(0.79,3.31){\color{purple}\circle*{0.08}}
	\put(0.79,6.58){\color{purple}\circle*{0.08}}
	\put(0.79,9.58){\color{purple}\circle*{0.08}}
\end{picture}
}
\caption{\label{fig:pre} The recursive construction of $\ep$ : here, the construction of $E_7$ from $E_6$.}
\end{center}
\end{figure}
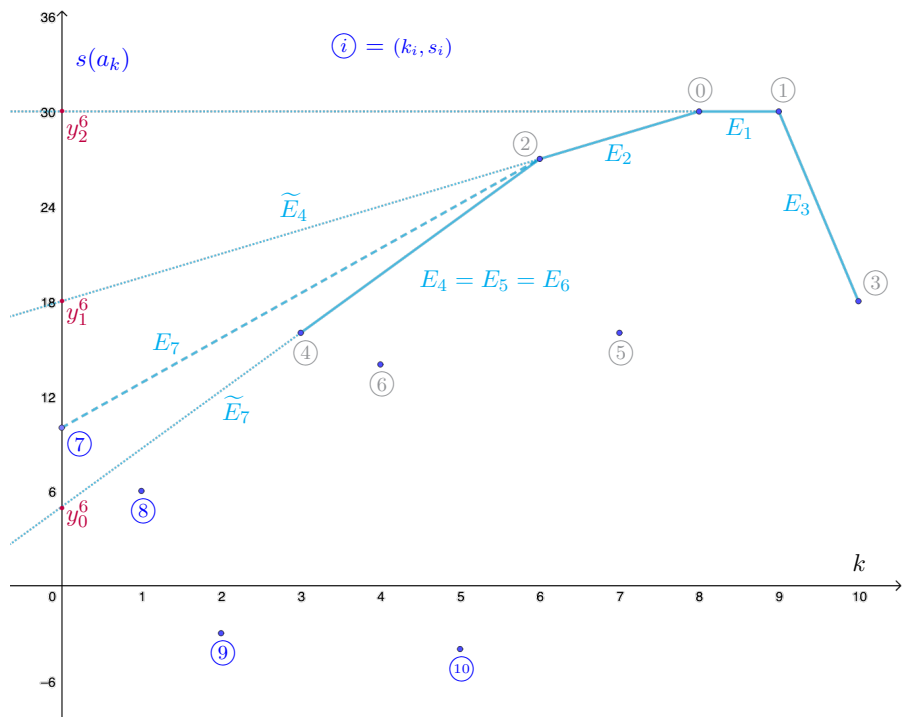

\proof
It is immediate that $\ell_0=r_0=k_0$, $E_0(k_0)=s_0$ and that the graph of $E_1$ is the segment $[(k_0, s_0), (k_1, s_1)]$.
Assume now that $E_{n}$ is constructed. Then $s_{n+1} \leq \min E_n$ and $k_{n+1} \notin \{\ell_n,r_n\}$.
If $k_{n+1} \in (\ell_n,r_n)$ we let $E_{n+1}\defequal E_n$
(on Figure~\ref{fig:pre}, this is the case for $E_5=E_4$ and $E_6=E_5$).
Otherwise, by symmetry, we may assume that $k_{n+1} < \ell_n$ and we build $E_{n+1}$ as a left extension of $E_n$
(alternatively, if $k_{n+1} > r_n$, then we would build  $E_{n+1}$ as a right extension of $E_n$ in a similar way).
Let $\widetilde E_{n+1}$ be the extension of $E_n$ to $[k_{n+1},r_n]$ with the property that $\widetilde E_{n+1}$ is affine on
the interval $\big[k_{n+1}, \ell_n+1/2 \big]$.

If $s_{n+1} \leq \widetilde E_{n+1}(k_{n+1})$ (on Figure~\ref{fig:pre}, this is the case for $n+1=4$), we define $E_{n+1}$ to
coincide with $E_n$ on $[\ell_n, r_n]$, $E_{n+1}(k_{n+1})\defequal s_{n+1}$ and be affine on $[k_{n+1}, \ell_n]$. 

If $s_{n+1} > \widetilde E_{n+1}(k_{n+1})$ (on Figure~\ref{fig:pre}, this is the case for $n+1=7$),
the previous construction would not be concave. Therefore, we search to restrict $E_n$ to some
interval $[\ell_n', r_n]$, $\ell_n' > \ell_n$ and to add a new segment that includes $s_{n+1}$,
such that $E_{n+1}$ is concave.
We remove the segments whose slopes are smaller than the slope of the new segment
(on Figure~\ref{fig:pre}, we drop the segment of $E_6$ that is below the dashed line in $E_7$).

Precisely, let $\Sigma^n_0, \ldots , \Sigma^n_s$ be the lines containing a segment in the graph of  $E_n$,
enumerated from left to right.
As $E_n$ is concave, the $y$-coordinate of $\Sigma^n_j \cap \{x = k_{n+1}\}$, denoted by $y_j^n$,
is increasing in $j$. A binary search finds $j$ such that $s_{n+1} \in (y_{j-1}^n, y_j^n]$ in time $\log_2 n$.
The value $\ell_n'$ is the leftmost abscissa $x$ such that $(x,E_n(x))\in \Sigma_j^n$.
Then $E_{n+1}$ is defined as affine on $[k_{n+1},\ell_n']$ with $E_{n+1}(k_{n+1})=s_{n+1}$ and
$E_n$ coincides with $E_n$ on $[\ell_n',r_n]$.

As $(s_n)_{n\in\ii{0}{d}}$ is decreasing, zero coefficients are sorted last and are only treated
when the graph is already complete. Indeed, as we have assumed that $a_0 \neq 0$ and
that $a_d \neq 0$, if $s_{n+1} = - \infty$ then $k_{n+1} \in (\ell_n, r_n)$ and in that case we set $E_d = E_n$.
\qed

\bigskip
A final parsing of the list of $s(a_k)$ is performed to mark the indices $k$ such thats
\begin{equation}\label{eq:threshold}
s(a_k) \geq \ep(k) - p - s(d) - \buff \,.
\end{equation}
We denote by $G_p \subseteq \ii{0}{d}$ the set of these \emph{good} indices and by
$B_p\defequal \ii{0}{d} \setminus G_p$ the set of \emph{bad} indices.
This step has a linear time complexity.
In subsequent evaluations, only the monomials $(a_k z^k)_{k\in G_p}$ are kept. In what follows,
we show that those associated with~$B_p$ cannot influence the first $p$ bits of the result. The set $G_p$ will be thinned even more during
the evaluation phase, depending on $|z|$.

\bigskip
Let us emphazise that the complexity of the preconditioning does \emph{not} depend on the precision $p$.
If the coefficients $a_k$ are provided in machine floating-point numbers, obtaining $s(a_k)$ is performed in constant
time using hardware-accelerated functions. In the case of an arbitrary precision $p$, the value
$s(a_k)$ is already computed and stored in the number format and there is nothing to do.
All computations for the preconditioning phase can thus be performed with machine floating-point numbers.

\begin{remark}\label{rem:partialpreproc}
Let us mention a slight variant of our algorithm, which is based on the fact that
the lines 2 and 3 of the algorithm $\A_p$ are independent of the value of $p$.
For certain applications, one could split the preconditioning in two parts.
The computation of the concave map~$\ep$ could be done during the compilation (if $P$ is known in advance)
or at early runtime without any knowledge of $p$ (if a low-precision version of $P$ is availlable).
Once the precision $p$ is known, one will finish the preconditioning (\ie determine the set $G_p$, line 4 of $\A_p$) in time $O(d)$.
Subsequent evaluations of $P$ will be performed as before, using only lines~5-11 of~$\A_p$.
\end{remark}

\subsection{Analysis of the evaluation phase}
\label{sect:eval}

To compute $k_\lambda\defequal \argmax\big(\ep + \lambda \mathrm{Id}\big)$, observe that $\ep + \lambda \mathrm{Id}$ is concave.
That is, its derivative (in our case the slope of the segments from some point~$(k, s(a_k) + \lambda k)$
to the next one~$(k',s(a_{k'}) + \lambda k')$) is decreasing.
Therefore, a binary search finds $k_\lambda$ in $\log_2 d$ operations.
The maximum value is
\begin{equation}\label{eq:defN}
N_\lambda\defequal \ep(k_\lambda) + \lambda k_\lambda \,.
\end{equation}

Next, as $\ep + \lambda \mathrm{Id}$ has at most two monotone branches (separated by $k_\lambda$)
we can perform a binary search on each of them to find respectively the two
indices~$\ell<k_\lambda$ and~$r>k_\lambda$ such that $\ii{\ell}{r}$ is the largest integer interval that satisfies
\begin{equation}\label{def:LRN}
\ii{\ell}{r} \subset \SetDef{k \in \ii{0}{d}}{\ep(k) + \lambda k \geq \max\big( \ep + \lambda \mathrm{Id} \big) - p - s(d) - \buff}.
\end{equation}
Each of these searches costs at most $\log_2 d$ operations.
Therefore lines 5-10 of $\A_p$ cost $2+3\log_2 d$ operations, which is the first part of~\eqref{eq:compAlgo} in Theorem~\ref{thm:algo}.

\bigskip
Let us now focus on the complexity analysis of the last step (line~11) of the algorithm~$\A_p$.
Formula \eqref{equ:cos} reads $L(f-(\tan\theta) \operatorname{Id},\delta,0)=L(f,\delta,\theta)\cos\theta$; 
joined with~\eqref{def:LRN}, it implies
\[
\frac{r - \ell}{d} \leq L(f, \delta, \theta ) \cos \theta
\]
where $f$ and $\delta$ are defined by~\eqref{eq:def_f_from_Ep} and $\lambda = -\tan\theta$ (the minus sign reflects that
positive slopes correspond to evaluation points $z$ such that $|z|<1$).
The metric on the Riemann sphere $\CC$ that is associated with a uniform probability measure is given by
\[
g_{\CC} = \frac{dx^2+dy^2}{\pi (1+x^2+y^2)^2} \,\cdotp
\]
The corresponding volume element is
\[
\sqrt{|g_{\CC}|} \, dx \wedge dy = \frac{dx \wedge dy}{\pi (1+x^2+y^2)^2} \,\cdotp
\]
For a radial function and $r^2=x^2+y^2$, the volume element becomes
\[
\frac{2 r dr}{(1+r^2)^2} \quad\text{on}\quad[0,\infty)
\]
and with the subsequent change of variable $\log_2 r = - \tan\theta$, it turns into
\[
-\frac{2 \ln 2}{\cos^{2}\theta} \frac{4^{\tan\theta}}{(1+4^{\tan\theta})^2} \, d\theta \quad\text{on}\quad
\left(-\frac{\pi}{2},\frac{\pi}{2}\right).
\]
Therefore, the average number of monomials that are required to evaluate a polynomial of degree $d$ with our algorithm, when
the point $z=x+i y$ is chosen uniformly on the Riemann sphere $\CC$, is bounded from above by
\begin{equation}\label{eq:omegatheta}
\avg_{\CC} \left(\# G_p \cap \ii{\ell}{r}\right)
\leq \avg_{\CC} \left(r-\ell  + 1 \right)
\leq 1+ d \int_{ - \frac \pi 2}^{\frac \pi 2} L(f, \delta, \theta ) \, \omega(\theta)  \, d\theta \,,
\end{equation}
with (note that $\omega$ is even):
\[
\omega(\theta) =  \frac{2 \ln 2}{\cos\theta} \frac{4^{\tan\theta}}{\left(1+4^{\tan\theta}\right)^2} \,\cdotp
\]
Therefore, Theorem~\ref{thm:conc2} implies
\begin{equation}\label{eq:avgestim}
\avg_{\CC} \left(\# G_p \cap \ii{\ell}{r} \right) \leq 1+C_\omega d \sqrt{\delta} =  1+C_\omega \sqrt{d(p +s(d) + \buff)}\,,
\end{equation}
with $C_\omega < 1.9046$ and whose exact numerical value is given by~\eqref{main_estim_ter_techsupport}.

\medskip
\begin{figure}[H]
\begin{center}
\includegraphics[width=14cm]{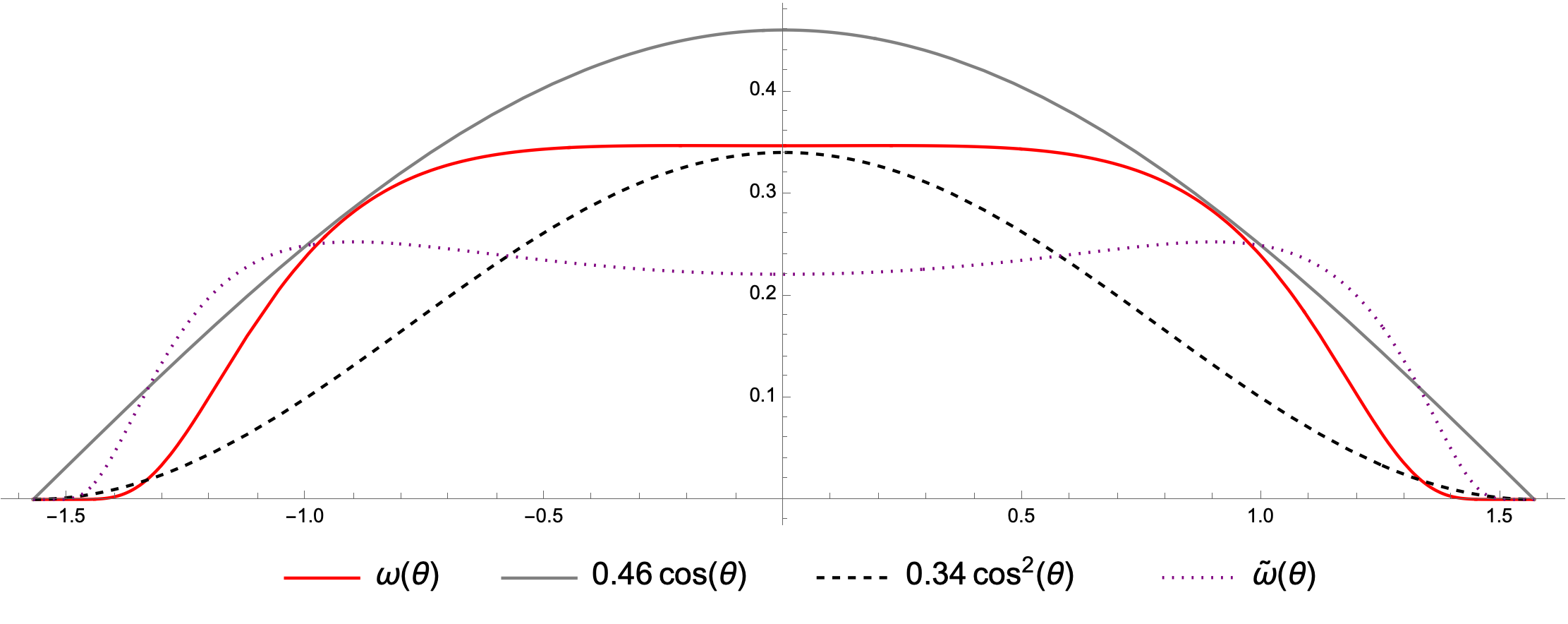}
\caption{\label{fig:avgcoscos2}
Graph of the weight $\omega(\theta)$ from~\eqref{eq:omegatheta} and its comparison with
$0.46 \cos\theta$ (gray) and $0.34 \cos^2\theta$ (dashed). This comparison justifies our interest for
those particular weights in Theorem~\ref{thm:conc}. The intermediary range $|\theta|<1.34$ corresponds roughly to $1/20<|z|<20$.
The weight~$\widetilde\omega(\theta)$ from~\eqref{eq:omegathetatilde} is for the real valued case.}
\end{center}
\end{figure}

\begin{remark}
Note that using the looser estimate $\omega(\theta)<0.46\cos\theta$ (see Figure~\ref{fig:avgcoscos2}) and
the numerical constant of Theorem~\ref{thm:conc} overshoots the value of~$C_\omega$ by~$42\%$.
\end{remark}

To conclude the evaluation of $P(z)$, we compute $z^\ell$ in $2\log_2 \ell$
steps and then use H\"orner's method to compute $Q_\lambda(z)$.
The average arithmetic complexity of line 11 of the algorithm~$\A_p$ is thus bounded by
\begin{equation}\label{eq:complexityAlgo}
2\log_2 d + C_\omega \sqrt{d(p +\log_2 d + \buffpo)}\,,
\end{equation}
while the bit/time complexity is $M(p)$ times larger.
Putting \eqref{eq:avgestim} and \eqref{eq:complexityAlgo} together gives the last part of~\eqref{eq:compAlgo} in Theorem~\ref{thm:algo}.

\bigskip
In the case of a real polynomial evaluated along the real line, the uniform measure on the circle~$\RR=\R\cup\{\infty\}$
obtained by stereographic projection is
\[
\frac{dx}{\pi(1+x^2)} \,\cdotp
\]
With the change of variable $\log_2 |x| = \tan\theta$, one gets
\begin{equation}\label{eq:omegathetatilde}
\avg_{\RR} \left(r-\ell\right) \leq d \int_{ - \frac \pi 2}^{\frac \pi 2} L(f, \delta, \theta ) \, \widetilde{\omega}(\theta)\, d\theta
\quad\text{with}\quad
\widetilde\omega(\theta) = \frac{2 \ln 2}{\pi \cos\theta}  \frac{2^{\tan\theta}}{1+4^{\tan\theta}} \,\cdotp
\end{equation}
Using~\eqref{main_estim_ter_techsupport} again provides
\begin{equation}
\avg_{\RR} \left(\# G_p \cap \ii{\ell}{r} \right) \leq 1+C_{\widetilde\omega} \sqrt{d(p +s(d) + \buff)}
\quad\text{with}\quad
C_{\widetilde\omega} <1.7673 \,.
\end{equation}
The average complexity in the real-valued case is given by~\eqref{eq:complexityAlgo} with $C_{\widetilde\omega}$
instead of $C_{\omega}$.

\begin{remark}\label{rmk:case_disk}
One can easily adapt the computation to the case of a complex polynomial evaluated at
$z$ uniformly distributed over the unit disk $D(0,1)$. The weight becomes
\[
\omega_{D}(\theta) = \frac{2 \ln 2}{\cos\theta} \frac{1}{4^{\tan\theta} } \qquad\text{on}\qquad \left[0,\frac{\pi}{2}\right).
\]
The asymmetry of $\omega_D$ implies that one must restrict the integral of Lemma~\ref{lemma:cgt-var45}
to~$y\geq0$, \ie[~]$x_2\geq x_1$. Theorem~\ref{thm:conc2} remains valid with
\[
C_{\omega_D} = 
\sqrt{2} \|\omega_D\|_{L^\infty\left(0,\frac{\pi}{4}\right)} + 2 \int_{-\infty}^{-\frac{\sqrt{2}}{4}}
\frac{\omega_D\left(\arctan(\frac{1}{2}-\sqrt{2}x)\right)}{\sqrt{1+2\left(x-\frac{\sqrt 2}{4} \right)^2}} dx
= \frac{1+8\ln 2}{2\sqrt{2}}\,\cdotp
\]
The numerical value satisfies $C_{\omega_D} < 2.3141$ and one can claim
\begin{equation}\label{eq:case_disk}
\avg_{D(0,1)} \left(\# G_p \cap  \ii{\ell}{r} \right) \leq 1+C_{\omega_D} \sqrt{d(p +s(d) + \buff)} \,.
\end{equation}
\end{remark}

\subsection{Example that (almost) saturates the upper bound on complexity.}
\label{par:optimal}

Because of the fast decay of $\omega(\theta)$ as $\theta\to\pm\pi/2$, it is not possible to reuse directly the lower
bound obtained in Theorem~\ref{thm:conc} for the weight $\cos^2\theta$.
However, the examples of Section~\ref{par:thmconcLower} can be adapted easily to saturate the complexity
of the algorithm~\A{}.

\bigskip
Inspired by the second example,
let us consider a polynomial $P$ whose coefficients have a scale profile that follows a half-circle, for example:
\begin{equation}\label{eq:halfcirclepoly}
P(z) = \sum_{n=0}^d 2^{\sqrt{(n+1)(d+1-n)}} z^n.
\end{equation}
Reasoning as in Section~\ref{sect:eval} and using the maximality of $\ii{\ell}{r}$ in~\eqref{def:LRN}, we get
\[
\frac{r-\ell+2}{d} \geq L(f,\delta,\theta)\cos\theta
\]
for any $z\in\C^\ast$ such that $\log_2 |z| = \tan\theta$ and $f$, $\delta$ defined by~\eqref{eq:def_f_from_Ep}.
The average arithmetic complexity of the evaluation of $P$ when $z\in\CC$ (resp. $z\in\RR$) is bounded from below by
\[
\avg_{\CC}(r-\ell+1) \geq -1 + d \int_{-\pi/2}^{\pi/2} L(f,\delta,\theta) \omega(\theta) d\theta
\]
or, respectively, the same integral with $\widetilde\omega$ in place of~$\omega$.
One can check easily that $E_P$ defined by~\eqref{def:EP} satisfies
\[
\forall n\in\ii{0}{d},\qquad
\ep(n) \geq E_P(n) = 1+ \left\lfloor \sqrt{(n+1)(d+1-n)} \right\rfloor \geq \sqrt{n(d-n)}
\]
and $\ep(n)\leq \sqrt{n(d-n)} + C\sqrt{d}$, thus
\[
 \sqrt{x(1-x)} \leq f(x)\leq \sqrt{x(1-x)} + \frac{C}{\sqrt{d}} \,\cdotp
\]
As $d\to\infty$, the graph of $f$ converges uniformly to that of $\sqrt{x(1-x)}$, which is concave.

\medskip
With the notations of Example 2 of Section~\ref{par:thmconcLower}, the average complexity is thus
asymptotically bounded from below by
\[
4d\sqrt{\delta}\int_0^{\theta _0}\sqrt{(1-\delta \cos\theta)\cos \theta}    \, \omega(\theta)d\theta
=C_3(\delta) \sqrt{d(p+s(d)+\buff)} \,.
\]
The leading coefficient $C_3(\delta)$ is maximal at $\delta\to0$ \ie $d\to\infty$.
The asymptotic value is $C_3(0)\simeq 1.32178$.
In the real case,  the average complexity is asymptotically bounded from below by $C_4(\delta) \sqrt{d(p+s(d)+\buff)}$
with $C_4(0)\simeq 1.04074$.

\medskip
The theoretical predictions of this section have been confirmed, in practice:
polynomials~\eqref{eq:halfcirclepoly} are the slowest to evaluate (see Section~\ref{sec:bench}
and, in particular, Figure~\ref{fig:compTypeSphLine}).

\bigskip
\begin{figure}[H]
\captionsetup{width=.95\linewidth}
\begin{center}
\includegraphics[width=\textwidth]{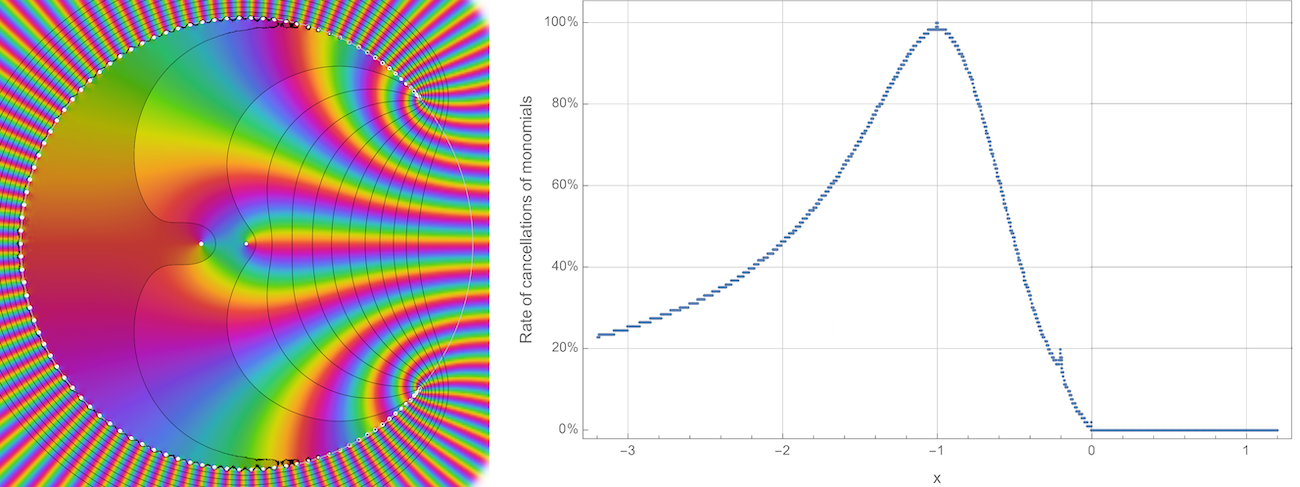}\\[1ex]
\includegraphics[width=\textwidth]{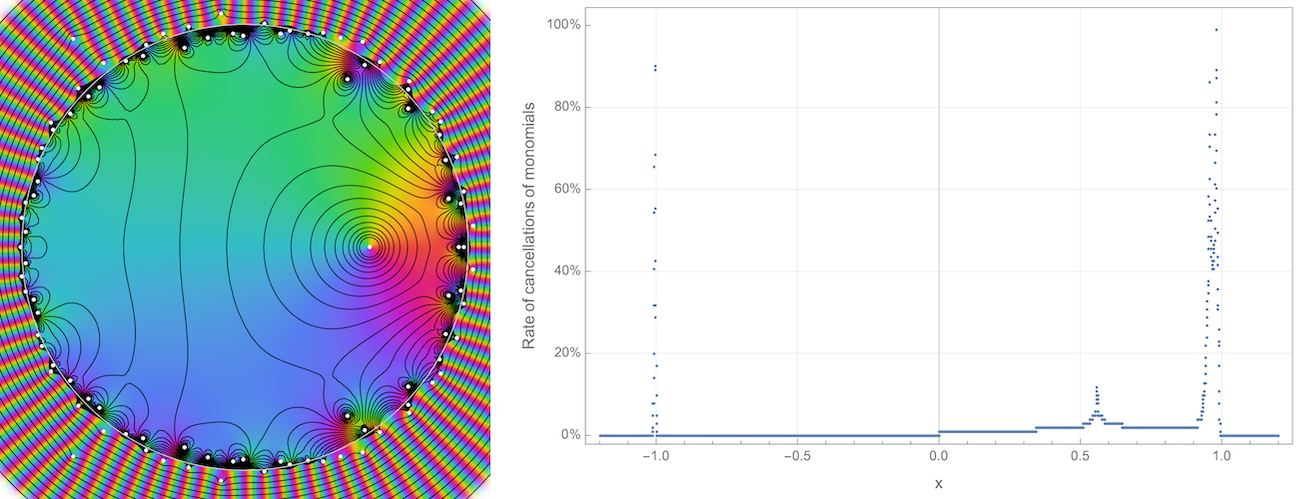}
\caption{\label{fig:HalfCircleBenchmark}
Comparison between a polynomial~\eqref{eq:halfcirclepoly} whose coefficients obey the half-circle law (above)
and a polynomial of the same degree whose coefficients obey a normal law (below). The complex plots (left)
illustrate phase (in color), level lines (gray) and roots (white). The rate of cancelation of monomials is
computed along the real axis (right) and illustrates that half-circle polynomials have extreme cancelations along the
negative real axis, which extend far beyond the immediate vicinity of their roots.
}
\end{center}
\end{figure}

Evaluation benchmarks with polynomials~\eqref{eq:halfcirclepoly} whose coefficients obey the half-circle law
lead to an interesting observation. As the degree increases, these polynomials appear to be extremely difficult
to evaluate precisely along the real line. At degree 1000, only half  of the computations with 600 bits along the real line are fully trustworthy;
about 20\% of the computations lead to at least 200 bits being identified as uncertain by our FPE algorithm. At degree 33\,113,
half of the 600-bit computations report that no bit is trustworthy. This exceptional situation piqued our interest because the
roots of these polynomials appear to concentrate mostly along a sub-arc of the unit circle, which means that evaluations
along the real line are usually not in the direct vicinity of a root.

A deeper analysis (see Figure~\ref{fig:HalfCircleBenchmark})
suggests that polynomials in this family have an extremely high \textit{cancelation rate} of the monomials, which at a given $z\in\C$
is defined as the proportion of the monomials $(a_n z^n)_{n\in\ii{0}{d}}$ such that $|a_nz^n|>|P(z)|$. Of course, along the positive half of the
real axis, no cancelations can occur because all the coefficients are positive.
In comparison, the cancelation rates for other families of polynomials seem to spike in much narrower regions
of the complex plane. This observation is consistent with our statements on the complexity of the \A{} algorithm
and \emph{draws a parallel} between slow \A{} evaluations and precision loss.

\section{Error analysis and proof of Theorem~\ref{thm:equivAlgo}}
\label{sect:err}

In this section, we prove the correctness of the algorithm $\A_p$, \ie Theorem~\ref{thm:equivAlgo}.
We adopt the notations from Section~\ref{par:algo} and show that 
\[ 
P(z) \approx_{p -c}Q_\lambda(z) \,,
\]
where $c\geq0$ is defined by~\eqref{eq:def_prec_loss}.
In the light of the property \eqref{eq:eqp_crit}, it is enough to show instead that 
\begin{equation}\label{eq:target}
\Big | P(z) - \sum_{k \in G_p \cap \ii{\ell}{r}} a_k z^k \Big | \leq 2^{-(p-c)-2} \left| P(z) \right |.
\end{equation}

\smallskip
Let us assume first that $|P(z)| \geq |a_{k_\lambda} z^{k_\lambda}|$ where $k_\lambda\defequal \argmax\big(\ep + \lambda \mathrm{Id}\big)$,
\ie (roughly speaking)
that there is no cancelation of leading bits. Using~\eqref{eq:defLambda} and~\eqref{equ:dsz}, one gets:
\[
|P(z)| \geq 
|a_{k_\lambda} z^{k_\lambda}| = |a_{k_\lambda}| 2^{\lambda k_\lambda}
\geq 2^{s(a_{k_\lambda}) + \lambda k_\lambda-1}  = 2^{N_\lambda-1} \,,
\]
with $N_\lambda$ defined by~\eqref{eq:defN}.
Thanks respectively to the definitions \eqref{def:LRN} and~\eqref{eq:threshold}, one has
\[
s(a_k) + \lambda k \leq \begin{cases}
\ep(k) + \lambda k \leq N_\lambda - p - s(d) - \buffpo & \text{if } k\notin\ii{\ell}{r} \,,\\
\ep(k) + \lambda k - p - s(d) - \buffpo & \text{if } k\in \ii{\ell}{r}\backslash G_p \,.
\end{cases}
\]
In both cases, we get $s(a_k) + \lambda k \leq N_\lambda - p - s(d) - \buffpo$
for $k\in D_p(\lambda)\defequal  \ii{0}{d} \setminus (\ii{\ell}{r} \cap G_p)$,
\ie for each dropout monomial. Using~\eqref{eq:defLambda} a second time, we get any $k\in D_p(\lambda)$:
\[
|a_k z^k| = |a_k| 2^{\lambda k} <  2^{s(a_k) + \lambda k} 
\leq 2^{N_\lambda-p-s(d)-\buffpo} \leq 2^{-p - s(d) - \buff} |P(z)| \,.
\]
We may now estimate $R(z)\defequal  P(z)-Q_\lambda(z)$ using $\#  D_p(\lambda) \leq d < 2^{s(d)}$ or, equivalently, using~\eqref{equ:bs}:
\begin{equation}
\label{equ:err}
 |R(z)| \leq \sum_{k \in  D_p(\lambda)} \big| a_k z^k \big| < d \, 2^{-p - s(d) - \buff} |P(z)| < 2^{-p-\buff} |P(z)| \,.
\end{equation}
In this case,~\eqref{eq:target} holds with $c=0$ (with a margin of 1 bit) and $P(z) \approx_{p} Q_\lambda(z)$.
The lazy algorithm is therefore essentially exact when no leading bits get canceled.

\medskip
In the case where $|P(z)| < |a_{k_\lambda} z^{k_\lambda}|$, some of the most significant bits cancel each other.
More precisely, let us define $c\in\N$ by
\begin{equation}\label{eq:slice}
c=s(a_{k_\lambda} z^{k_\lambda}) - s(P(z)) = \left\lfloor\log_2 |a_{k_\lambda} z^{k_\lambda}|\right\rfloor - \big\lfloor \log_2 | P(z)| \big\rfloor\,.
\end{equation}
According to Remark~\ref{rmk:cancel}, exactly $c$ leading bits have been canceled while computing~$P(z)$.
In this case, as we carry all computations with a fixed
precision of~$p$ bits, only the first~$p-c$ bits of the result are meaningful (plus one implicit leader).
One still has
\[
|P(z) - Q_\lambda(z)| =  |R(z)| < d\, 2^{N_\lambda-p-s(d)-\buffpo} < 2^{N_\lambda-p-\buffpo} \,.
\]
On the  other hand, using~\eqref{equ:shear}, one has now $ |P(z)| \geq 2^{s(a_{k_\lambda}z^{k_\lambda})-c-1}\geq 2^{N_\lambda-c-2}$ thus
\[
|R(z)|  < 2^{-(p-c)-\buffmo} |P(z)|
\]
and~\eqref{eq:target} holds in this case too.

\section{Applications}
\label{sect:applic}

In this section, we expose a few possible applications of the \A{} algorithm, at both the theoretical and practical levels.

\subsection{Parsimonious representation of polynomials}\label{par:parsimonious}

At a theoretical level, Theorem~\ref{thm:equivAlgo} states the existence of a \emph{parsimonious}
representation of any polynomial. This reduction can be computed algorithmically, is valid on any given
annulus of $\C$ and guarantees a fixed arbitrary bound on the relative error.

\bigskip
For example, let us consider the Chebyshev polynomials $T_n(\cos x)=\cos(nx)$.
They are the archetype of evaluations with extreme cancelations because each $T_n$ maps the interval~$[-1,1]$ onto
itself while the coefficients~$(a_{n,j})_{0\leq j\leq n}$ of~$T_n$ grow exponentially
(namely $\max_j |a_{n,j}|\lesssim 2^{1.26n}$, as indicated by the maximum point of Figure~\ref{fig:chebyshev}, left).
The scale profile of the coefficients of~$T_n$ renormalized with~\eqref{eq:def_f_from_Ep},
\ie $s(a_{n,j})/n$ appears to converge towards a fixed profile (red curve on Figure~\ref{fig:chebyshev}).
Taking this fact for granted,
Theorem~\ref{thm:equivAlgo} predicts the degree~$q(n,\alpha)$ such that the reduced polynomial
\[
Q_{n,\alpha}(x) \defequal  T_n(x) \,\operatorname{mod}\, x^{q(n,\alpha)}
\]
provides an accurate approximation of~$T_n(x)$ over the interval~$[-\alpha,\alpha]$.
For example, for $\alpha=0.3\simeq 2^{-1.74}$, Figure~\ref{fig:chebyshev} shows
that the maximum of $|a_{n,k} 0.3^k|$ is achieved for $k_n\simeq 0.285n$ and that $s(a_{n,k_n}) \simeq 0.9n$.
For $n=200$, $\max |a_{n,k} 0.3^k| \simeq 2^{200(0.9-1.73\times 0.285)} \simeq 2^{81}$, which means that $c\simeq 81$ leading bits will be lost in the computation
of $T_{200}(x)$ when $x\simeq 0.3$. Theorem~\ref{thm:equivAlgo} with $p=85$ ensures that $T_{200}(0.3)$ can be computed with at
least~3 significant bits if we keep the coefficients above the dashed line on Figure~\ref{fig:chebyshev} (offset $\delta\simeq 96$ bits),
\ie if we drop the last 25\% of the coefficients.
In general, this proportion is independent of $n$ and we can claim that $q(n,\alpha)/n$ too is asymptotically independent of $n$.
A direct proof of this result (without Theorem~\ref{thm:equivAlgo}) does not seem obvious.

\begin{figure}[H]
\begin{center}
\includegraphics[width=.97\textwidth]{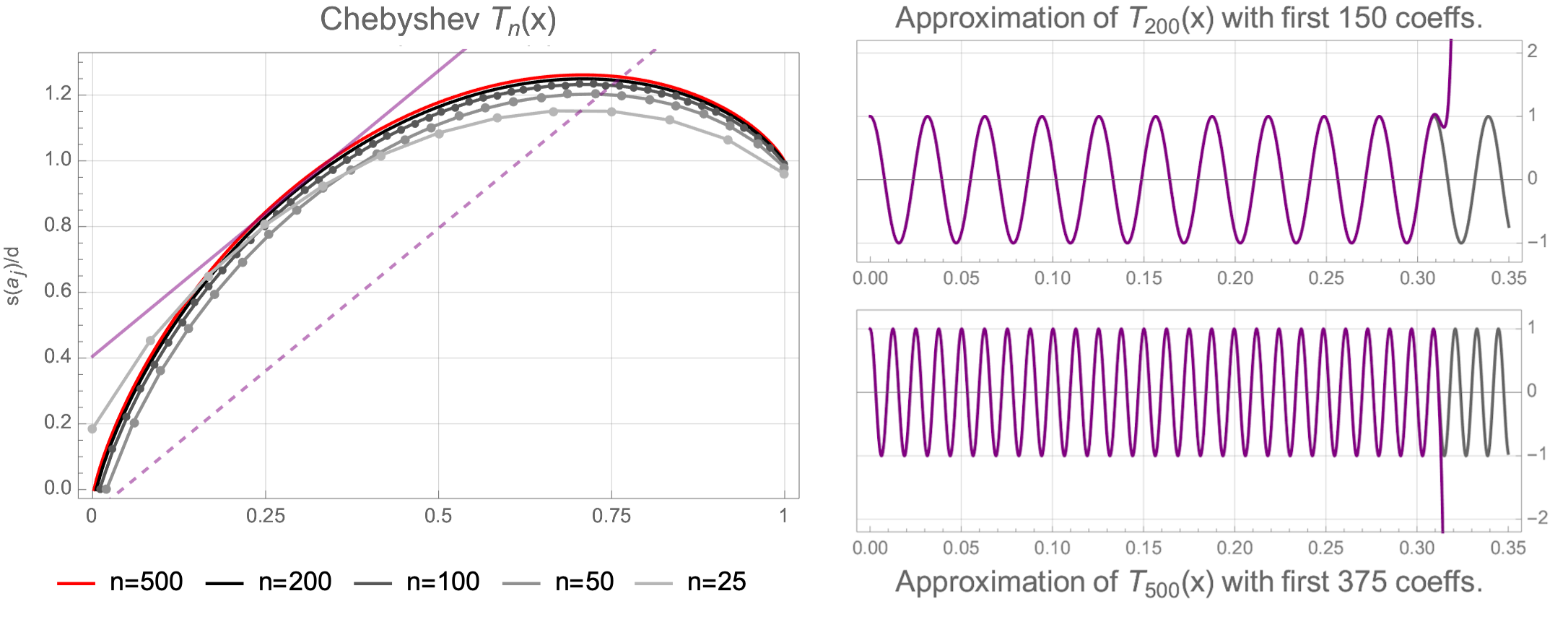}
\caption{\label{fig:chebyshev}
Scale of the coefficients of Chebyshev polynomial $T_n(\cos x)=\cos(nx)$, normalized with~\eqref{eq:def_f_from_Ep}.
The slope of the purple lines correspond to $|z|=0.3$.
As the renormalized profile of the coefficients is asymptotically independent of $n$ (red curve),
the reduction of $T_n$ to the first 75\% of the coefficients (above dashed line) is accurate on [-0.3,0.3] for any large $n$ (right).
}
\end{center}
\end{figure}

\begin{figure}[H]
\begin{center}
\includegraphics[width=.97\textwidth]{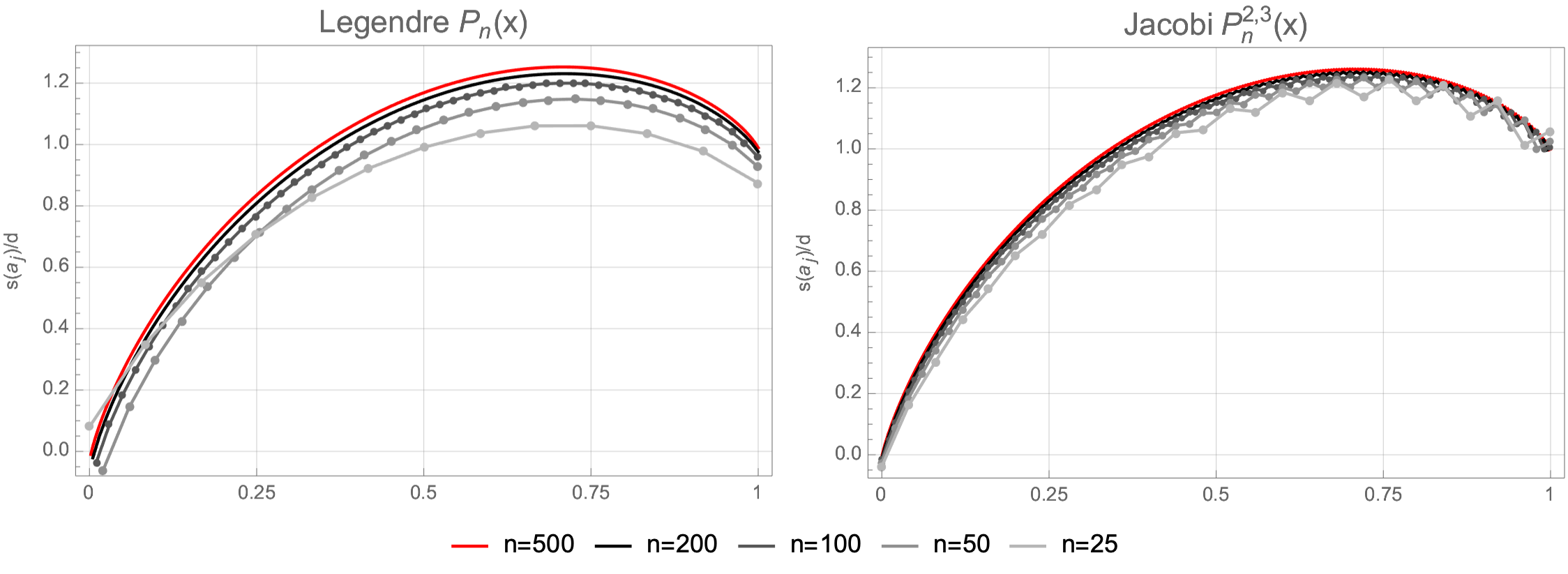}
\caption{\label{fig:jacobi}
Scale of the coefficients of Legendre polynomials $P_n$ (left)
and of a generic example of Jacobi polynomials~$P_n^{(\alpha,\beta)}$ (right),
normalized according to~\eqref{eq:def_f_from_Ep}.
Numerically, the renormalized profiles appear, as in Figure~\ref{fig:chebyshev}, asymptotically independent of $n$.}
\end{center}
\end{figure}

\medskip
The Jacobi polynomials~$P_n^{(\alpha,\beta)}$ and, in particular, the Legendre polynomials~$P_n$
enjoy a similar property (Figure~\ref{fig:jacobi}), which may be of interest for mathematical physics.

\medskip
The engineering pressure towards better onboard electronics, using microcontrolers and field-programmable gate arrays,
requires that some non-linear functions be computed quickly, often in reduced precision (\eg 32, 16 and even 8 bits),
with hardware-specific optimizations.
This problem has revived interest$^\ast$\footnote{$^\ast$ The authors thank \href{https://www.lri.fr/~falcou/}{Joel Falcou} (LRI, Université Paris Saclay) for pointing out this application.} in the Remez algorithm on the polynomial approximation of an arbitrary function that minimizes the $L^\infty$-error, \ie minimax approximation~\cite{REM1934}, \cite{REMLibrary}.
For example, when dealing with periodic functions, engineers  are interested in bypassing a costly reduction mod~$\pi/2$
if a suitable interpolator provides accurate values on the natural range of angles for their problem.

\bigskip

For a given range of evaluation points, the algorithm~\A{} will either provide a further reduction of the number of coefficients
needed at a given precision, or conversely, it will show that no further reduction is possible (see~\eg[~]Figure~\ref{fig:chebyshev}).
In both cases, such a result provides theoretical backing for the implementation choices. The \texttt{-analyse} task in
our implementation~\cite{FPELib} (see Section~\ref{sect:code}) provides a rudimentary tool to perform this analysis.

\medskip
In practice, the level of parsimony achieved by the \A{} algorithm can be remarkably high.
For example, Figure~\ref{fig:decimationCircle1000Sph}
illustrates the proportion of monomials that are kept in $Q_\lambda(z)$ and therefore lead the value of $P(z)$.

\begin{figure}[H]
\captionsetup{width=.95\linewidth}
\begin{center}
\includegraphics[width=\textwidth]{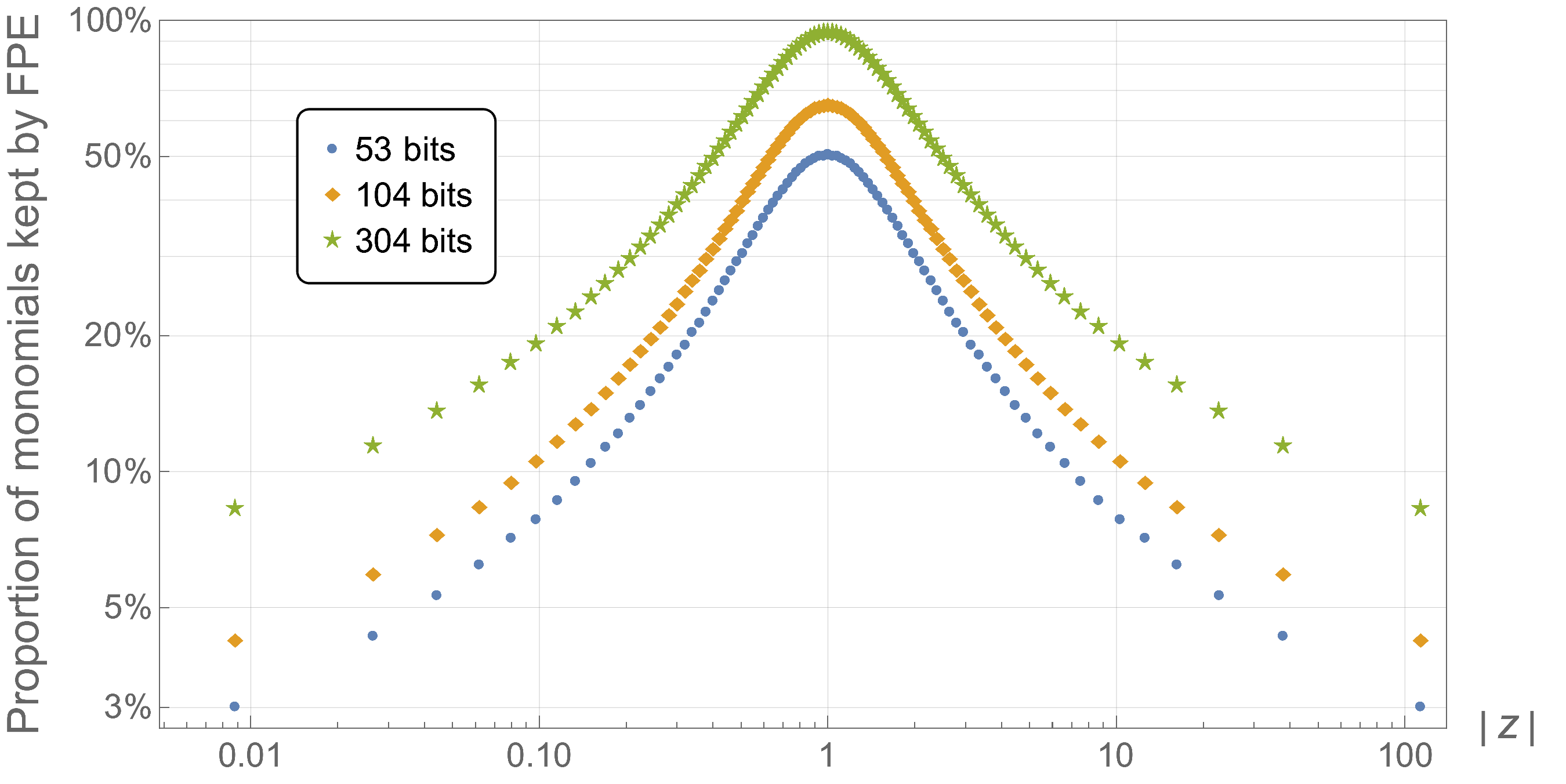}
\caption{\label{fig:decimationCircle1000Sph}
Proportion of monomials kept by the \A{} algorithm
in the evaluation of a half-circle polynomial~\eqref{eq:halfcirclepoly} of degree 1000 for various precisions.
Observe how most values are determined by very few leading monomials.
}
\end{center}
\end{figure}

\subsection{Application to root finding with Newton's method}

On the practical side, Theorem~\ref{thm:algo} ensures the following two benefits.

\medskip
Firstly, for a given allotment of computation time, one can perform $k_H$ evaluations of a 
certain polynomial with H\"orner's
method, or $k_\A$ evaluations with the \A{} algorithm.
For a given precision $p$ and large $d$, the \emph{asymptotic} ratio can be extracted from~\eqref{eq:compAlgo},
provided that the set of evaluation points is statistically diverse. One gets:
\begin{equation}\label{eq:asymptGain}
\frac{k_{\A}}{k_{H}} \simeq 
\frac{\sqrt{d}} {1.9046 \sqrt{p+\log_2 d+\buffpo}}\gg1 \,.
\end{equation}
The corresponding \emph{asymptotic gain factor} is illustrated on Figure~\ref{fig:HornerFPETheory}.

\medskip
Secondly, using Remark~\ref{rmk:cancel}, it is also very easy to \emph{detect cancelations} of leading bits, which means that
the \A{} algorithm allows not only faster computations, but also provides a hint at runtime on the precision that should
be used to achieve a certain level of accuracy (typically, the desired accuracy plus the number of canceled bits).
Running error bounds (\ie estimates of the absolute error committed during the evaluation process)
are also available for H\"orner~\cite{HIG02}; however they do not directly indicate the number of leading bits that where canceled,
contrary to \A{}, which can compare the scale of the largest monomial to the final result at no extra cost (see Figure~\ref{fig:estimPrecSphere}).

\bigskip
\begin{figure}[H]
\captionsetup{width=.95\linewidth}
\begin{center}
\includegraphics[width=.49\textwidth]{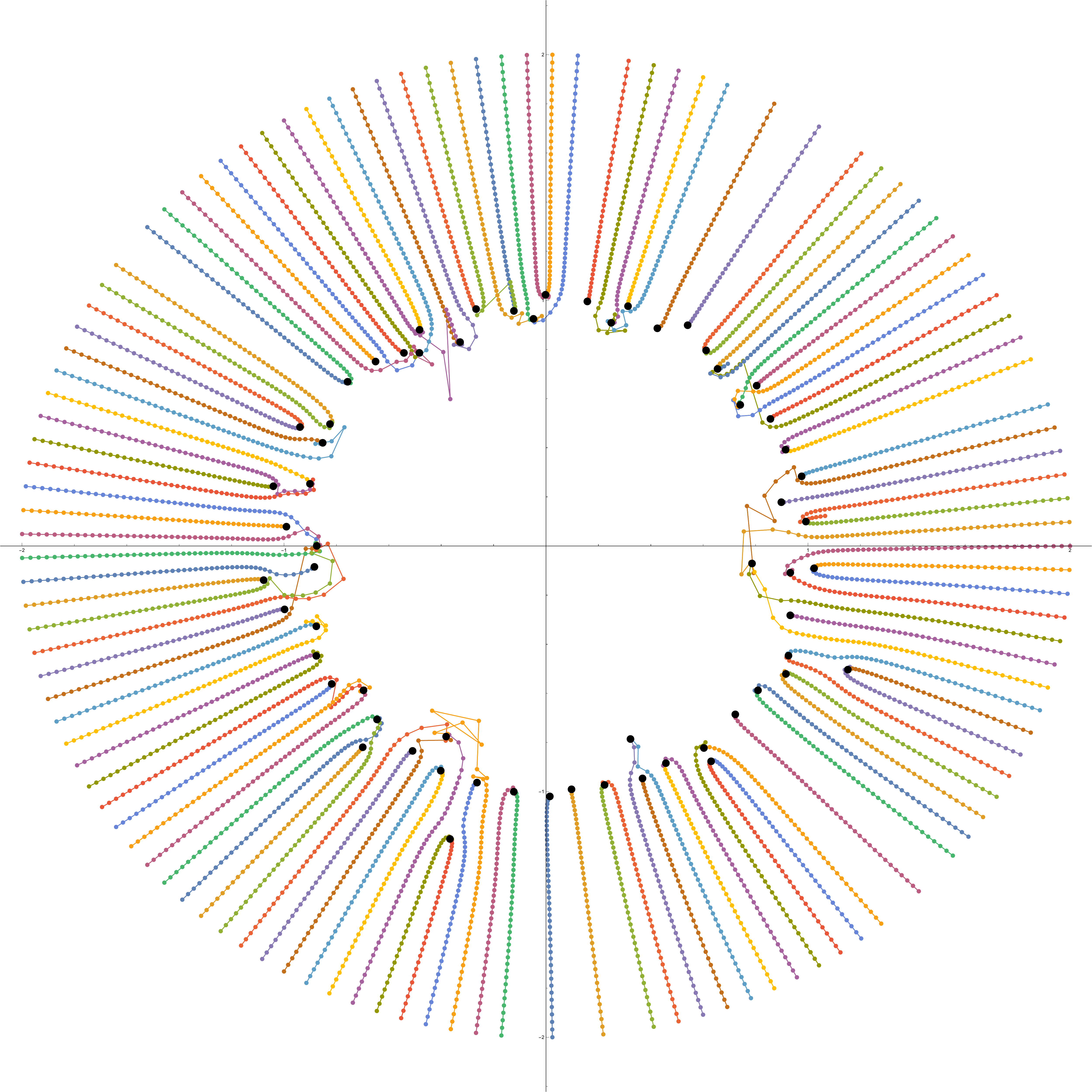}\hfill\includegraphics[width=.495\textwidth]{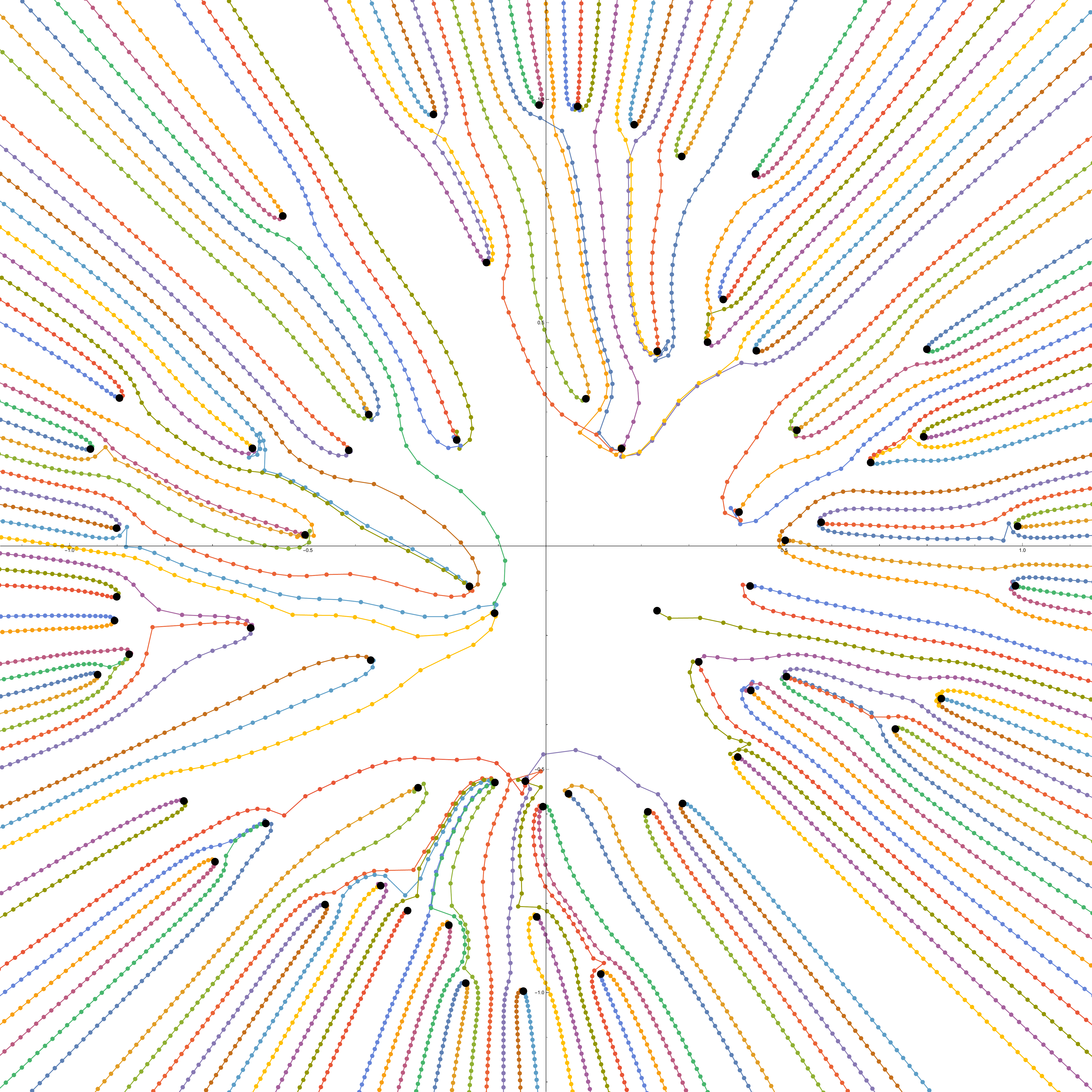}
\caption{\label{fig:newton}
Examples of root-finding with Newton's method for polynomials of degree 65 with random Gaussian coefficients (left)
or uniformly distributed roots on the disk (right).
The starting points are uniformly distributed on the circle of radius 2 and each step is computed with \texttt{FastPolyEval -evalN}
and a precision set to $p=100$ bits. Only the trajectories that avoid critical points (no far jumps) and that have ultimately converged
are shown.
}
\end{center}
\end{figure}

A typical application that takes advantage of these two properties is finding roots with Newton's method. Given a starting point $z_0\in\C$,
one computes the sequence
\begin{equation}\label{eq:newton}
z_{n+1} = N_P(z_n) \qquad\text{with}\qquad N_P(z) = z-\frac{P(z)}{P'(z)}\,\cdotp
\end{equation}
Almost surely, the sequence will converge towards a root of~$P$;
divergence occurs when $z_0$ is in the Julia set of $N_P$ (which is of Hausdorff dimension $<2$, see \cite{MILNOR}, \cite{CARLESON}).
Costly excursions near $\infty$ occur also if the sequence visits a small neighborhood of a critical point.
Using enough starting points (see the algorithm described in \cite{HSS2001}),
one can compute all the roots of $P$.
We refer the reader to our work~\cite{MV21} for a refinement of~\cite{HSS2001}
that allowed us to split a tera-polynomial,~\ie$\deg P = 2^{40} \simeq 10^{12}$ using a set of carefully chosen starting points
for Newton's method.
Here, we focus on the simpler task of showing the benefits of applying~\A{} to compute~\eqref{eq:newton}
instead of H\"orner's scheme.

\medskip
The first  benefit is that the preconditioning of $P$ and $P'$ can be done \emph{simultaneously}. Indeed,
if~$a_j$ are the coefficients of~$P$, then those of the derivatives satisfy:
\[
s(j a_j) = s(j) + s(a_j) + \{-1,0\} = \lfloor \log_2 j \rfloor + s(a_j) + \{0,1\} \,.
\]
In a first approximation, the maps $\ep$ and $\wideparen{E}_{P'}$ are thus simply
offset from one another by the concave map~$j\mapsto \log_2 j$. 
In practice however, $\wideparen{E}_{P'}$ may have more segments than $\ep$.
If speed is of the essence, one can choose to keep a low-resolution profile $\wideparen{E}_{P'}$ and
increase the safety margin $\delta$ (our implementation choice for the Newton demonstrator).
Alternatively, one could perform a separate preprocessing for $P$ and $P'$.

\medskip
The second and main key point is that the computation of $P(z)/P'(z)$ can be largely improved if one takes
into account the \emph{cancelation of valuation}$^\ast$\footnote{$^\ast$ The valuation of a polynomial
is the lowest degree of its non-zero monomials, \ie the multiplicity of zero as a root.} induced by~\A{}. Precisely, if one assumes that
\[
Q_1(z) = \sum_{k\in \ii{\ell}{r}\cap G_p} a_j z^j
\quad\text{and}\quad
Q_2(z) = \sum_{k\in \ii{\ell'}{r'}\cap G_p'} ja_j z^{j-1}
\]
are the respective $p$-bit reductions of $P(z)$ and $P'(z)$, then
\[
\frac{P(z)}{P'(z)} \approx \frac{Q_1(z)}{Q_2(z)} = \frac{ \sum_{k\in \ii{\ell}{r}\cap G_p} a_j z^{j-m} }{\sum_{k\in \ii{\ell'}{r'}\cap G_p'} ja_j z^{j-1-m}} \,,
\]
where $m=\min\{\ell, \ell'\}$. Taking the simplification of $z^m$ into account improves both the speed and the accuracy.
It is especially important in the early phase of Newton's sequence, where $P(z)$ and $P'(z)$ may still be huge,
which would cause a substantial loss of precision in the computation of the increment, or even an overflow. In~\cite{MV21}, we encounter
examples where neither $P(z)$ nor $P'(z)$ can be represented accurately with the precision chosen, but where~$P(z)/P'(z)$
can be computed flawlessly.

For example, with $P(z)=z^{64}+1$ and $p=24$ bits, the evaluation of $P(10)$ and $P'(10)$ with our \texttt{FP32}
implementation produces \texttt{inf} because of the obvious overflow. However, we can compute
the correct Newton increment $10-N_P(10) \simeq 0.15625$ with \texttt{FP32} hardware arithmetic,
which is actually an accurate value up to $2\times 10^{-65}$.

\medskip
The third point in favor of the \A{} algorithm occurs when the sequence~$z_n$ eventually approaches a root of~$P$, as it should; the
computation of $P(z_n)$ then leads to an increasing number of \emph{cancelations}. Using Remark~\ref{rmk:cancel},
we can easily issue a warning when it is time to switch the computations to a higher precision.

\medskip
The last point is that \A{} is \emph{embarrassingly parallel}, which means that multiple roots can be searched for
simultaneously on different cores using the method of~\cite{HSS2001}. Also, contrary to more global algorithms that can be influenced
negatively if some of the evaluation points lead to overflow values (\eg if $z_n$ is a near miss of a root of $P'$), each
computation with \A{} is carried out independently of the others, even on a single core. 

\medskip
An example of root-finding using our implementation is illustrated in Figure~\ref{fig:newton}.

\subsection{Perspectives}

Quadrature methods are at the heart of numerical analysis \cite{BERMAD}, \cite{SSD04}.
Using \emph{Gaussian quadrature}, one may use $n$ evaluations to compute exactly the integral
of a polynomial of degree $2n-1$ over a given interval. The evaluation points (and the
weights of the linear combination) are determined by orthogonal polynomials.
The \A{} algorithm can be used to speed up the evaluations without compromising precision
in the case of high-degree polynomials (typically $n\gg 100$).

\medskip
Clenshaw's algorithm~\cite{CL1955} generalizes H\"orner's method in order to \emph{evaluate recursively} linear combinations
of a polynomial basis, which is itself defined by a three-term recurrence relation.
The principal of lazy addition at the heart of the~\A{} algorithm could be used in this general context
to reduce finite precision computations to a parsimonious summation. The practical condition
is the ability to compute easily the scale of the basis functions at a given point (like $\log_2 z^k=k\log_2 z$).

\medskip
Extending the \A{} algorithm to the \emph{multivariate case} would be a welcome generalization
because the number of terms increases drastically. There are
$\frac{(d+n-1)!}{d! (n-1)!}$ monomials of total degree $d$ in $n$ variables, \ie $O(d^{n-1})$ if $d\gg n$.
For example, a polynomial of degree 68 in 4 variables contains more than a million monomials, which is an instance of the
well known curse of the dimension.
For a recent study of the error estimates in the multivariate H\"orner algorithm, we refer the reader to~\cite{PS2000}.

The key idea of the \A{} algorithm (namely the lazy addition) is independent of the dimension.
The transfer of the analysis of the dominant coefficients to an arbitrary evaluation point~$(x_1,\ldots,x_n)\in\C^n$ remains similar 
to the~1D case~\eqref{equ:shear}:
\[
\log(|a_{k_1,\ldots,k_n} x_1^{k_1}\ldots x_n^{k_n}|) = \log|a_{k_1,\ldots,k_n}| + \sum_{j=1}^n k_j \log|x_j| \,.
\]
The main question will be to estimate the average complexity of the~\A{} algorithm, which is essentially equivalent to the question
of computing the average area of the horizontal projection of the largest hyperplane wafer that can be sandwiched between
two copies, vertically offset by $\delta$, of the graph of a concave function. A preliminary numerical exploration with
a half-sphere function, \ie $f(x_1,\ldots,x_n)=\sqrt{1-|x|^2}$, confirms that the area does scale as $\delta^{n/2}$ for $n=1,2,3$
when $\delta\to0$, which is encouraging.

\medskip
Finally, let us mention that the \A{} algorithm is of interest when evaluating polynomials or \emph{analytic functions
on a disk}. Remark~\ref{rmk:case_disk} gives the appropriate weight to compute the average complexity
when $z$ is chosen at random uniformly on a disk. In Section~\ref{sec:bench}, this case is benchmarked,
along with the Riemann sphere and the real line.

\section{Implementation and benchmarks}
\label{sect:code}
We have implemented our algorithm in the \texttt{C} language
and we release the implementation as an open-source project~\cite{FPELib}.
Our implementation aims for the highest versatility and user-friendliness, without compromising performance.
As a general rule, special cases that can lead to a substantial optimization are automatically recognized and dealt with.

\subsection{General considerations}

The main function, \texttt{FastPolyEval}, is called at the command line.
Polynomials are specified as \texttt{CSV} files (passed as arguments) in which
each coefficient, starting with~$a_0$, is written as a pair of its real and imaginary part in decimal form.
For example, the polynomial $P(z)=2+(3-5i)z$ is represented by the listing
\begin{verbatim}
2, 0
3, -5
\end{verbatim}
Similarly, the set of evaluation points is specified as a \texttt{CSV} file that obeys the same format.

\medskip
The first argument is systematically the precision at which the result of the operation is desired.
If the requested precision is at most 24, 53 or 64 (depending on compile time options),
\texttt{FastPolyEval} uses machine floating numbers, respectively \texttt{FP32}, \texttt{FP64} or \texttt{FP80}.
Otherwise, arbitrary-precision MPFR floating numbers~\cite{MPFR} are used.
It is therefore possible to store the values of a polynomial with a high precision in a file
and only use machine precision in a first set of low-precision evaluations, without worrying about a performance loss.
On the contrary, if the precision requested exceeds that of the input, the input is considered exact (in decimal form)
and padded with zero trailing bits if necessary.

\medskip
\texttt{FastPolyEval} automatically identifies the case of real polynomials (all imaginary parts of coefficients are identically zero)
because one can preprocess this case faster. Similarly, evaluations along the real line are also silently optimized
by the evaluator.
Computing the scale of a real number is indeed about twice as fast as computing the scale of a complex number.
In all cases, when using high-precision numbers, the scale, which is integer valued, is computed efficiently using only machine precision.

\medskip
Our implementation of the \A{} algorithm is complemented by a set of tasks that can generate
polynomials (interpolation from a given set of roots, four common orthogonal families, the family of polynomials
associated with the hyperbolic centers of the Mandelbrot set) and to manipulate them (sum, products, derivatives).
Rescaling can be done by evaluating $\lambda z$ on the coefficients. We also provide a comprehensive
set of tools to build and operate on sets of complex numbers. See Appendix~\ref{FPETasks}.

\medskip
The tasks \texttt{-eval}, \texttt{-evalD} and \texttt{-evalN} can be used directly in production cases to evaluate
a polynomial, its derivative or one Newton step with the \A{} algorithm.
An optional argument can be passed to generate a report on the number of bits that can reasonably be trusted
in each evaluation, in accordance with Remark~\ref{rmk:cancel}.
Additional arguments enable the benchmark mode (timing, comparison with H\"orner).
One important optional argument is the \texttt{errorsFile} specification, that generate a complementary report
on the estimated quality of the evaluation at the given precision (see Remarks~\ref{rmk:cancel} and~\ref{rmk:HornerAndBitLoss}).
For each evaluation point, it contains an upper bound
for the evaluation errors (in bits), a conservative estimate on the number of correct bits of the result,
and the number of terms that where kept by the~\A{} algorithm.

\medskip
The \texttt{-iterN} task is for the convenience of the user and provides a reasonably optimized stopping criterion
for Newton's method. For best results, we recommend multiple runs, each with a limited number of iterations, and where the precision is gradually increased.
The choice of the starting point and the pruning of duplicates
is left to the end-user; see~\cite{HSS2001} for guidance.
For a complete implementation of a splitting algorithm, we refer the reader to~\cite{MV21}.

\medskip
The \texttt{-analyse} task computes the concave cover $\ep$, the strip $G_p$,
and the intervals of $|z|$ for which the evaluation strategy (\ie the reduced
polynomial $Q_\lambda(z)$) changes. It is intended mostly for an illustrative purpose on low degrees, when the internals of the \A{}
algorithm can still be checked by hand. However, the intervals where a parsimonious representation is valid may also be of practical use;
see Section~\ref{par:parsimonious}.

\medskip
The question of parallelization is a legitimate one if one wishes to get the most out of modern hardware.
If the number of evaluation points is high compared to the core count,  the algorithm~\A{}
is embarrassingly parallel. Further optimization could be achieved by performing
evaluations at points of similar size on the same core.
To avoid an excessive complexity of the code that may only be of use in some specialized
application, we chose to only implement a single-core version of \A{}.

\subsection{Implementation notes}

The fact that $E_P$, defined by~\eqref{def:EP}, is discrete valued helps build a concave cover $\ep$ with few segments
(see Section~\ref{sect:preCond}), which in turn speeds up
the binary searches for $k_\lambda$, $\ell$ and $r$ in the evaluation phase.
Note that even if $\log_2(|a_k|)$ is concave, the scale function is integer valued, which, in practice, may prevent $E_P$
from being concave. See Figure~\ref{fig:discreteCover}.

In the course of sorting the values $s(a_k)$, we could check whether the profile is concave and, if it is indeed concave,
we could identify the maximum in an overall of $2d$ operations and reduce the complexity of the preconditioning to
only~$O(d)$ operations. However, in general, it induces a loss in the evaluator (more segments in $\ep$) and it is not worth
the trouble. Similarly, using a non-integer scale would be ill advised. 

\begin{figure}[H]
\begin{center}
\includegraphics[width=.9\textwidth]{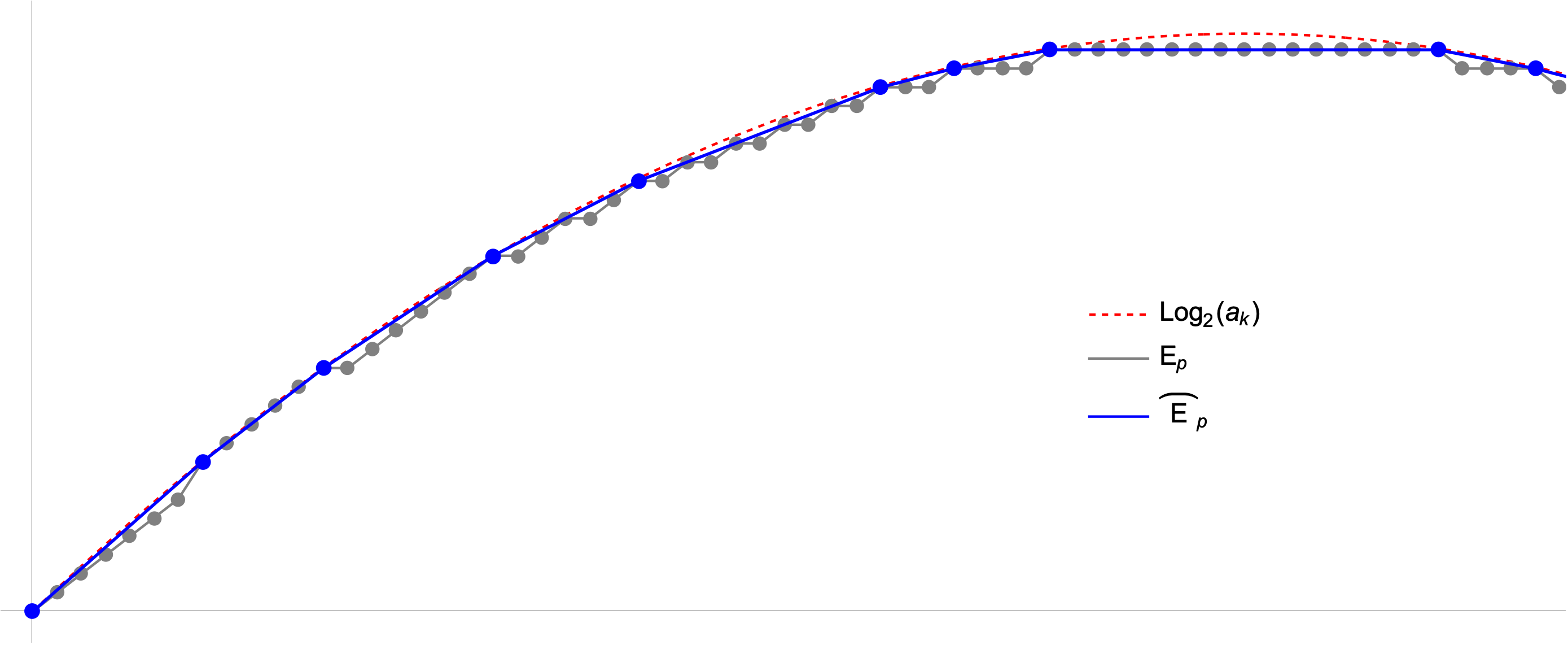}
\caption{\label{fig:discreteCover}
Even if $\log_2|a_k|$ is concave, the scale function $E_P$ is ``pixelated'' and not necessarily concave.
It is a good thing because $\ep$ contains fewer segments, which optimizes the evaluator.
}
\end{center}
\end{figure}

At the end of each evaluation, one needs to compute the valuation monomial, \ie$z^\ell$. The canonical method
consists in writing $\ell = 2^{\alpha} + \beta$ with $\beta<2^{\alpha}$.
As $z^{2^\alpha}$ can be computed with $\alpha$ successive squaring that can be kept in memory to compute $z^\beta$,
the evaluation costs at most $2\log_2 \ell$ multiplications.

In the preconditioning phase, we mark a set of indices $B_p \subseteq \ii{0}{d}$ that never need to be computed
for the given precision $p$ (see below~\eqref{eq:threshold}).
If their density is close to one in $\ii{\ell}{r}$, which is the case for sparse polynomials, then, to evaluate $z^\ell$,
we  pre-compute other powers of $z$ than the canonical $z^{2^j}$ with $j \leq \log_2 \ell$.
For example, the gaps can be filled more efficiently and be dynamically optimized for the interval $\ii{\ell}{r}$,
with a negligible overhead. This remark is implemented in \texttt{FastPolyEval}, which ensures an equal
treatment of all possible types of lacunarity, be it regular or not.

\subsection{Benchmarks}\label{sec:bench}

We have tested the correctness and efficiency of our implementation on
\emph{a few classes of polynomials} that are either of large interest, hard to handle in general, or both.

The \textit{Chebychev}, \textit{Legendre}, \textit{Laguerre} and \textit{Hermite polynomials} are classic. The \textit{hyperbolic polynomials} play a central
role in the study of the Mandelbrot set and are defined recursively by
$p_1(z)=z$ and  $p_{n+1}(z) = p_n(z)^2+z$.
\textit{Normal polynomials} are of the form $P_\omega(z) = \sum a_n(\omega) z^n$ where
$a_n(\omega)\sim\mathcal{N}(0,1)$ are either real- or complex-valued random variables
following normal law. The so-called \textit{half-circle} family is defined by~\eqref{eq:halfcirclepoly}; the
coefficients are real valued and form a half-circle when drawn in logarithmic coordinates.
The complex version of the half-circle family is obtained by multiplying the coefficients 
of the previous family by a random phase uniformly distributed on the unit circle.
As explained in Section~\ref{par:optimal}, these polynomials are remarkably hard  to evaluate accurately.

\bigskip
Systematic benchmarks were performed on \emph{Romeo}, in the HPC center of the University of Reims.
The overall benchmark time depends on the family of polynomials and, obviously, on the degrees and precisions used;
in our case, it took a total CPU time of 70-90h per family.
Multiple identical runs (typically two consecutive runs of the H\"orner algorithm, and ten runs of the~\A{} algorithm) ensured
that the average time is not biased by the loading time of a library or by fluctuations in the ambient load of the server.
We performed complementary benchmarks on our personal computers to confirm
the data points for the smaller degrees.

\bigskip
As explained in Section~\ref{sect:eval}, we compute the average complexity when the
evaluation points are chosen uniformly on either $\CC$, $\RR$ or the unit disk.
Before going further in our analysis, let us comment on the \emph{number of evaluation points}
used for the benchmarks. 

We used 10\,084 points uniformly distributed on the Riemann sphere
(modulus ranging from $8\times 10^{-3}$ to $2\times 10^2$), 5\,000 points uniformly distributed
on the unit disk (modulus ranging from $4\times 10^{-3}$ to~$1$) and 5\,000 on the real line (ranging from $\pm 6\times 10^{-5}$ to $\pm 2\times 10^3$).
Using more points does not change the average time significantly (see Figure~\ref{fig:compNbPts}), however it can
dramatically and unnecessarily extend the CPU time.

\begin{figure}[H]
\captionsetup{width=\linewidth}
\begin{center}
\includegraphics[width=\textwidth]{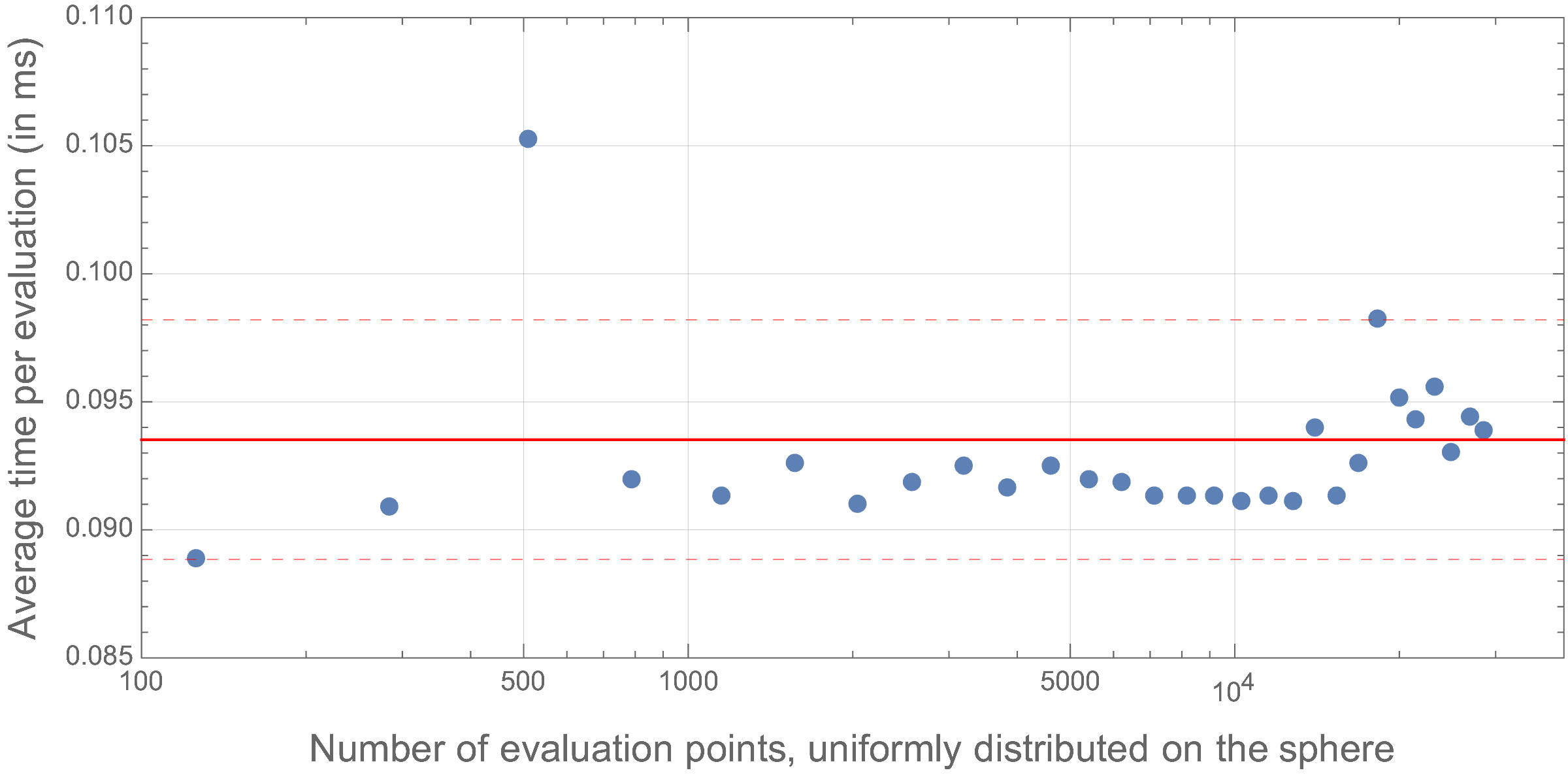}
\caption{\label{fig:compNbPts}
Influence of the number of evaluation points on the Riemann sphere on the average evaluation time with the~\A{} algorithm for
a polynomial of degree 1047 in the half-circle family, using 100 bits of precision. The red line is the mean value, the
dashed lines are the $\pm5\%$ deviations.}
\end{center}
\end{figure}

An \emph{order of magnitude} of the computation time is given in Figure~\ref{fig:compType}.
On a modern laptop${}^\ast$\footnote{${}^\ast$ MacBook Pro 2018, Intel Core i7, 2.6GHz, 16G RAM.}\!,
the evaluation of a polynomial of degree 1024 with a precision of 100 bits with
H\"orner's method takes in the ballpark of~$143\mu s$ $\pm23\%$${}^{\ast\ast}$.
\footnotetext{${}^{\ast\ast}$ Value obtained as average of four benchmarks on $\CC$, one on the unit disk and one on $\RR$,
amounting to 60 FPE preprocessings, 503\,360 FPE evaluations, and 100\,672 H\"orner evaluations for each of the 9 polynomial
families mentioned in Figure~\ref{fig:compType}.}  With the~\A{} algorithm,
the computation time depends significantly on the shape of $E_P$.
It can affect the preprocessing time negatively if $\ep$ has many segments. Conversely, it can affect  the evaluation time
positively if $\ep$ has few segments (ideally, with radically different slopes) or if $G_p$ contains a small number of terms (see Section~\ref{par:algo}
for the definition of $G_p$).
The order of magnitude of the preprocessing step is~$84\mu s$ $\pm32\%$ and subsequent evaluations with~\A{} boil down
to $46\mu s$ $\pm62\%$.

\medskip
For a \emph{one time evaluation}, the balance tilts slightly in favor of~\A{},
which is an interesting practical update on the optimality of the H\"orner scheme. Note that our experiment does
not contradict the theoretical result of Ostrowski~\cite{OST54} and Pan~\cite{PAN66} because our advantage holds on average
and only for computations with a fixed precision.

For a one time evaluation, \A{} outperforms H\"orner when $M(p)\gg \log d$. In the preprocessing
phase, we only need to read the exponents of the coefficients, which remains a small amount
of data to handle ($O(d\log d)$ with a small fixed constant). The evaluator then performs a minimal
number of costly high-precision operations.
H\"orner on the other hand, has $O(d M(p))$ bit-operations to perform and may end up being slower.
The advantage is especially pronounced in the complex
case, where each numerical product costs 4 real multiplications.

\begin{figure}[H]
\captionsetup{width=\linewidth}
\begin{center}{\small
\begin{tabular}{|l|c|cc|cc||cc||cc|}\cline{2-10}
\multicolumn{1}{c|}{} & \multicolumn{5}{c||}{Evaluation on $\CC$\rule[-6pt]{0pt}{18pt}}
& \multicolumn{2}{c||}{on $D(0,1)$} & \multicolumn{2}{c|}{on $\RR$}\\\hline
\multirow{2}{*}{Family} & \multicolumn{3}{c|}{Average time (in ms)} & \multicolumn{2}{c||}{Gain}
& \multicolumn{2}{c||}{Gain} & \multicolumn{2}{c|}{Gain}\\
&  H\"orner & Preproc. & FPE &  sing. & asym. &  sing. & asym. &  sing. & asym.\\\hline
Half-circle $\C$ & 0.176 & 0.137 & 0.090 & $\times$0.8 & $\times$2.0 & $\times$0.9 & $\times$1.7 & $\times$0.8 & $\times$2.8 \\\hline
Half-circle $\R$ & 0.149 & 0.076 & 0.072 & $\times$1.0 & $\times$2.1 & $\times$0.9 & $\times$1.6 & $\times$0.7 & $\times$3.4 \\\hline
Hyperbolic 	& 0.152 & 0.084 & 0.069 & $\times$1.0 & $\times$2.2 & $\times$0.9 & $\times$1.7 & $\times$0.6 & $\times$3.2 \\\hline
Normal $\C$	& 0.174 & 0.100 & 0.055 & $\times$1.1 & $\times$3.2 & $\times$1.0 & $\times$2.2 & $\times$0.9 & $\times$5.0 \\\hline
Normal $\R$	& 0.160 & 0.049 & 0.076 & $\times$1.3 & $\times$3.3 & $\times$1.2 & $\times$2.4 & $\times$0.8 & $\times$5.8 \\\hline
Chebychev 	& 0.152 & 0.057 & 0.036 & $\times$1.7 & $\times$4.3 & $\times$1.6 & $\times$3.4 & $\times$0.9 & $\times$4.6 \\\hline
Legendre	 	& 0.163 & 0.061 & 0.036 & $\times$1.7 & $\times$4.5 & $\times$1.5 & $\times$3.5 & $\times$0.9 & $\times$4.7 \\\hline
Laguerre 		& 0.140 & 0.082 & 0.017 & $\times$1.4 & $\times$8.3 & $\times$1.5 & $\times$10.0 & $\times$0.8 & $\times$10.9 \\\hline
Hermite 		& 0.141 & 0.059 & 0.015 & $\times$1.9 & $\times$9.3 & $\times$2.3 & $\times$13.6 & $\times$1.0 & $\times$8.1 \\\hline
\end{tabular}}
\caption{\label{fig:compType}
Average computation time on the Riemann sphere for polynomials of degree~1024 for various polynomial families,
using 100 bits of precision on a modern laptop.
The gain refers to the average benefit in computation time that can be expected from switching from H\"orner to the~\A{} algorithm,
either in a single evaluation or asymptotically, if the number of evaluation points is large.
The last four columns give the average gain if the evaluation points are chosen instead on the unit disk or along the real line.
}
\end{center}
\end{figure}

\begin{figure}[H]
\captionsetup{width=.95\linewidth}
\begin{center}
\includegraphics[width=\textwidth]{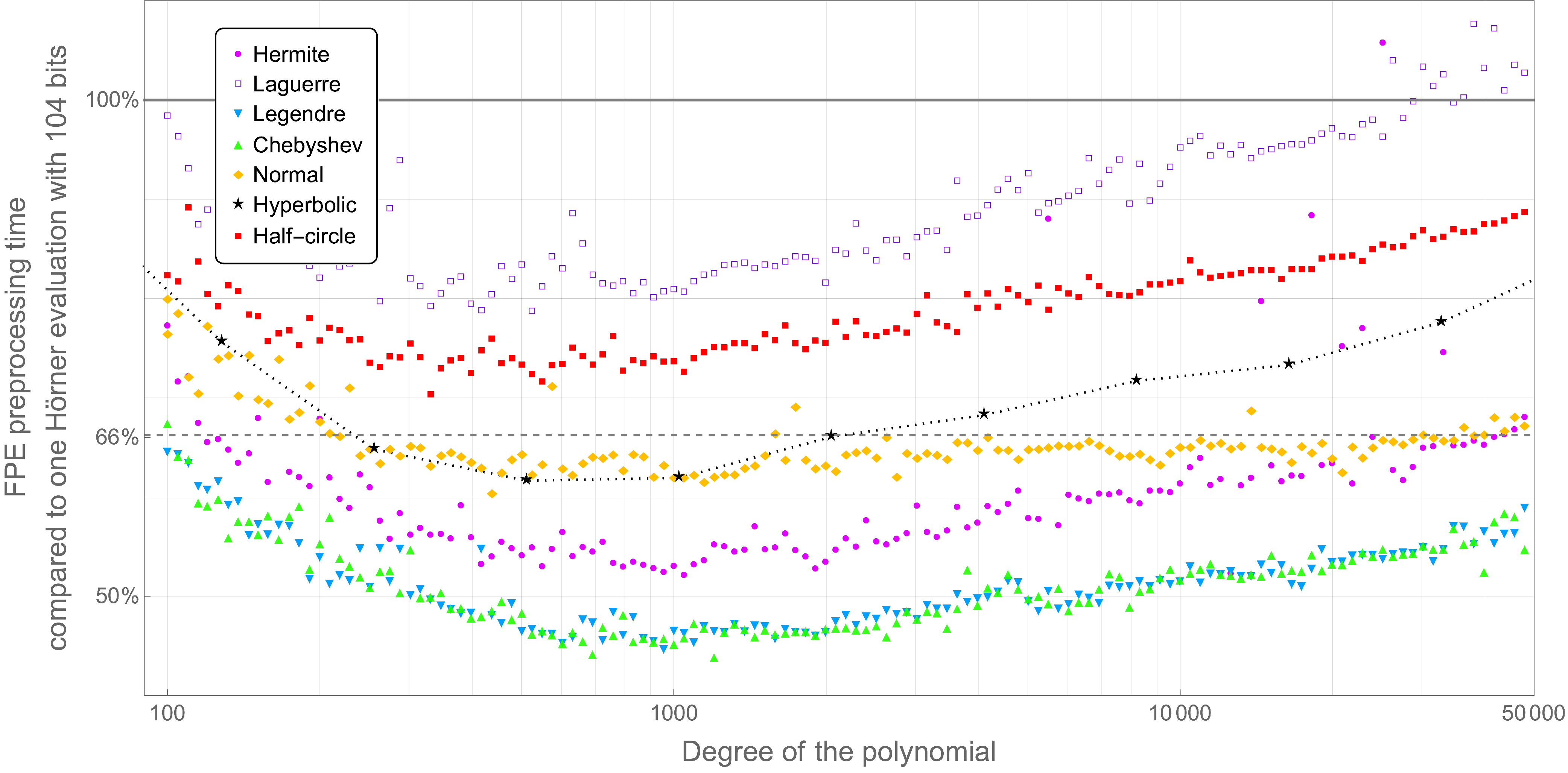}
\caption{\label{fig:preproc}
The FPE preprocessing time, which is independent of the precision,
represents, on average, only~66\% of one typical H\"orner evaluation with 100 bits on $\CC$.
Benchmark data generated on \textit{Romeo} (HPC center of the University of Reims).
}
\end{center}
\end{figure}

\begin{figure}[H]
\captionsetup{width=.95\linewidth}
\begin{center}
\includegraphics[width=\textwidth]{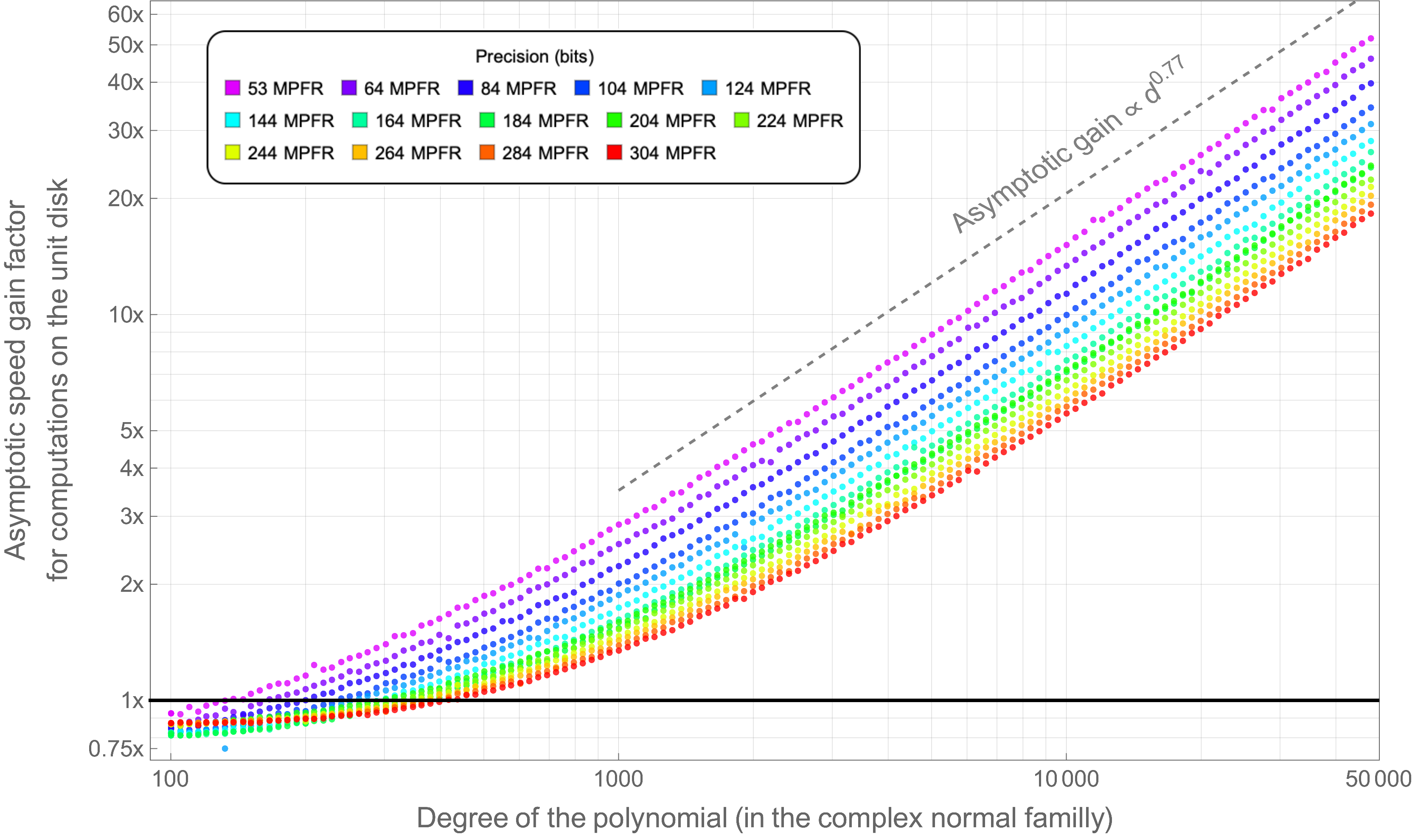}
\caption{\label{fig:diskNormalGain}
Asymptotic speed gain of FPE over H\"orner for evaluations of complex normal polynomials on the unit disk $\{z\in\C\,;\,|z|<1\}$, with various 
high precisions.
Benchmark data generated on \textit{Romeo} (HPC center of the University of Reims).
}
\end{center}
\end{figure}

\begin{figure}[H]
\captionsetup{width=.95\linewidth}
\begin{center}
\includegraphics[width=\textwidth]{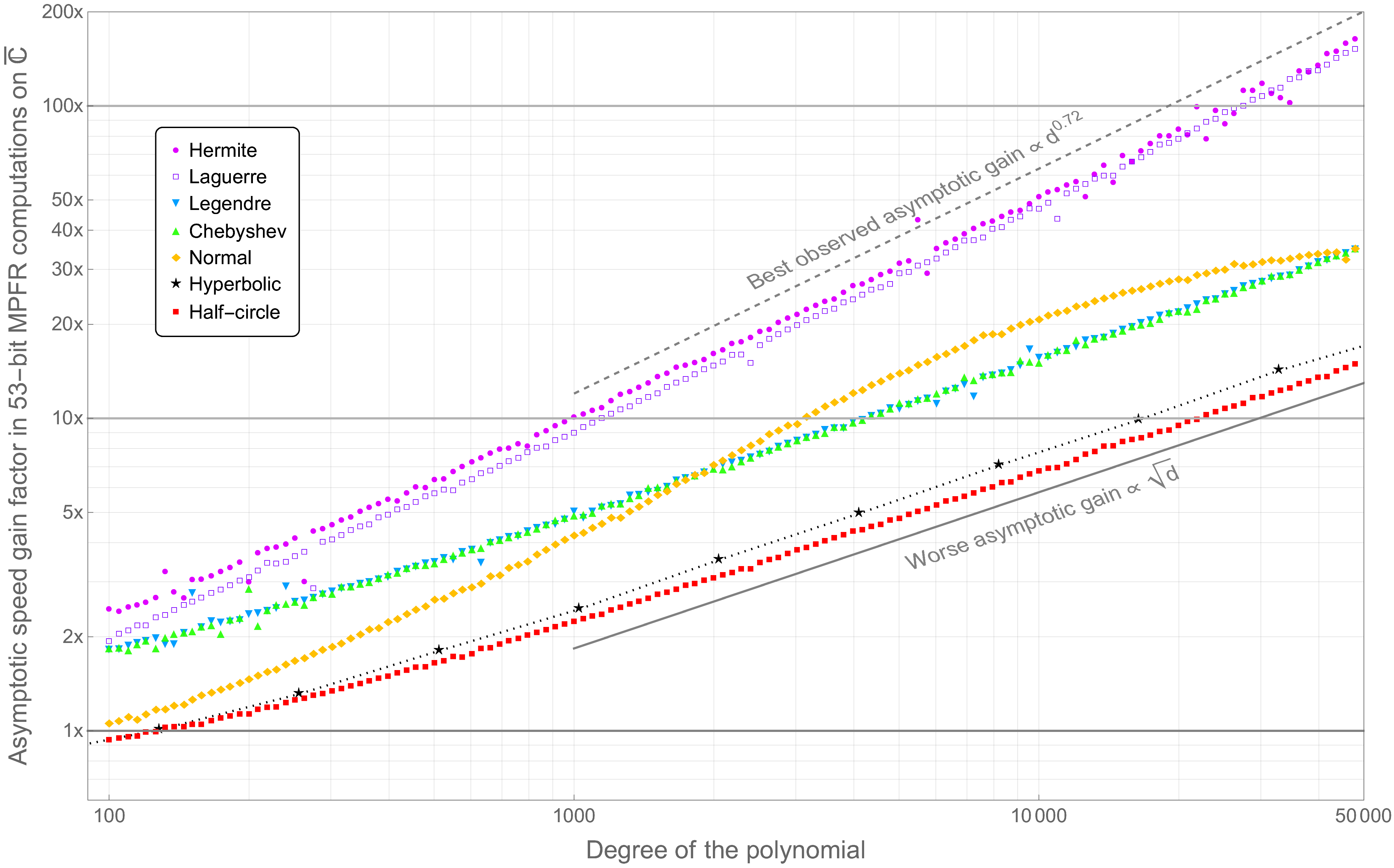}\\[1em]
\includegraphics[width=\textwidth]{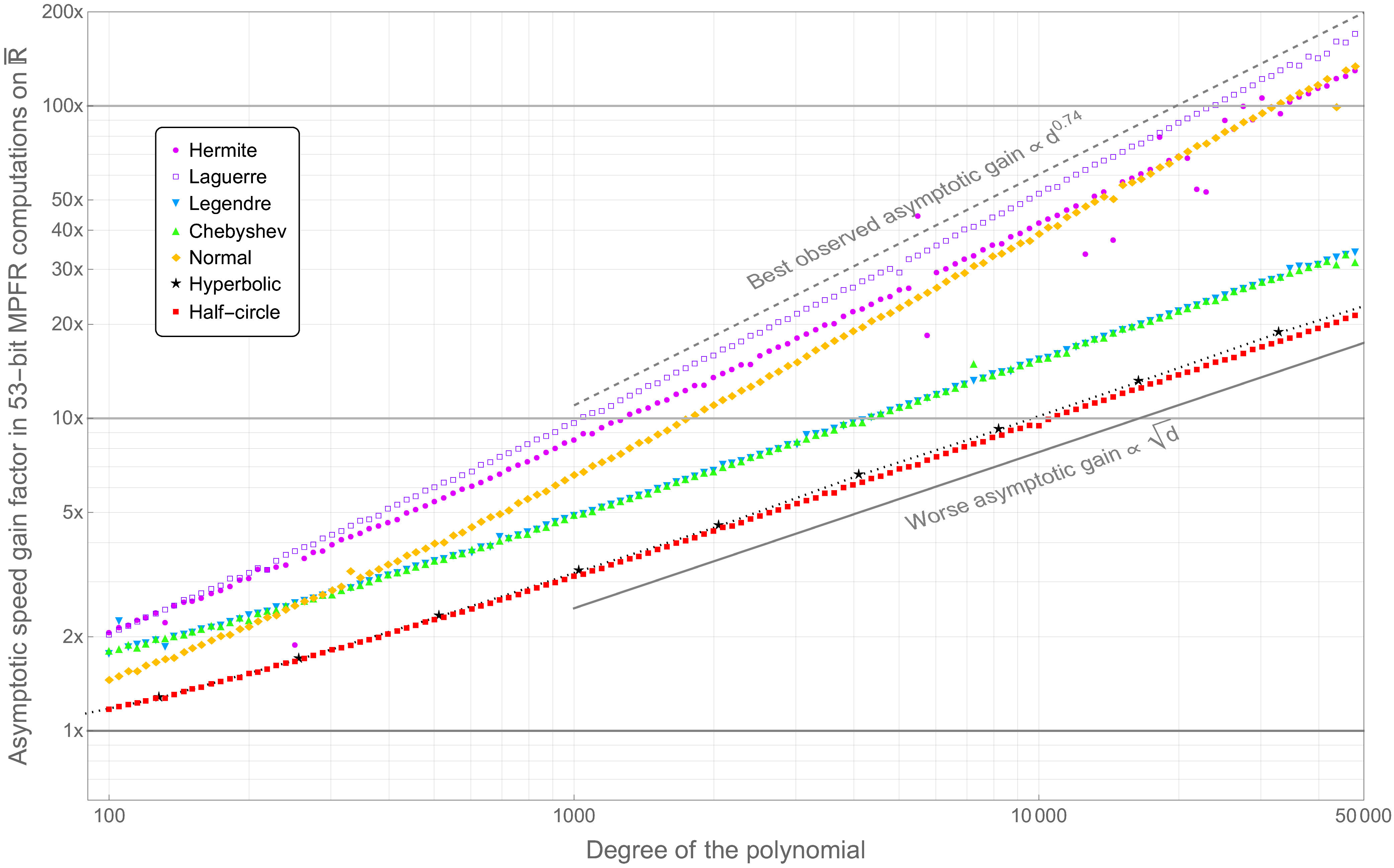}
\caption{\label{fig:compTypeSphLine}
Asymptotic speed gain of \A{} over H\"orner for various polynomial families on the Riemann sphere $\CC$ (top)
and on $\RR$ (bottom) for computations with 53-bit MPFR numbers.
Benchmark data generated on \textit{Romeo} (HPC center of the University of Reims).
}
\end{center}
\end{figure}

If the polynomial is evaluated \emph{repeatedly} (which is what~\A{} is designed for), the preprocessing  overhead becomes negligible
(see Figure~\ref{fig:preproc}) and the \emph{asymptotic gain} obtained by~\A{} becomes substantial.
Speedup in excess of $\times10$ occurs for some families like Hermite or Laguerre for degrees as low as 1000 (see Figure~\ref{fig:compType}).
In accordance with Theorem~\ref{thm:algo}, if the degree is high enough, the speed gain is bounded from below by~$O(\sqrt{d/\log_2 d})$,
which is observed in practice (see Figures~\ref{fig:diskNormalGain} and~\ref{fig:compTypeSphLine}).

In the best cases (Laguerre and Hermite; see Figure~\ref{fig:compTypeSphLine}),
the complexity of the FPE evaluation scales, in practice, as $O(d^{0.26})$ along the real line
and $O(d^{0.28})$ on the Riemann sphere.
In our data range, this complexity is consistent with $O(\log^2 d) \pm 10\%$, which is the theoretical bound suggested
by the last example of Section~\ref{par:algoResults}.
Note that Figure~\ref{fig:diskNormalGain} also hints that, in general,
the exponent of the scaling law of the complexity does not depend on the precision used for the computations.
Finally, let us point out that Chebychev, Legendre and  Hermite polynomials are mildly lacunary (alternately odd or even); the others are not.

\bigskip
Sorting the polynomial families by increasing \emph{asymptotic gain} as in Figure~\ref{fig:compType} and~\ref{fig:compTypeSphLine}
is effectively a way of measuring the complexity
in the variability of the scales of the coefficients. As explained in Section~\ref{par:optimal}, the slowest case is that of the half-circle family.
Similarly, the hyperbolic polynomials are slow to evaluate from their coefficients because of systematic
compensations among monomials on the Mandelbrot set, which represents a substantial part of the Riemann sphere (about 29\%).
On the contrary, if~$\ep$ is composed of only a few segments, there will be very few values of~$|z|$ for which massive compensations among monomials
can occur; in this case, the \A{} algorithm produces a very parsimonious representation of the polynomial (see Section~\ref{par:parsimonious}), which in turn
is responsible for extreme speed gains.

Most polynomial families behave qualitatively the same on $\CC$, $\RR$ and on the unit disk.
The only substantial anomaly in this classification occurs with the normal family (both real and complex), which is asymptotically evaluated significantly faster
on~$\RR$ than on~$\CC$: for a polynomial of degree $30\,000$ and a precision of 53 bits MPFR,
FPE evaluations are asymptotically 100 times faster than H\"orner's on the real line but only 30 times faster on the Riemann sphere (see Figure~\ref{fig:compTypeSphLine}).
A reasonable explanation beyond the fact that real powers are easier to
compute than complex ones, is the fact that the roots of normal polynomials accumulate uniformly
along the unit circle (Hammersley's theorem~\cite{Ham1956}, \cite{SZ2003}); therefore, one may expect fewer cancelations along
the real line than for other families. 
However, the anisotropic example at the end of Section~\ref{par:optimal} suggests caution and
further studies would be required to confirm this explanation. In particular, the reason why the evaluation time of
the normal family on the Riemann sphere fails to obey a power law contrary to all other families is not clear.

\bigskip
Our implementation~\cite{FPELib} handles both \emph{hardware number formats} \texttt{FP32}, \texttt{FP64} or \texttt{FP80}
and arbitrary-precision MPFR floating point numbers~\cite{MPFR}. 
The main limitation of hardware formats is the short range of exponents: roughly speaking, one can only represent numbers whose absolute value lies
between $10^{\pm38}$ with  \texttt{FP32} numbers and between $10^{\pm308}$ with \texttt{FP64} numbers. Concretely, this means it
is simply impossible to compute a monomial $z^n$ in \texttt{FP64} when $|z|\geq 2$ and $n>1024$.
The \texttt{FP80} format provides a slightly more comfortable range between $10^{-4\, 951}$ and $10^{4\, 932}$,
but it is still not enough to handle polynomials of degree 50~000 as in our benchmark.

In the range of exponents where a comparison was possible, hardware numbers behave about 4 times faster than 53-bit MPFR numbers;
however, the gain factor of \A{} over H\"orner obeys the same scaling law as for MPFR. 
In practice, \emph{the sweet spot} for using the \A{} algorithm with \texttt{FP80} numbers is for polynomials of degree 1\,000 to 5\,000 and $|z|<10$.
When using MPFR numbers, this range is extended to essentially any degree above 100 with almost no practical limitation on $|z|$.

\bigskip
The last crucial part in our benchmarks is the question of the \emph{accuracy} of the \A{} algorithm,
which is guaranteed by Theorem~\ref{thm:equivAlgo}.
To put it to the test, we systematically computed a 600-bit evaluation of our polynomials with a H\"orner scheme, which
served as a reference value. For each benchmarked precision (up to 304 bits), the outputs of both H\"orner and \A{}~algorithms with the current precision
were compared to the reference value to identify the absolute computation error.

The most significant data that can be extracted from this computation is the accuracy bias, defined as the difference of the number
of exact bits between the two algorithms, which is presented in Figure~\ref{fig:accuracy}. The practical conclusion is that the values
computed either by \A{} or by H\"orner are essentially identical, up to 1 \textit{exact} bit. Of course, when cancelations occur, the displayed
result may differ by many bits, but the divergence only affects the non significant bits at the end.
 
\bigskip
 \begin{figure}[H]
\captionsetup{width=.95\linewidth}
\begin{center}
\includegraphics[width=\textwidth]{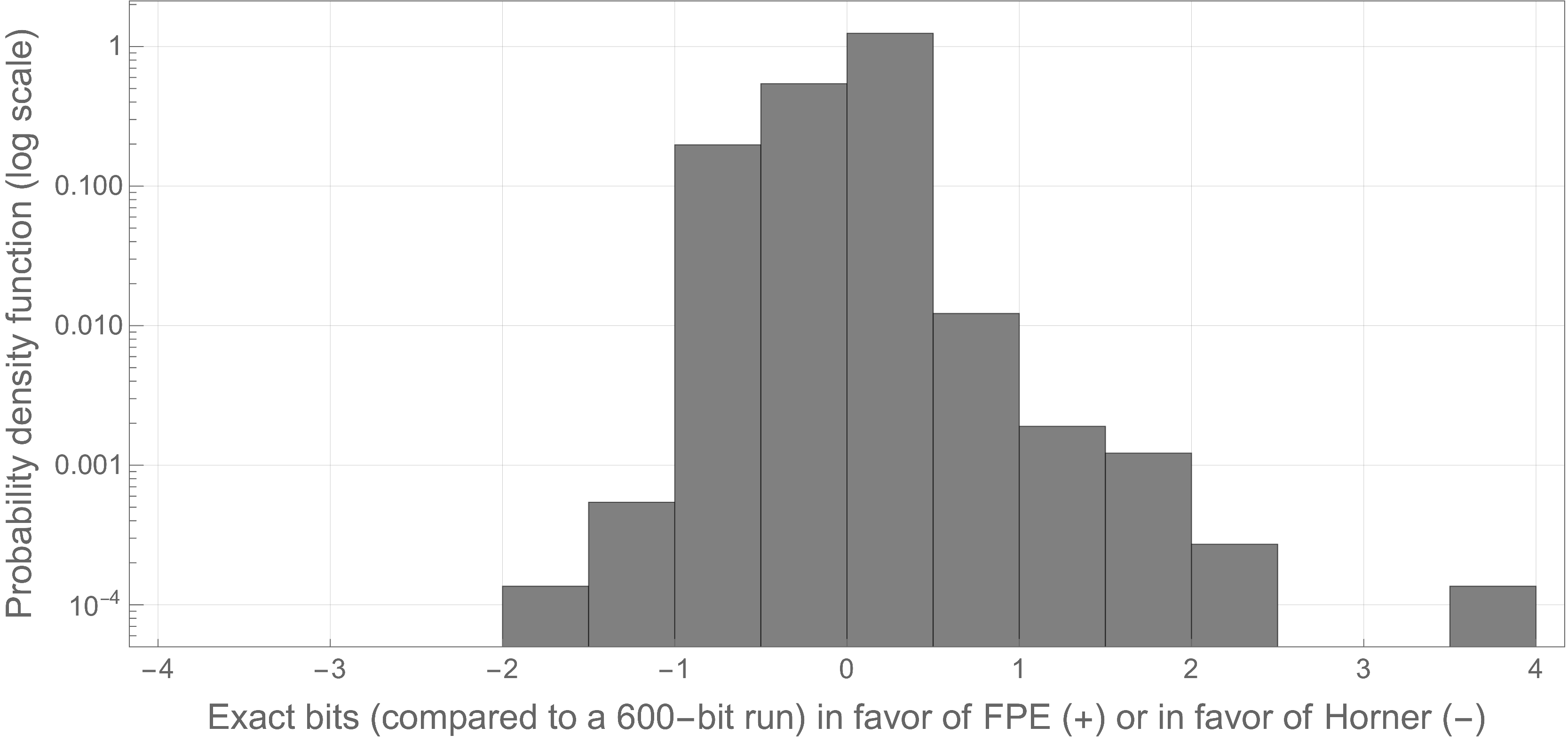}
\caption{\label{fig:accuracy}
No significant accuracy bias can be detected between H\"orner and \A{}.
This histogram is based on 14\,724 data points collected among the different families of polynomials that we have benchmarked
and various precisions, from 53 to 304 bits.
Overall, this represents 102\,910\,668 polynomial evaluations
generated on \textit{Romeo} (HPC center of the University of Reims).
}
\end{center}
\end{figure}

Based on this extensive benchmark, we can now confirm, in practice, that
the~\A{} algorithm holds the promise of Theorems~\ref{thm:equivAlgo} and~\ref{thm:algo} and \emph{performs as accurately as H\"orner,
only faster}.

\medskip 
Let us conclude this section by pointing out that our implementation of~\A{} provides additional tools for \emph{analyzing polynomial evaluations} like
the proportion of leading monomials at a given evaluation point (see Figure~\ref{fig:decimationCircle1000Sph}).
Similarly, the localization of cancelations in the evaluation process can easily be deduced from the output files (see Figure~\ref{fig:estimPrecSphere}),
which may guide practical decisions to ensure the precision of subsequent computations.

\begin{figure}[H]
\captionsetup{width=.95\linewidth}
\begin{center}
\includegraphics[width=.85\textwidth]{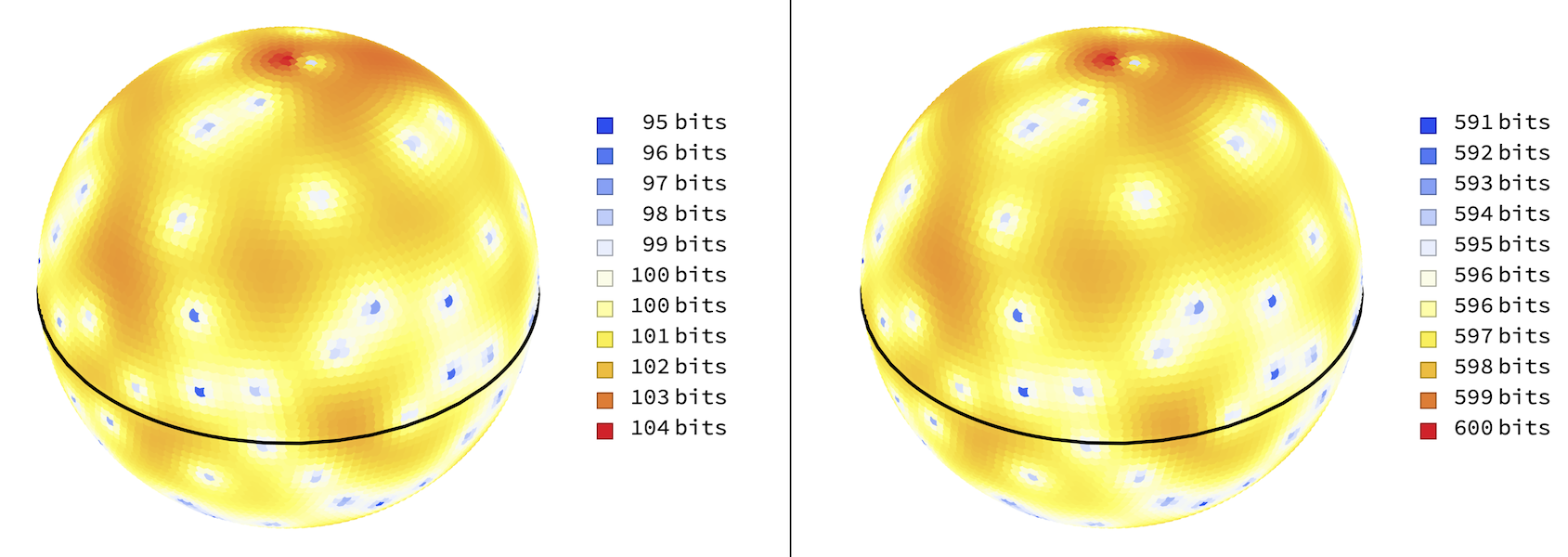}
\caption{\label{fig:estimPrecSphere}
Confidence regions of the \A{} algorithm in the evaluation of a half-circle polynomial~\eqref{eq:halfcirclepoly} of degree 100
along the Riemann sphere. A drop in the number of bits reported correct indicates exceptional cancelations.
As expected, cancelations are a feature mostly independent of the precision used (left 104 bits, right 600 bits).
}
\end{center}
\end{figure}

\medskip
We hope that the ideas presented in this article will inspire future developments, either theoretical or applied.
We also thank the reader for reaching this point.

\newpage
\appendix
\section{Proof of the geometric statements}
\label{par:geom_proof}

In this appendix, we prove the geometric results stated in Section~\ref{sect:analyse}.

\medskip
The proof of Theorem \ref{thm:conc} is based on two complementary geometric constructions
and is split into several lemmas.
In Lemma~\ref{lemma:segment-tangent} we show the existence of a segment~$[A(\tau)B(\tau)]$
in $S(f, \delta)$ of maximal length and of slope $\tau=\tan\theta $.
This segment  touches the graph of~$f-\delta$ in a ``tangent'' way (in the convex sense, \ie as a subderivative). 
\smallskip
In Lemma~\ref{lemma:conc-change-of-var}  we express the main integral from Theorem \ref{thm:conc} in terms of $x_R-x_L$
where $x_L$ and $x_R$ are the respective abscissae of $A(\tau)$ and $B(\tau)$.
Next, we make an alternative geometric construction of $x_R$ and $x_L$ based on area computations.
This second construction is the key to Lemma \ref{lemma:est-diagonal} where
we estimate the diagonal of a square built upon the graph of~$f'$.
A change of variable in the plane (Lemma \ref{lemma:cgt-var45}) allows us to collect all
the prior estimates and leads to the proof of the upper bound in Theorems \ref{thm:conc} and~\ref{thm:conc2}.
Finally, we obtain the lower bounds by constructing explicit examples.

\subsection{First geometric construction based on the graph of $f$}
Let us start by presenting the construction in a simple case. 

\paragraph{A typical example.}
For now, we suppose $f\in C^2([0,1])$ with $f''(x)<0$ for all $x\in (0,1)$
and we single out~$x_0\in (0,1)$.
The line of slope $f'(x_0)$ and passing through the point $(x_0, f(x_0)-\delta)$
is tangent to the graph of $f-\delta$.
This line intersects the graph of $f$ in at most two points $A$ and $B$, one on each side of $x_0$.
Indeed, the points of intersection with the graph of $f$ correspond to the solutions of the equation
\begin{equation}
\label{eq:end-segment}
f(x_0) -\delta + f'(x_0)(x-x_0)=f(x)
\end{equation}
\ie $F(x)=-\delta$, where $F(x)=f(x)-f(x_0) - f'(x_0)(x-x_0)$.
The function $F$ is of class~$C^2([0,1])$ and $F''(x)=f''(x)<0$ for all $x\in (0,1)$,
thus $F$ is concave. Therefore $F'(x)>0$ if $x<x_0$,  $F(x_0)=F'(x_0)=0$ and $F'(x)<0$ if $x>x_0$.
Consequently, there exist at most two points, one on each side of $x_0,$ such that $F(x)=-\delta.$
When they exist, we denote them by $0\leq x_L < x_0< x_R \leq 1.$ When they do not,
we simply take respectively $x_L=0$ or $x_R=1$.
In Figure~\ref{fig:strip}, the common abscissa of the points $A$, $A'$ is $x_L$ while
that of $B$, $B'$ is $x_R$.

\paragraph{General case.}
Let us now come back to the general setting where $f$ is concave, but not necessarily of class~$C^2([0,1])$.
The next statement extends the simpler case presented in the previous paragraph.
It is essentially an elementary version of F.~Riesz's
rising sun lemma~\cite{RieszRisingSun} in a concave setting.

\begin{lemma}\label{lemma:segment-tangent}
For any real number $\tau\in\R$, there exists a unique segment $[A(\tau)B(\tau)]$ of maximal length,
of slope $\tau = \tan \theta $, contained in the strip $S(f, \delta)$ and
that touches the graph of $f-\delta$ in a tangent way in the convex sense.
\end{lemma}
\noindent%
We denote by $g(x_0\pm 0)$ or $g(x_0^\pm)$
the sided limits of a function~$g$:
\[
g(x_0\pm 0) \defequal  \lim\limits_{\substack{{x\to x_0}\\ {\pm (x-x_0)>0}}} g(x) \,.
\]
A segment $[AB]$ of slope $\tan \theta $
is said to be \textit{tangent in the convex sense} to the graph of $g\in\F$ if, at any contact point $(x_0,g(x_0))\in[AB]$,
one has $g'(x_0+0)\leq \tan \theta  \leq g'(x_0-0)$. If~$g$ is smooth, then $g'(x_0)=\tan \theta $.
If an endpoint~$x_0\in\{0,1\}$ is a contact point, then the requirement is lightened  respectively
to $\tan \theta  \geq g'(0^+)$ or $\tan \theta  \leq g'(1^-)$.

\medskip
The abscissa of the endpoints of the maximal segment $[A(\tau)B(\tau)]$
will be respectively denoted by $x_L(\tau)$ for the left side and $x_R(\tau)$ for the right side.

\bigskip
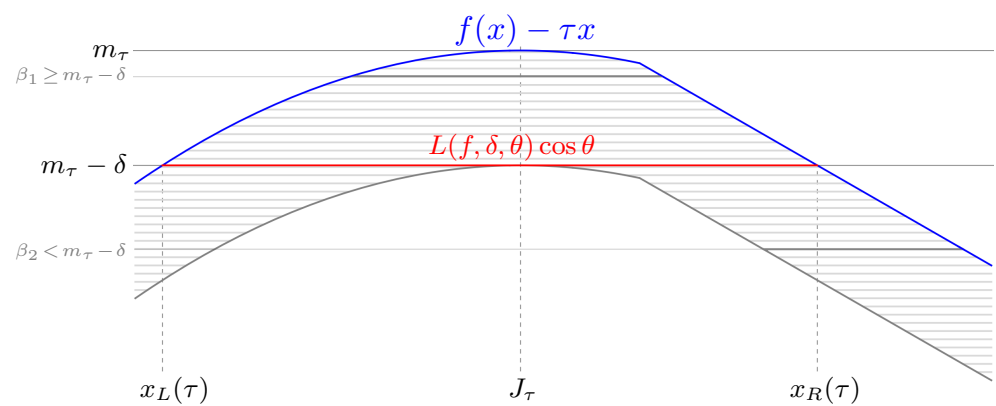
\begin{figure}[H]
\captionsetup{width=.95\linewidth}
\begin{center}
\resizebox{.8\linewidth}{!}{%
\setlength{\unitlength}{1cm}
\includegraphics[width=10cm]{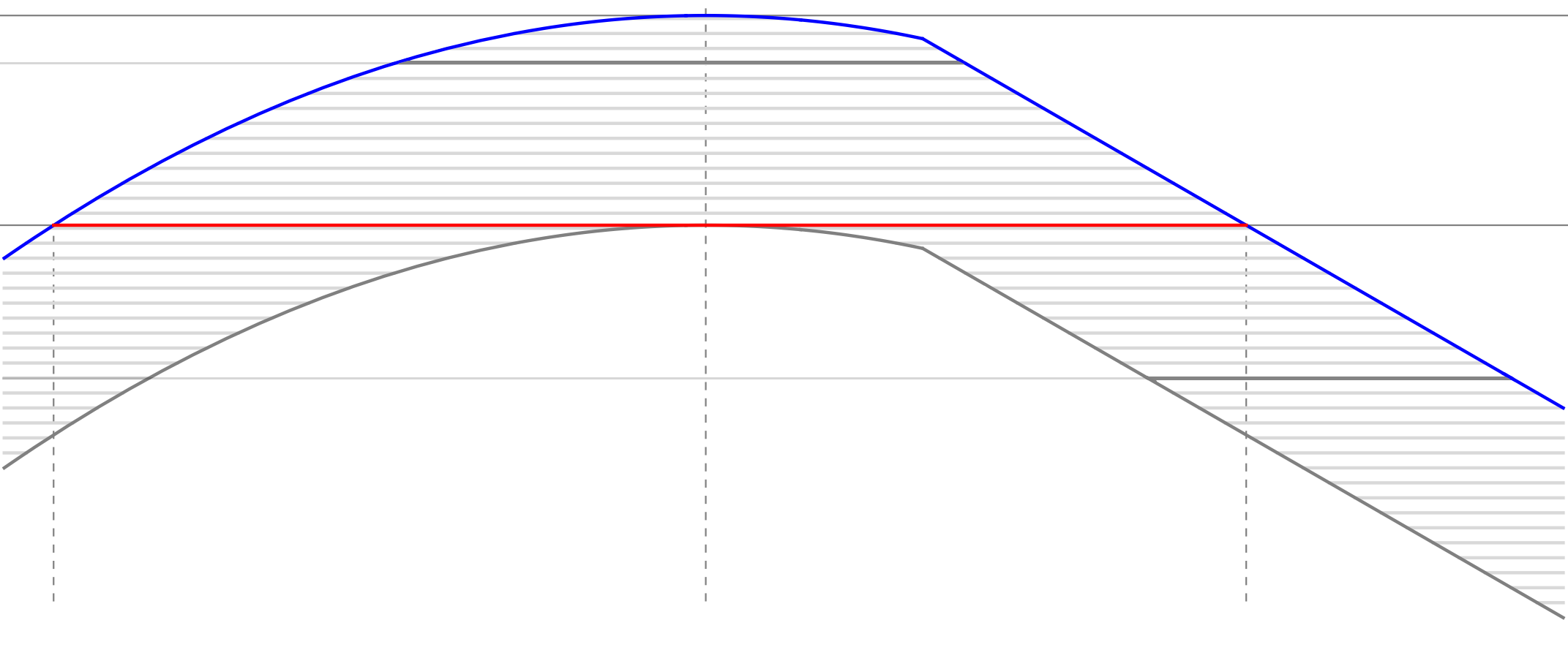}
\begin{picture}(0, 0)(10.1,0)
		\put(3.7, 4.2){\color{blue} $f(x)-\tau x$}
 		\put(3.4,2.85){\color{red}\footnotesize $L(f,\delta,\theta )\cos\theta $}
		\put(0.05,0){\footnotesize $x_L(\tau)$}
		\put(4.33,0){\footnotesize $J_\tau$}
		\put(7.6,0){\footnotesize $x_R(\tau)$}
		\put(-0.5,3.98){\footnotesize $m_\tau$}
		\put(-1.1,2.63){\footnotesize $m_\tau-\delta$}
		\put(-1.4,3.71){\tiny\color{gray} $\beta_1\!\geq\! m_\tau\!-\!\delta$}
		\put(-1.4,1.65){\tiny\color{gray} $\beta_2\!<\! m_\tau\!-\!\delta$}
\end{picture}
}
\caption{\label{fig:risingSun}%
The affine map $(x,y)\mapsto (x,y-\tau x)$ rearranges Figure~\ref{fig:strip} in a so-called
``rising sun'' configuration~\cite{RieszRisingSun}.
Recall that $\tau=\tan\theta $.
Two non-optimal segments illustrate the last part of the proof: the longest segment is the one
``tangent'' to the graph of $f(x)-\tau x-\delta$.
}
\end{center}
\end{figure}

\begin{proof}
As $f$ is a concave function, it is continuous, it is differentiable almost everywhere, and its derivative is decreasing.
Even when the derivative is not continuous at a point, it necessarily has left and right limits.
Therefore it has at most a countable number of jump points.

Let us first construct the segment $[A(\tau)B(\tau)]$.
For any $\tau\in\R$, the concave function~$f(x)-\tau x$ presents a maximum $m_\tau$ in~$[0,1]$,
which is reached on some non-empty compact sub-interval $J_\tau\subset [0,1]$ (usually a singleton).
The function $f(x)-\tau x$ is monotone on each of the connected components of  $[0,1]\backslash J_\tau$.
As a consequence, the set
\[
I(\tau,\delta)=\{x\in[0,1] \,;\, f(x)-\tau x \geq m_\tau-\delta\}
\]
is an interval increasing in $\delta$ that contains $J_\tau=I(\tau,0)$.
Let us define
\begin{equation}\label{eq:defxLxR}
\begin{gathered}
x_L(\tau) \defequal  \inf I(\tau,\delta)\,,
\qquad
x_R(\tau) \defequal  \sup I(\tau,\delta)\,.
\end{gathered}
\end{equation}
For any $x_0 \in J_\tau$, one has $m_\tau=f(x_0)-\tau x_0$ and the line of equation~$y=y_\tau(x)$ with
\[
y_\tau(x) \defequal  f(x_0) -\delta  +  \tau(x-x_0)  = \tau x + (m_\tau -\delta)
\]
does not depend on the actual choice of $x_0$ within $J_\tau$.
Let us define
\begin{equation}\label{defAB}
A(\tau) = (x_L(\tau), y_\tau(x_L(\tau)))\,, \qquad B(\tau) = (x_R(\tau), y_\tau(x_R(\tau)))\,.
\end{equation}
By definition~\eqref{eq:defxLxR}, for all $x\in[x_L(\tau), x_R(\tau)]=I(\tau,\delta)$,
one has $f(x) +\delta \geq m_\tau + \tau x \geq f(x)$,
thus the segment~$[A(\tau)B(\tau)]$ is of slope $\tau$ and is included in $S(f,\delta)$. 

\bigskip
Conversely, any segment~$[AB]$ of slope~$\tau$ included in $S(f,\delta)$ is supported
by a line of equation $y=\tau x+\beta$ and must satisfy (denoting by $x_A$, $x_B$ the abscissa of $A$ and $B$)
\begin{equation}\label{eq:constrain}
\forall x\in [x_A,x_B], \qquad
f(x) \geq \tau x+ \beta \geq f(x)-\delta \quad\ie\quad \beta+\delta \geq f(x)-\tau x\geq \beta\,.
\end{equation}
Let us show that $x_B-x_A \leq x_R(\tau)-x_L(\tau)$. 

\smallskip
If $\beta\geq m_\tau-\delta$ (\ie $[AB]$ is above $[A(\tau)B(\tau)]$) then
$f(x_A)-\tau x_A \geq \beta \geq m_\tau-\delta$;
the monotony of $f(x)-\tau x$ outside~$J_\tau$ and the definition of $x_L(\tau)$
imply $x_L(\tau)\leq x_A$.
Similarly, one has $x_B\leq x_R(\tau)$.
In other words, one has
\[
[x_A,x_B] \subset [x_L(\tau), x_R(\tau)]\,.
\]
If $\beta<m_\tau-\delta$, the constraint~\eqref{eq:constrain}
cannot be satisfied for $x\in J_\tau$ because
\[
\forall x\in J_\tau, \quad f(x)-\tau x = m_\tau > \beta+\delta\,,
\]
thus $[x_A,x_B]$ is a subset of~$[0,1]\backslash J_\tau$.
The concavity of $f(x)-\tau x$ implies that one can increase $x_B-x_A$ by shifting the interval towards $J_\tau$.
More precisely, let us assume for example that $[x_A,x_B]$ is on the right side of $J_\tau$ and that
$x_R(\tau)<1$ (otherwise nothing needs to be proved).
The inclusion $[AB]\subset S(f,\delta)$ implies
\begin{equation}\label{tech01}
f(x_A)-\tau x_A-\delta \leq \beta \leq f(x_B)-\tau x_B\,.
\end{equation}
The function $\tau-f'$ is defined almost everywhere and is positive and increasing
on the right-hand side of $J_\tau$. The inequality~\eqref{tech01} can thus be rephrased
\[
\int_{x_A}^{x_B} \tau-f'(x) dx \leq  \delta\,.
\]
Similarly, for any $x_0\in J_\tau$, one has $ f(x_0) -\tau x_0=m_\tau$ and  $x_R(\tau)<1$ 
implies :
\[
\int_{x_0}^{x_R(\tau)} \tau-f'(x) dx = m_\tau - \left(f(x_R) - \tau x_R(\tau)\right)= \delta\,.
\]
If $x_A \geq x_R(\tau)$ then the smaller integrand on $[x_0,x_R(\tau)]$ implies $x_R(\tau)-x_0\geq x_B-x_A$.
If~$x_A< x_R(\tau)$, the integrals on $[x_A,x_R(\tau)]$ cancel out, therefore
\[
 \int_{x_R(\tau)}^{x_B}  \tau-f'(x) dx \leq \int_{x_0}^{x_A}  \tau-f'(x) dx
\]
and thus $x_A-x_0\geq x_B-x_R(\tau)$. In both cases, $x_R(\tau)-x_L(\tau) \geq x_R(\tau)-x_0 \geq x_B-x_A$.

This proves that the maximal length is obtained for $\beta=m_\tau-\delta$.
As the values of the projection of $A,B$ on the $x$-axis are unique, the equations~\eqref{eq:defxLxR}-\eqref{defAB} ensure
that the segment~$[A(\tau)B(\tau)]$ is unique.
\end{proof}

The quantity that interests us for Theorem \ref{thm:conc} is obviously related to this first geometric construction.
\begin{lemma} 
\label{lemma:conc-change-of-var} With the notations of Theorem \ref{thm:conc}, we have
\begin{equation}
\int_{ - \frac \pi 2}^{\frac \pi 2} L(f, \delta, \theta ) \cos \theta  \, d\theta  =
\int_{-\infty}^{+\infty} \frac{x_R(y) - x_L(y)}{1+y^2} \, dy\,.
\end{equation}
We also have
\begin{equation}
\int_{ - \frac \pi 2}^{\frac \pi 2} L(f, \delta, \theta ) \cos^2 \theta  \, d\theta  =
\int_{-\infty}^{+\infty} \frac{x_R(y) - x_L(y)}{(1+y^2)^{3/2}} \, dy
\end{equation}
and for any positive measurable weight $\omega$ on $\left[-\frac{\pi}{2},\frac{\pi}{2}\right]$ :
\begin{equation}
\int_{ - \frac \pi 2}^{\frac \pi 2} L(f, \delta, \theta ) \, \omega(\theta)  \, d\theta  =
\int_{-\infty}^{+\infty} \frac{x_R(y) - x_L(y)}{\sqrt{1+y^2}} \:  \omega(\arctan y)\, dy\,.
\end{equation}
\end{lemma}
\begin{proof}
The definition~\eqref{defAB} ensures that the length of the projection on the $x$-axis
of the segment $[A(\tau)B(\tau)]$ is $L(f, \delta, \theta ) \cos \theta  = x_R(\tan \theta) - x_L(\tan \theta )$.
The identities are then obtained by the change of variable~$y=\tan \theta .$
\end{proof}

\subsection{A second geometric construction based on the graph of $f'$}
We are now going to provide a second geometric construction of $x_L(\tau)$, $x_R(\tau)$.
We consider the graph of $f'$ and we complete it to a continuous curve in the following way.
At a jump point we add the vertical segment that joins the left and the right limits.
If~$f'(0^+)<\infty$ then we add the half-line $\{0\}\times [f'(0^+), \infty)$.
Similarly, if~$f'(1^-)>-\infty,$ we add the half-line~$\{1\}\times (-\infty, f'(1^-)]$.
We thus obtain a continuous curve, which we denote by $\gamma_{f'}$,
that contains the graph of $f'$, whose projection on the $x$-axis contains $]0,1[$ and is included in $[0,1]$
and whose projection on the $y$-axis is $\R$.

\medskip
Let us introduce
\[
\Gamma_{f'}^\pm \defequal  \left\{ (x,y)\in [0,1]\times\R \,;\, \exists y' \text{ such that } \pm(y-y')\geq0 \text{ and } (x,y')\in\gamma_{f'}\right\}\,.
\]
The sets $\Gamma_{f'}^\pm$ are the closed subsets of the strip $[0,1]\times\R$ that are respectively above and below $\gamma_{f'}$.
For every $y_0\in \R$ we consider
\[
x_0^- \defequal  \min \left\{x\in [0,1] \,;\, (x,y_0)\in \gamma_{f'} \right\}
\quad\text{and}\quad
x_0^+ \defequal  \max \left\{x\in [0,1] \,;\, (x,y_0)\in \gamma_{f'} \right\}\,.
\]
The segment $\{(x,y_0) \,;\, x_0^- \leq x \leq x_0^+\}$ is the intersection between $\gamma_{f'}$
and the horizontal line $y=y_0$ ; it reduces to a point when $x_0^+=x_0^-$.
We now build a family of ``triangles'' whose hypotenuse rests on $\gamma_{f'}$ and that
collapse on the segment $[x_0^-,x_0^+]\times\{y_0\}$ (see Figure~\ref{fig:gamma-f'}). For $x\in[0,1]$, let
\[
T_{f'}(y_0;x) \defequal  \begin{cases}
\Gamma_{f'}^- \cap \left\{(x',y) \,;\, x\leq x'\leq x_0^- \text{ and } y\geq y_0 \right\} & \text{if }x<x_0^-,\\
\{x,y_0\} & \text{if }x\in [x_0^-,x_0^+],\\
\Gamma_{f'}^+ \cap \left\{(x',y) \,;\, x_0^+\leq x'\leq x \text{ and } y\leq y_0 \right\} & \text{if }x>x_0^+.
\end{cases}
\]
The continuity of $\gamma_{f'}$ implies that the area $|T_{f'}(y_0; x)|$ of this triangle
is a continuous function of $x$ and the monotonicity of $f'$ implies that the area vanishes
along $[x_0^-,x_0^+]$  and is respectively strictly decreasing on $[0,x_0^-]$ and strictly increasing on $[x_0^+,1]$.

\medskip
We are interested in the two points where either
the area of the triangle equals $\delta$ or the triangle hits
the edge of the strip:
\begin{align*}
\widetilde{x_L}(y_0) &\defequal  \inf\left\{ x\in [0,x_0^-] \,;\, |T_{f'}(y_0; x)| \leq \delta\right\}\,,\\[1pt]
\widetilde{x_R}(y_0) &\defequal  \sup\left\{ x\in [x_0^+,1] \,;\, |T_{f'}(y_0; x)| \leq \delta\right\}\,.
\end{align*}
Let us prove that $\widetilde{x_L}(y_0) = x_L(y_0)$ and $\widetilde{x_R}(y_0) = x_R(y_0)$, \ie they are the same values
as the ones defined by \eqref{eq:defxLxR}. Using elementary calculus, we know that the area of the triangle is
\[
|T_{f'}(y_0;x)| =
\begin{cases}
\displaystyle\int_{x}^{x_0^-} \left( f'(x) - y_0\right) dx = -(x_0^--x) y_0 +f(x_0^-) - f(x) & \text{if } x<x_0^-,\\[1.5ex]
\displaystyle\int_{x_0^+}^{x} \left( y_0-f'(x) \right) dx = (x-x_0^+) y_0 +f(x_0^+) - f(x)& \text{if } x>x_0^+.
\end{cases}
\]
The conditions defining $\widetilde{x_L}(y_0)$ and $\widetilde{x_R}(y_0)$ thus boil down to
\[
x\in [\widetilde{x_L}(y_0), \widetilde{x_R}(y_0)] 
\quad\Longleftrightarrow\quad
f(x) -  y_0 x \geq f(x_0^\pm) -y_0 x_0^\pm   - \delta \,.
\]
Note that, by definition, $f(x)- y_0x$ is constant on $[x_0^-,x_0^+]$ and one recovers~\eqref{eq:defxLxR}
with $\tau =y_0$, $m_\tau= f(x_0^\pm) -y_0 x_0^\pm$ and $[x_0^-,x_0^+]=J_\tau$ from the proof
of Lemma~\ref{lemma:segment-tangent}.

\bigskip
\begin{figure}[H]
\begin{center}
\setlength{\unitlength}{1cm}
\includegraphics[width=14cm]{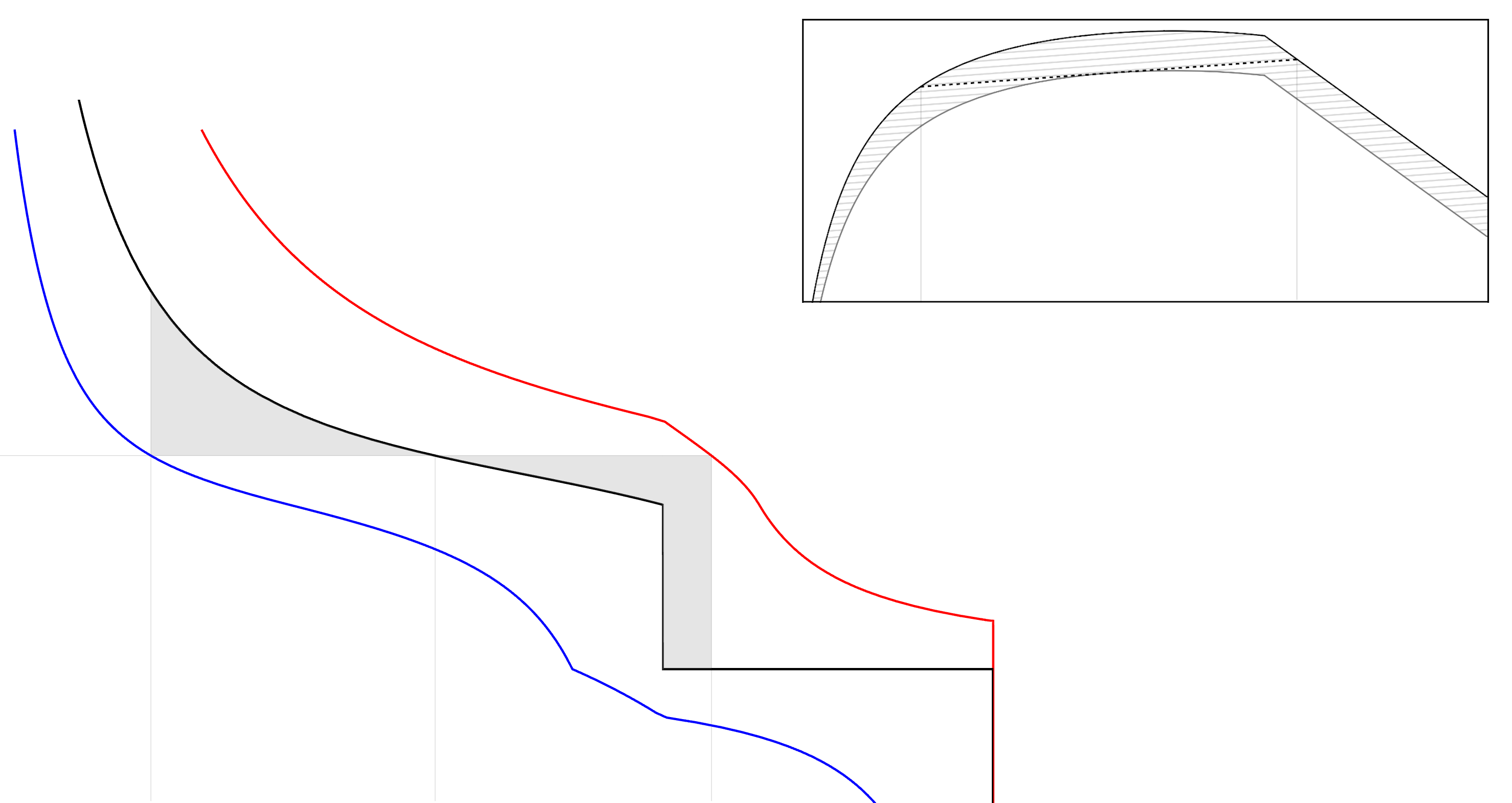}
\begin{picture}(0, 0)(14.1,0)
	\put(-0.1, 4.8){\color{blue} $\gamma_L$}
	\put(1.0, 6.2){\color{black} $\gamma_{f'}$}
	\put(2.6, 5.4){\color{red} $\gamma_R$}
	\put(0,3.1){$y_0$}
	\put(2,3.5){\color{gray} $\delta$}
	\put(6.38,2.4){\color{gray} $\delta$}
	\put(3.9,0){$x_0^\pm$}
	\put(1.1,0){$x_L(y_0)$}
	\put(6.25,0){$x_R(y_0)$}
	\put(8.3,4.8){\tiny $x_L(y_0)$}
	\put(8.1,6.77){\tiny $A(y_0)$}
	\put(11.78,4.8){\tiny $x_R(y_0)$}
	\put(12.2,6.9){\tiny $B(y_0)$}
	\put(8.9,4.2){\small $S(f,\delta)$ and $x\mapsto y_0 x+m_{y_0}-\delta$}
\end{picture}
\smallskip
\captionsetup{width=.95\linewidth}
\caption{\label{fig:gamma-f'}
Images of the curves $\gamma_{f'}$ in black, $\gamma_R$ in red and $\gamma_L$ in blue.
The triangles $T_{f'}(y_0;x_L)$ and $T_{f'}(y_0;x_R)$ are grayed out; their area is~$\delta$.
Inset: the profiles of $f$ and $f-\delta$ used to generate the figure; the dashed segment $[A(y_0)B(y_0)]$ has a slope $\tau=y_0$.
Note that, contrary to the profile presented in Figure~\ref{fig:strip}, this one has an unbounded slope near zero.
}
\end{center}
\end{figure}
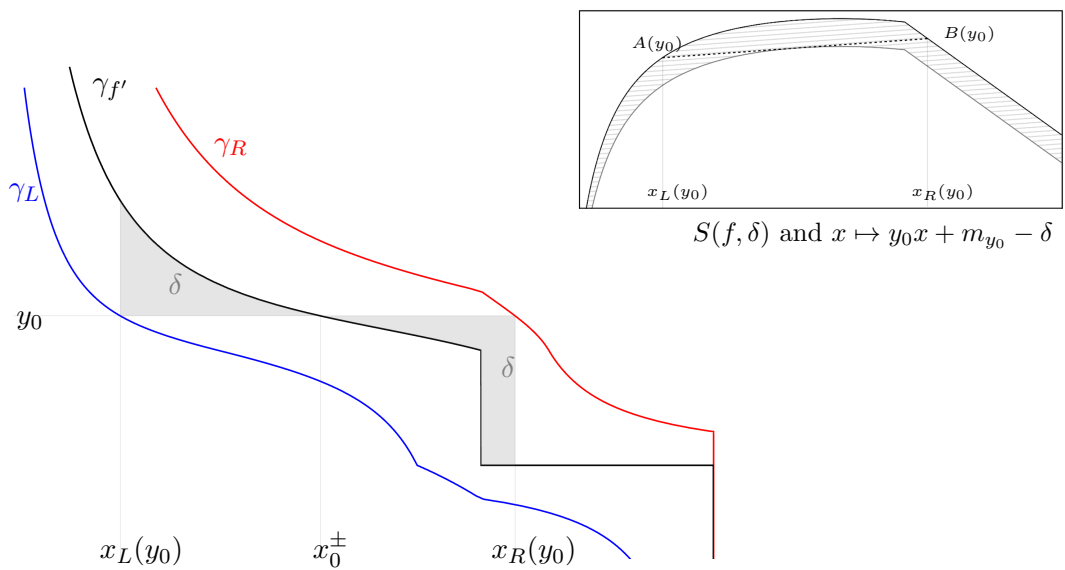

We now have two equivalent constructions of $x_L(y)=\widetilde{x_L}(y)$ and $x_R(y)=\widetilde{x_R}(y)$
for every $y\in \R$.
The second construction ensures that  the maps $y\mapsto x_R(y)$ and $y\mapsto x_L(y)$ are continuous
and decreasing on $\R$. They are strictly decreasing respectively when $x_L(y)>0$ and $x_R(y)<1$.
We denote by  $\gamma_L(\delta)=\{ (x_L(y),y) \,;\, y\in\R\}$ and $\gamma_R(\delta)=\{ (x_R(y),y) \,;\, y\in\R\}$ the two
curves that are ``offset''  from $\gamma_{f'}$ by a triangular area of $\delta$ (see Figure~\ref{fig:gamma-f'}).

\begin{lemma}\label{lemma:est-diagonal}
For every $y\in \R$ let us construct the unique square with an upper-right corner at $(x_R(y), y)\in \gamma_R(\delta)$
and a lower-left corner on $\gamma_L(\delta)$.
The area of this square is smaller than $2 \delta$ and therefore its diagonal is smaller than $2\sqrt \delta$.
\end{lemma}
\begin{proof}
The construction of the square is obvious.
The curve $\gamma_L$ is below the curve $\gamma_{f'}$, which is itself below  $\gamma_R$.
As $y\mapsto x_L(y)$ is decreasing, the curve $\gamma_{L}$ intersects a line of slope $\pi/4$
passing through $(x_R(y), y)$ in a unique point whose coordinates are, by definition, of the form $(x_L(y'), y')$ for
some $y'<y$. The two points $(x_R(y), y)$ and $(x_L(y'), y')$ are the opposite corners of a square, which we
will denote by $Q(y',y)$ in the rest of this proof; see Figure~\ref{fig:sqDiag} (left).
Notice that this square is always included in the strip between the curves $\gamma_L$ and $\gamma_R.$ 

\bigskip
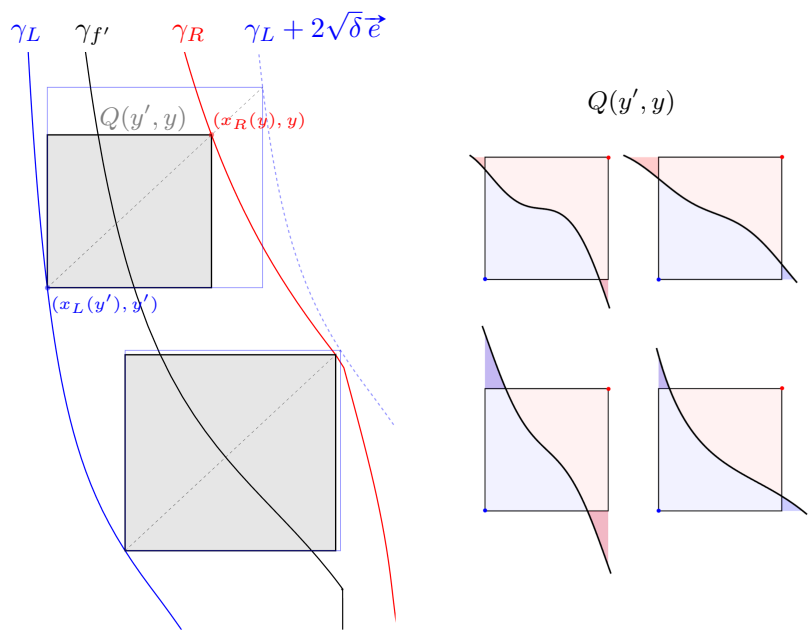
\begin{figure}[H]\label{fig:diagonale}
\begin{center}
\resizebox{.75\linewidth}{!}{%
\setlength{\unitlength}{1cm}
\includegraphics[width=10cm]{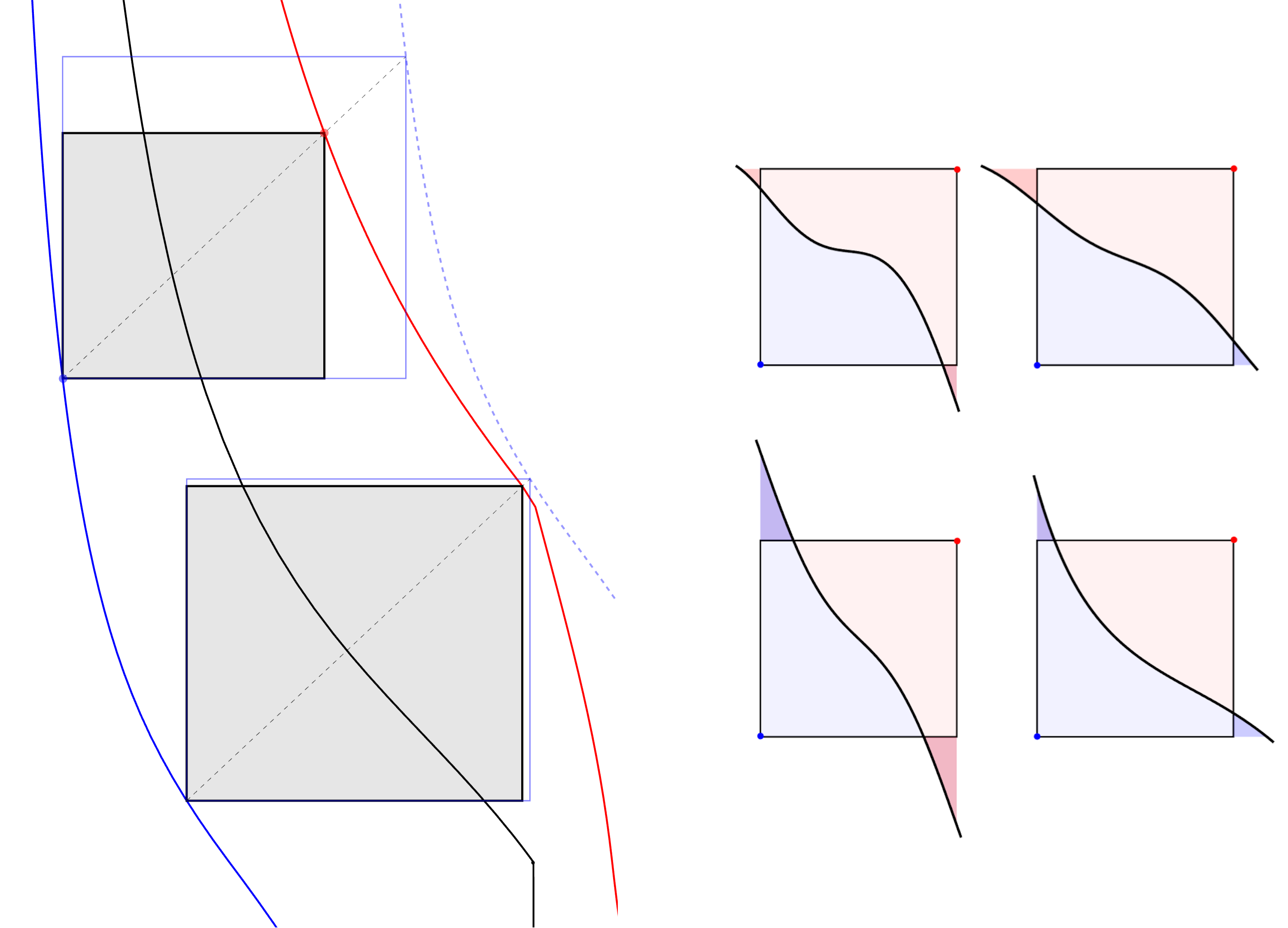}
\begin{picture}(0, 0)(10.1,0)
	\put(0, 7.4){\color{blue}\small $\gamma_L$}
	\put(0.8, 7.4){\color{black}\small $\gamma_{f'}$}
	\put(2.0, 7.4){\color{red}\small $\gamma_R$}
	\put(2.9, 7.4){\color{blue}\small $\gamma_L+2\sqrt{\delta} \protect \vv{e}$}
	\put(1.1,6.3){\color{gray} \footnotesize $Q(y',y)$}
	\put(2.5,6.3){\color{red} \tiny $(x_R(y),y)$}
	\put(0.5,4){\color{blue} \tiny $(x_L(y'),y')$}
	\put(7.15,6.5){\color{black} \footnotesize $Q(y',y)$}
\end{picture}
}
\captionsetup{width=\linewidth}
\caption{\label{fig:sqDiag}%
{\bf Left}: An example of squares $Q(y',y)$
with opposite corners $(x_L(y'), y')\in\gamma_L$ and $(x_R(y), y)\in\gamma_R$.
According to Lemma~\ref{lemma:est-diagonal}, their diagonal is bounded by $2\sqrt{\delta}$:
the curve $\gamma_R$ is below the offset $\gamma_L+2\sqrt{\delta} \protect \vv{e}$ (dashed blue) 
where $\protect \vv{e}=(1,1)/\sqrt{2}$ is the unit vector along the diagonal.
In this example, note how tight the estimate is near the point where~$f''(x)\simeq -1$.\\[0.5ex]
{\bf Right}: The four possible configurations corresponding to how $\gamma_{f'}$ can enter or exit $Q(y',y)$. 
The area of $Q(y',y)$ complemented by the two highlighted triangles is, by construction, exactly~$2\delta$.
}
\end{center}
\end{figure}

The curve $\gamma_{f'}$ can only enter the square on its left or upper side and
can only leave the square on its right or bottom side,
as seen in Figure~\ref{fig:sqDiag} (right). 
In all 4 cases, one has
\begin{align*}
Q(y',y)\cap \Gamma_{f'}^-  &\subset T_{f'}(y'; x_L(y')) \,,\\
Q(y',y)\cap \Gamma_{f'}^+ &\subset T_{f'}(y;x_R(y)) \,,
\end{align*}
thus $Q(y',y)\subset T_{f'}(y'; x_L(y')) \cup T_{f'}(y;x_R(y))$.
As this is a measurably disjoint union of 2~triangles of area at most $\delta$ (the area of $\gamma_{f'}$ is zero),
the area of the square $Q(y',y)$ is smaller than 
or equal to $2\delta$ and consequently its diagonal is smaller than or equal to~$2\sqrt{\delta}$. 
\end{proof}

From this point on, the idea is to use Fubini's theorem to slice the strip
between the curves $\gamma_L$ and $\gamma_R$ along the first diagonal.
In this direction, according to Lemma~\ref{lemma:est-diagonal}, the girth does not exceed $2\sqrt{\delta}$
and the decay of the integrands will ensure the integrability.
We prepare this computation by a suitable change of variable.

\begin{lemma}\label{lemma:cgt-var45}
Let us denote by $\Omega$ the strip between the curves $\gamma_L$ and $\gamma_R$.  One has
\begin{equation}\label{eq:cgt-var45-Omega} 
\int_{ - \frac \pi 2}^{\frac \pi 2} L(f, \delta, \theta ) \cos \theta  \, d\theta   =
\iint_{\Omega} \frac{dxdy}{1+y^2} =
\iint_{\Omega'} \frac{dx_1dx_2}{1+\frac{1}{2}(x_1-x_2)^2} \,,
\end{equation}
\begin{equation}\label{eq:cgt-var45-Omega-bis} 
\int_{ - \frac \pi 2}^{\frac \pi 2} L(f, \delta, \theta ) \cos^2 \theta  \, d\theta   =
\iint_{\Omega} \frac{dxdy}{(1+y^2)^{3/2}} =
\iint_{\Omega'} \frac{dx_1dx_2}{\left(1+\frac{1}{2}(x_1-x_2)^2\right)^{3/2}} \,,
\end{equation}
where $\Omega'=R_{\pi/4} \left(\Omega\right)$ is the image of $\Omega$ by the 
rotation of angle $\pi/4$ that maps $(1/2,0)$ to the origin. 
More generally, for any positive weight $\omega \in L^1_{\text{loc}}\left(-\frac{\pi}{2},\frac{\pi}{2}\right)$
such that $\omega\left(\pm(\frac{\pi}{2}-t)\right)\leq C \left|\ln t\right|^{-\beta}$ with $\beta>1$ as $t\to0^+$, one has
\begin{equation}\label{eq:cgt-var45-Omega-ter}
\int_{ - \frac \pi 2}^{\frac \pi 2} L(f, \delta, \theta ) \, \omega(\theta)  \, d\theta 
=\iint_{\Omega} \frac{\omega(\arctan y)}{\sqrt{1+y^2}}  \, dxdy
=\iint_{\Omega'} \frac{\omega\left(\arctan \frac{x_2-x_1}{\sqrt{2}}\right)}{\sqrt{1+\frac{1}{2}(x_1-x_2)^2}}  \, dx_1dx_2 \,.
\end{equation}
\end{lemma}

\begin{proof} 
In view of Lemma \ref{lemma:conc-change-of-var}
and the geometric construction above (and because $x_R$ and~$x_L$ are continuous), we can re-write the integral 
with Fubini's theorem, which gives the first identities.
The change of variable is the composition of a translation by $(-1/2,0)$
that moves the domain $\Omega$ into the strip $[-1/2, 1/2] \times \R$
followed by a rotation around the origin of angle $\frac{\pi}{4}$.
We denote by $\Omega'$ the image of $\Omega$ by this isometry (see Figure~\ref{fig:Omega'}).
The new coordinates $(x_1,x_2)\in\Omega'$ are thus related to the old ones by
\[
x=\frac{1}{2} + \frac{1}{\sqrt{2}}(x_1 + x_2)
\qquad\text{and}\qquad 
y=\frac{1}{\sqrt{2}}(x_2-x_1) \,.
\]
The Jacobian determinant is obviously equal to $1$.
For~\eqref{eq:cgt-var45-Omega-ter}, the assumption on the weight $\omega$ ensures that $\omega(\arctan y)\leq C \ln^{-\beta} |y|$
at infinity, which in turn ensures the integrability thanks to Bertrand's criterion.
\end{proof}

\begin{remark}
In Section~\ref{sect:algo}, we apply~\eqref{eq:cgt-var45-Omega-ter} with a bounded regular even weight that decreases
away from zero and vanishes at the endpoints~$\pm\pi/2$.
In particular, the assumption will be satisfied because $\omega\left(\frac{\pi}{2}-t\right)\leq C |t| \ll \left|\ln t\right|^{-2}$ as $t\to0$.
\end{remark}

\subsection{Proof of the upper bounds in Theorems \ref{thm:conc} and~\ref{thm:conc2}}
Using the previous lemmas we can now prove the upper bounds stated in Theorems~\ref{thm:conc} and~\ref{thm:conc2}.
As $\Omega$ is a subset of $[0,1]\times\R$, 
$\Omega'$ lies between the lines $(\pm\frac{\sqrt 2}{2}, 0) + \R\cdot(1,-1)$.
Therefore, 
\[
\forall (x_1, x_2)\in \Omega',\qquad
-\frac{\sqrt 2}{2}-x_1 \leq x_2 \leq \frac{\sqrt 2}{2} -x_1\,.
\]
Consequently, 
\begin{equation}\label{est:hOmega}
-2\left(\frac{\sqrt 2}{4} +x_1\right) \leq x_2-x_1 \leq 2\left(\frac{\sqrt 2}{4} -x_1\right) .
\end{equation}
Moreover, Lemma \ref{lemma:est-diagonal} ensures that for every $x_1\in \R$, the length of any vertical section
of $\Omega'$ is bounded in the following way:
\begin{equation}
\label{est:section1-Omega}
\forall x_1\in\R,\qquad
\left| \{x_2 \,;\, (x_1, x_2)\in \Omega'\} \right| \leq 2\sqrt{\delta}\,.
\end{equation}

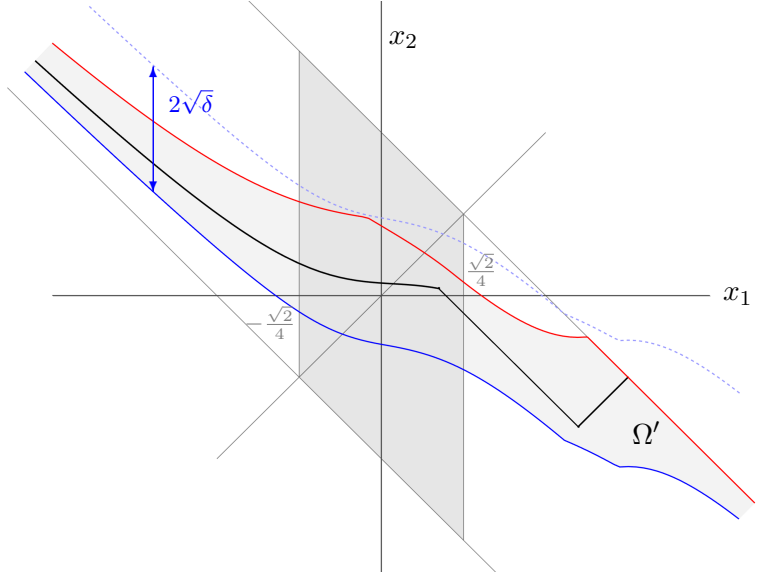
\begin{figure}[H]
\begin{center}
\setlength{\unitlength}{1cm}
\includegraphics[width=10cm]{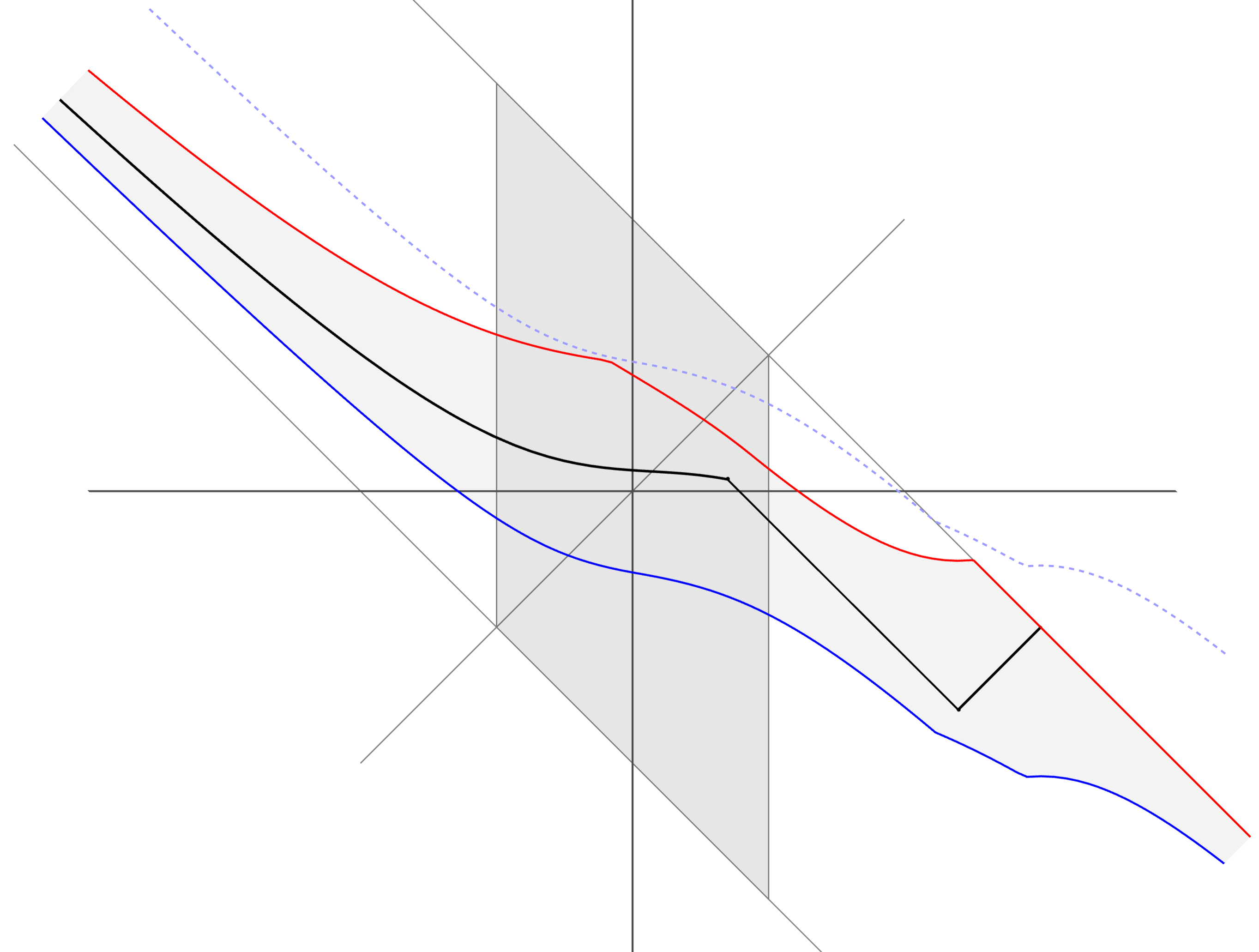}
\begin{picture}(0, 0)(10.1,0)
	\put(8.3,1.7){$\Omega'$}
	\put(2,5.04){\color{blue}\vector(0,1){1.66}}
	\put(2,6.7){\color{blue}\vector(0,-1){1.66}}
	\put(2.2,6.1){\color{blue}\footnotesize $2\sqrt{\delta}$}
	\put(9.5,3.6){$x_1$}
	\put(5.1,7){$x_2$}
	\put(6.1,3.9){\color{gray}\footnotesize $\frac{\sqrt{2}}{4}$}
	\put(3.2,3.2){\color{gray}\footnotesize $-\frac{\sqrt{2}}{4}$}
\end{picture}
\captionsetup{width=.95\linewidth}
\caption{\label{fig:Omega'}%
The length of the vertical sections of $\Omega'=R_{\frac{\pi}{4}}(\Omega)$ do not exceed~$2\sqrt{\delta}$.}
\end{center}
\end{figure}

We now split $\Omega'$ into three parts $\Omega'=\Omega'_\text{I} \cup \Omega'_{\text{II}} \cup \Omega'_\text{III}$ with
\[
\Omega'_\text{I} = \Omega' \cap \left\{x_1\leq -\frac{\sqrt{2}}{4}\right\},
\quad
\Omega'_\text{II} = \Omega' \cap \left\{x_1\geq \frac{\sqrt{2}}{4}\right\}
\quad\text{and}\quad
\Omega'_\text{III} = \Omega' \cap \left\{-\frac{\sqrt{2}}{4}<x_1<\frac{\sqrt{2}}{4}\right\}\,.
\]
On $\Omega'_\text{I}$ we have $\frac{\sqrt{2}}{4} +x_1\leq 0$ and all the terms that appear in the estimate~\eqref{est:hOmega}
are positive. Taking the square of the left-hand side gives $1+2\left(x_1+\frac{\sqrt 2}{4}\right)^2 \leq 1+\frac{1}{2}(x_2-x_1)^2$
and therefore 
\[
\iint_{\Omega'_\text{I}} \frac{dx_1dx_2}{1+\frac{1}{2} (x_1-x_2)^2}  \leq
\iint_{\Omega'_\text{I}} \frac{dx_1dx_2}{1+2\left(x_1+\frac{\sqrt 2}{4}\right)^2}\cdotp
\]
Notice that the function to integrate on the right-hand side does not depend on $x_2$.
Using Fubini's theorem and the estimate~\eqref{est:section1-Omega} we have
\[
\iint_{\Omega'_\text{I}} \frac{dx_1dx_2}{1+\frac{1}{2} (x_1-x_2)^2}  \leq
2\sqrt{\delta}\int_{-\infty}^{-\frac{\sqrt 2}{4}} \frac{dx_1}{1+2\left(x_1+\frac{\sqrt 2}{4}\right)^2} \,\cdotp
\]
With a change of variable $t=\frac{1}{2}+\sqrt{2}x_1$ we obtain 
\begin{equation}
\label{est:Omega'_I}
\iint_{\Omega'_\text{I}} \frac{dx_1dx_2}{1+\frac{1}{2} (x_1-x_2)^2}  \leq
\frac{\pi}{\sqrt{2}} \, \sqrt{\delta} \,.
\end{equation}
On $\Omega'_\text{II}$ we have that $\frac{\sqrt{2}}{4} -x_1\leq 0$ and a similar computation to the one on $\Omega'_\text{I}$
leads to 
\begin{equation}\label{est:Omega'_II}
\iint_{\Omega'_\text{II}} \frac{ dx_1dx_2}{1+\frac{1}{2} (x_1-x_2)^2} \leq
2\sqrt{\delta} \int_{\frac{\sqrt 2}{4}}^{\infty} \frac{dx_1}{1+2\left(x_1-\frac{\sqrt 2}{4}\right)^2}
=\frac{\pi}{\sqrt{2}} \, \sqrt{\delta} \,.
\end{equation}
On $\Omega'_\text{III}$, the decay of the integrand is negligible so we use $\frac{1}{1+\frac{1}{2} (x_1-x_2)^2}\leq 1$.
The geometric estimate~\eqref{est:section1-Omega} of the length of the vertical slices provides
\begin{equation}\label{est:Omega'_III}
\iint_{\Omega'_\text{III}} \frac{ dx_1dx_2}{1+\frac{1}{2} (x_1-x_2)^2} \leq
2\sqrt{\delta} \times 2\frac{\sqrt{2}}{4} = \sqrt{2\delta}\,.
\end{equation}
We put together the estimates \eqref{est:Omega'_I}, \eqref{est:Omega'_II}, \eqref{est:Omega'_III} 
 into the expressions given by Lemma~\ref{lemma:cgt-var45} and conclude that 
\[
\int_{ - \frac \pi 2}^{\frac \pi 2} L(f, \delta, \theta ) \cos \theta  \, d\theta  \leq (1+\pi)\sqrt{ 2 \delta} \,.
\]
Normalizing by $1/\pi$ gives~\eqref{main_estim} with the numerical constant $\frac{1+\pi}{\pi} \sqrt{2}\approx 1.86437$.
A similar computation can be performed for the second integral:
\begin{align*}
\int_{ - \frac \pi 2}^{\frac \pi 2} L(f, \delta, \theta ) \cos^2 \theta  \, d\theta  &= 
\iint_{\Omega'_\text{I} \cup \Omega'_\text{II} \cup \Omega'_\text{III}} \frac{dx_1dx_2}{\left(1+\frac{1}{2}(x_1-x_2)^2\right)^{3/2}}
\\
&\leq 2\sqrt{\delta}\left(
2\times \int_{\frac{\sqrt 2}{4}}^{\infty} \frac{dx_1}{\left(1+2\left(x_1-\frac{\sqrt 2}{4}\right)^2\right)^{3/2}}
+2\times \frac{\sqrt{2}}{4}
\right).
\end{align*}
The right-hand side is equal to $3\sqrt{2\delta}\approx \pi \times 1.35047\sqrt{\delta}$, as claimed by~\eqref{main_estim_bis}.

\medskip
Similarly, for a general even and positive weight function~$\omega$ on~$\left(-\frac{\pi}{2},\frac{\pi}{2}\right)$
that is decreasing on $[0,\pi/2)$, thanks to~\eqref{est:hOmega}, one has on $\Omega'_\text{I}\cup \Omega'_\text{II}$:
\[
\omega\left(\arctan \frac{x_2-x_1}{\sqrt{2}}\right) \leq \omega\left(\arctan\left( \frac{1}{2}-\sqrt{2}\, |x_1| \right)\right).
\]
The estimate~\eqref{main_estim_ter} follows immediately, provided
that the constant~\eqref{main_estim_ter_techsupport} is finite (\eg under the assumptions stated in Lemma~\ref{lemma:cgt-var45},
which are recalled in Theorem~\ref{thm:conc2}).
One can easily check that this estimate boils down to the previous~\eqref{main_estim} when $\omega(\theta)=\cos\theta$
and to~\eqref{main_estim_bis} when~$\omega(\theta)=\cos^2\theta$.

\subsection{Proof of the lower bounds in Theorem \ref{thm:conc}}\label{par:thmconcLower}

We end this section with computations on particular cases that 
assert the quasi-optimality of the constants from Theorem \ref{thm:conc}.
The best (\ie highest) lower bound is given by the second example, however the others
are instructive for getting a feel for which cases are the least favorable to our algorithm (see Section~\ref{sect:algo}).

\paragraph{Example 1 :} If $f(x)=ax+b$ for some $a,b\in \R$, then $S(f, \delta)$ is a parallelogram.
Let us introduce $\alpha = \arctan(a)$ and $\theta _0, \theta _1\in (0,\frac{\pi}{2})$ the geometric angles that
the diagonals make with the long sides of the parallelogram. 
Let us reason with $a\geq0$ as in Figure~\ref{fig:exemple1}.

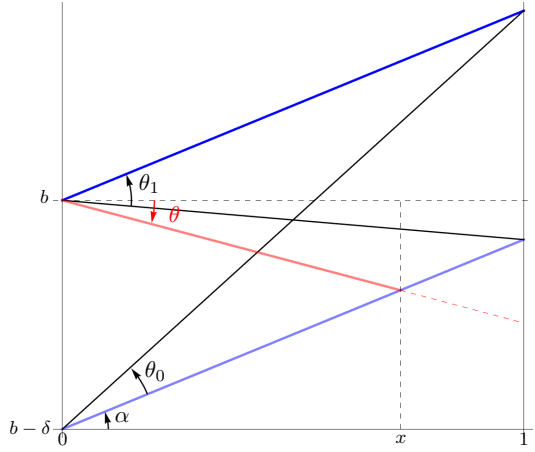
\begin{figure}[H]
\begin{center}
\resizebox{0.45\linewidth}{!}{%
\setlength{\unitlength}{1cm}
\includegraphics[width=8cm]{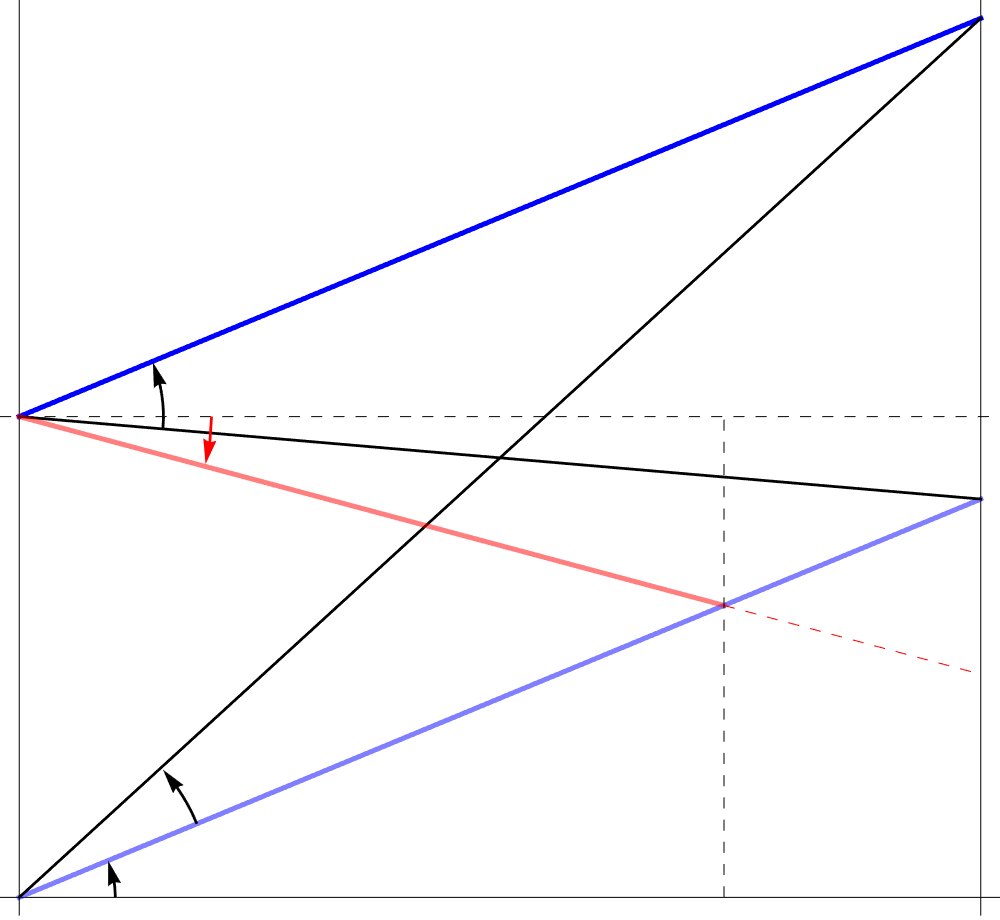}
\begin{picture}(0, 0)(8.1,0)
	\put(1.4,4.2){$\theta_1$}
	\put(1.5,1.05){$\theta_0$}
	\put(1,0.3){$\alpha$}
	\put(0.045,-0.08){\footnotesize $0$}
	\put(5.68,-0.05){\footnotesize $x$}
	\put(7.74,-0.08){\footnotesize $1$}
	\put(-0.25,3.98){\footnotesize $b$}
	\put(-0.76,0.1){\footnotesize $b-\delta$}
	\put(1.9,3.67){\color{red} $\theta$}
\end{picture}
}
\captionsetup{width=.75\linewidth}
\caption{\label{fig:exemple1}%
Case of $f(x)=ax+b$. The red segment is of length $ L(f, \delta, \theta )$.
The graphic corresponds to $\theta  \in (-\frac{\pi}{2}, \alpha -\theta _1)$
and illustrates the identity $\delta + x\tan(\theta ) =x\tan (\alpha)$
satisfied by $x=L(f, \delta, \theta ) \cos \theta $.
}
\end{center}
\end{figure}

One has $\tan(\alpha + \theta _0)=a+\delta$ and $\tan(\alpha - \theta _1)=a-\delta$ and
for $\theta \in (\alpha - \theta _1, \alpha + \theta _0)$ we have $L(f, \delta, \theta ) \cos \theta =x_R(\theta ) -x_L(\theta )=1.$
For $\theta  \in \left(-\frac{\pi}{2}, \alpha -\theta _1\right)$ the length of the projection $\ell=L(f, \delta, \theta ) \cos \theta $
satisfies $\delta + \ell \tan(\theta ) =\ell \tan (\alpha)$,
\ie $\ell= \frac{\delta}{\tan \alpha - \tan \theta }\cdotp$
Similarly, for $\theta  \in (\alpha +\theta _0, \frac{\pi}{2})$ we have $L(f, \delta, \theta ) \cos \theta = \frac{\delta}{\tan \theta  - \tan \alpha} \cdotp$
Splitting the integral thus gives
\[
\int_{ - \frac \pi 2}^{\frac \pi 2} L(f, \delta, \theta ) \cos \theta  \, d\theta  =
\int_{ -\frac{\pi}{2}}^{\alpha-\theta _1}  \frac{\delta  \, d\theta  }{\tan \alpha - \tan \theta } + \theta _0+\theta _1
+\int_{\alpha+\theta _0}^{\frac{\pi}{2}} \frac{\delta  \, d\theta }{\tan \theta  -\tan\alpha}\,\cdotp
\]
This integral is easiest to compute when $a=\alpha=0$ \ie when $f$ is a constant;
in this case one has $\theta_0=\theta_1\simeq \delta$ and
\[
\int_{ - \frac \pi 2}^{\frac \pi 2} L(f, \delta, \theta ) \cos \theta  \, d\theta  =
2\left( \int_{ 0}^{\theta _0} 1\, d\theta  + \int_{\theta _0}^{\frac{\pi}{2}}\frac{\delta}{\tan \theta } \, d\theta  \right)
=2(\theta _0 - \delta \log(\sin \theta _0))\,,
\]
which is of leading order $-2\delta \ln(\delta)\ll \sqrt{\delta}$.
In the general case, one has 
\[
\theta _0+\theta _1 =  \arctan(a+\delta)-\arctan(a-\delta) = \frac{2\delta}{1+a^2} + O(\delta^3)\,,
\]
and a primitive
\[
\int \frac{d\theta }{\tan\theta -\tan\alpha} = \frac{- a\theta  + \log\left| (a - \tan\theta )\cos\theta  \right| }{1+a^2}\,\cdotp
\]
One thus obtains
\begin{equation}\label{eq:ex1}
\int_{ - \frac \pi 2}^{\frac \pi 2} L(f, \delta, \theta ) \cos^2 \theta  \, d\theta 
\leq \int_{ - \frac \pi 2}^{\frac \pi 2} L(f, \delta, \theta ) \cos \theta  \, d\theta 
 = -\frac{2\delta \log\delta}{1+a^2} + O(\delta)\ll \sqrt{\delta}\,.
\end{equation}

\paragraph{Example 2 :}
Let us consider $f(x)=\sqrt{x(1-x)}$ for $x\in [0,1]$.
The graph of $f$ is a half circle of radius $1/2$;  the tangent at the origin is vertical.
Assuming $\delta<1/2$,
we denote by $\theta _0\in (0,\frac{\pi}{2})$ the angle of the tangent to the graph of $f-\delta$ that
passes through the origin and by $x_0$ the first coordinate of the tangence point. One has
$\tan \theta _0=f'(x_0)$ and $f(x_0) -\delta + (0-x_0)\tan \theta _0=f(0)$.
A simple computation provides $x_0 =\frac{4\delta^2}{1+4\delta^2}$ and
$\theta _0 =\arctan(\frac{1-4\delta^2}{4\delta})=\frac{\pi}{2}-4\delta+O(\delta^2)$.

\bigskip  
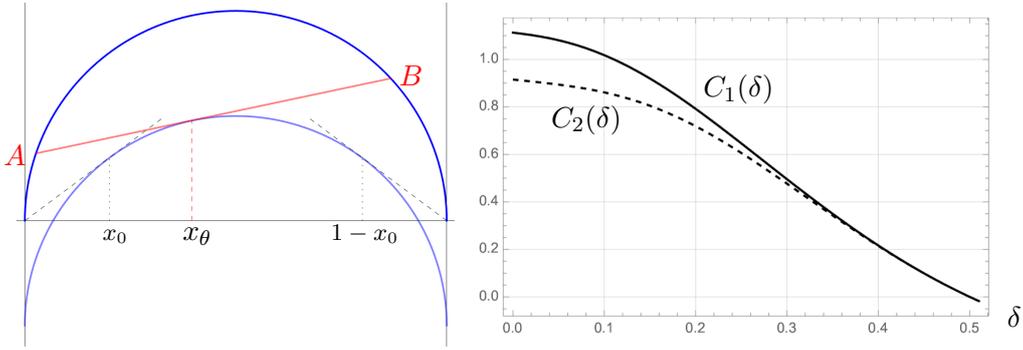
\begin{figure}[H]
\begin{center}
\setlength{\unitlength}{1cm}
\includegraphics[width=13cm]{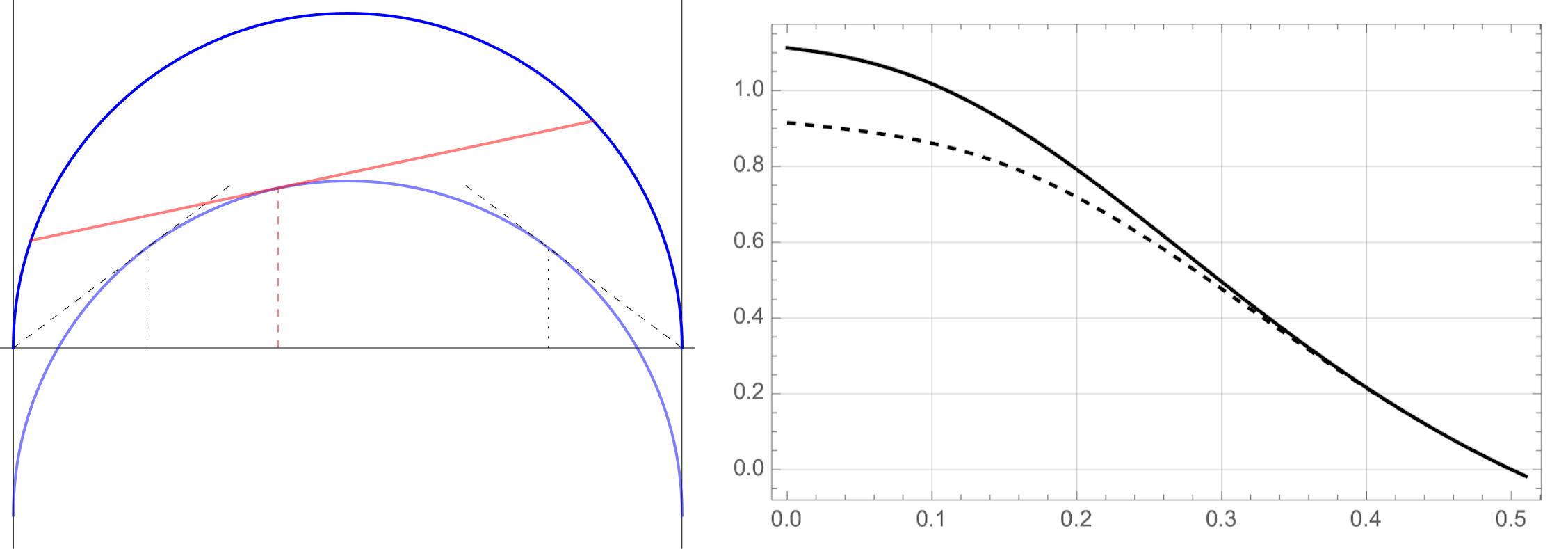}
\begin{picture}(0, 0)(13.1,0)
	\put(2.15,1.4){$x_\theta $}
	\put(1.1,1.4){\footnotesize $x_0$}
	\put(4.1,1.4){\footnotesize $1-x_0$}
	\put(-0.2,2.4){\color{red} $A$}
	\put(5,3.45){\color{red} $B$}
	\put(9,3.3){$C_1(\delta)$}
	\put(7,2.9){$C_2(\delta)$}
	\put(13,0.25){$\delta$}
\end{picture}
\captionsetup{width=.95\linewidth}
\caption{\label{fig:exemple2}%
Case of $f(x)=\sqrt{x(1-x)}$ from Example 2 (left).
The red segment $[AB]$ is of length $ L(f, \delta, \theta )$ and
the dashed lines mark the thresholds of the ``generic'' zone $\theta \in[-\theta _0,\theta _0]$.
The constants $C_1(\delta)$, $C_2(\delta)$ of the corresponding lower bounds (right) tend
to a non-zero value as $\delta\to0$ (see~\eqref{eq:defC1C2} below).
}
\end{center}
\end{figure}

For $\theta \in [-\theta _0, \theta _0]$ the longest segment $[AB]\subset S(f,\delta)$
of angle $\theta $ has both of its ends on the graph of~$f$
and is tangent to the graph of $f-\delta$.
Using the symmetry of the graph, one can assume that~$\theta \geq0$.
Let $(x_\theta ,f(x_\theta )-\delta)$ denote the point where $[AB]$ is tangent to the graph of $f-\delta.$
As $\tan \theta =f'(x_\theta )=\frac{1-2x_\theta }{2\sqrt{x_\theta (1-x_\theta )}}$, the equation of $[AB]$ gives
\[
4x_\theta ^2 - 4x_\theta  + \cos^2\theta =0.
\]
As $\theta \geq 0$ and $x_\theta <\frac{1}{2}$ by symmetry, then $x_\theta =\frac{1}{2}(1-\sin \theta ).$  
The abscissae $x_A$, $x_B$ of the endpoints satisfy
$f(x_\theta ) -\delta + (x-x_\theta )\tan \theta =f(x)$, which is equivalent to
\[
x^2 - x(1-\sin \theta  +\delta \sin2\theta  ) + \frac{1}{4}(1-\sin \theta  -2 \delta \cos \theta )^2=0\,.
\] 
Their difference $x_B-x_A=L(f, \delta, \theta ) \cos \theta $ is therefore given
by $x_B-x_A = \sqrt{\Delta}$ where $\Delta$ denotes
the discriminant, namely $\Delta = (1-\sin \theta  +\delta \sin2\theta )^2-(1-\sin \theta  -2 \delta \cos \theta )^2 =
4 \delta \cos^3 \theta  (1-\delta \cos\theta )$.
Using a similar estimate on $[-\theta _0,0]$ one gets the lower bounds
\[
\frac{1}{\pi}\int_{ - \frac \pi 2}^{\frac \pi 2} L(f, \delta, \theta ) \cos \theta  \, d\theta  \geq
 \frac{4\sqrt{\delta}}{\pi} \int_0^{\theta _0}\sqrt{1-\delta \cos\theta }   \cos^{3/2} \theta  \, d\theta  = C_1(\delta)\sqrt{\delta}
\]
and
\[
\frac{1}{\pi}\int_{ - \frac \pi 2}^{\frac \pi 2} L(f, \delta, \theta ) \cos^2 \theta  \, d\theta  \geq
 \frac{4\sqrt{\delta}}{\pi} \int_0^{\theta _0}\sqrt{1-\delta \cos\theta }   \cos^{5/2} \theta  \, d\theta  = C_2(\delta)\sqrt{\delta}\,.
\]
When $\delta\ll1$, the contributions to the integrals outside $[-\theta _0,\theta _0]$ are of a lower order. 
The dependence on $\delta$ of the lower bounds is illustrated in Figure~\ref{fig:exemple2}.

The functions $C_1$ and $C_2$ are continuous and strictly decreasing, with
\begin{equation}\label{eq:defC1C2}
C_1(0) = \frac{4\sqrt{2}K(1/2)}{3\pi} > 1.11283\,, \qquad
C_2(0) = \frac{12 \sqrt{2} \Gamma(3/4)^2}{5\pi^{3/2}} > 0.915311\,,
\end{equation}
where $K$ and $\Gamma$ are classical special functions
(respectively the complete elliptic integral of the first kind and the Gamma function).

\medskip
See Section~\ref{par:optimal} for an adaptation of this example to polynomials that
saturate the upper-bound on the complexity of the \A{} algorithm, both theoretically and in practice.

\paragraph{Example 3 :} We study $f(x)=-a(x-\frac{1}{2})^2$ for $x\in [0,1]$ with $a>0$
and with $\delta<a/4$.
Notice the symmetry with respect to the line $x=1/2$ and that the maximum of $f-\delta$
is~$-\delta$ and is superior to~$f(0)=-a/4.$ 
We denote by $\theta _0\in (0,\frac{\pi}{2})$ the angle  between the $x$-axis and
the line tangent to the graph of $f-\delta$ that passes through $(0,-a/4)$ and
by $x_0$ the first coordinate of the point where this tangent intersects the graph of $f$.
Substituting $\tan \theta _0=f'(x_0)$ in the equation $f(x_0) -\delta + (0-x_0)\tan \theta _0=f(0)$
ensures that $x_0=\sqrt{\delta/a}$ and $\tan \theta _0 = a-2\sqrt{a\delta}.$ 

\bigskip
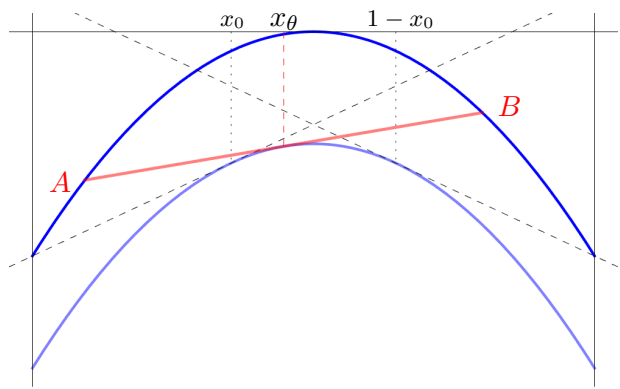
\begin{figure}[H]
\begin{center}
\setlength{\unitlength}{1cm}
\includegraphics[width=8cm]{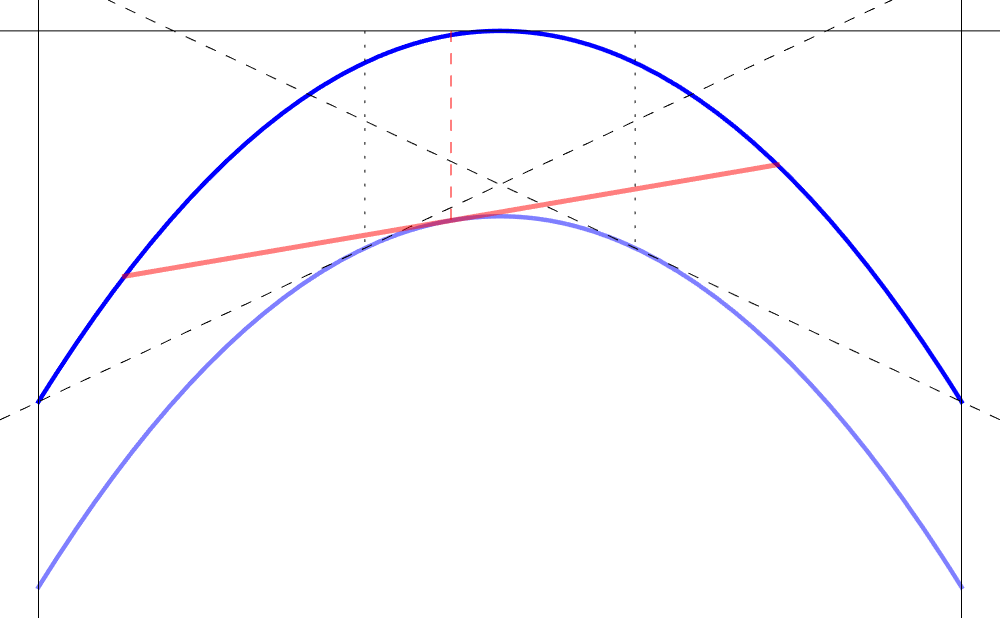}
\begin{picture}(0, 0)(8.1,0)
	\put(3.4,4.8){$x_\theta $}
	\put(2.75,4.8){\footnotesize $x_0$}
	\put(4.68,4.8){\footnotesize $1-x_0$}
	\put(0.5,2.6){\color{red} $A$}
	\put(6.4,3.6){\color{red} $B$}
\end{picture}
\captionsetup{width=.75\linewidth}
\caption{\label{fig:exemple3}%
Case of $f(x)=-a(x-\frac{1}{2})^2$ from Example 3. The markings are similar to those of Figure~\ref{fig:exemple2}.
}
\end{center}
\end{figure}

For $\theta \in (0, \theta _0)$ both ends of the longest segment $[AB]\subset S(f,\delta)$ of slope $\theta $
belong to the graph of $f$ and $[AB]$ is tangent to the graph of $f-\delta$.
As before, let $(x_\theta ,f(x_\theta )-\delta)$ denote the point where $[AB]$ is tangent to the graph of $f-\delta$
and $x_A$, $x_B$ the first coordinate of the endpoints.
Then $f'(x_\theta )=\tan \theta $ gives  $x_\theta =\frac{1}{2}-\frac{\tan \theta }{2a}$  
and the equation of $[AB]$ implies that $x_A$, $x_B$ satisfy
\[
a\left(x-\frac{1}{2}\right)^2 + \left(x-\frac{1}{2}\right) \tan \theta   +\frac{\tan^2 \theta }{4a}-\delta =0\,.
\]
The difference $x_B -x_A=L(f, \delta, \theta ) \cos \theta $
is thus given by $x_R-x_L = \sqrt{\Delta}/a$ with $\Delta =4a\delta$.
Using a similar estimate on $[-\theta _0,0]$ one gets the lower bounds
\[
\frac{1}{\pi}\int_{ - \frac \pi 2}^{\frac \pi 2} L(f, \delta, \theta ) \cos \theta  \, d\theta 
\geq \frac{2}{\pi} \int_0^{\theta _0} 2\sqrt \frac{\delta}{a}  \, d\theta  
= \frac{4\sqrt{\delta}}{\pi \sqrt a}  \arctan(a-2\sqrt{a\delta})\,,
\]
\[
\frac{1}{\pi}\int_{ - \frac \pi 2}^{\frac \pi 2} L(f, \delta, \theta ) \cos^2 \theta  \, d\theta 
\geq \frac{4\sqrt{\delta}}{\pi \sqrt a} \int_0^{\theta _0} \cos\theta   \, d\theta  
= \frac{4\sqrt{\delta}}{\pi \sqrt a}  \frac{a-2\sqrt{a\delta}}{\sqrt{1+(a-2\sqrt{a\delta})^2}}\,\cdotp
\]
Notice that for $\delta=a/4$ the right-hand side vanishes. For $\delta \ll a/4$, both bounds
are of order $\sqrt{\delta}$.
The first constant is approximately $\frac{4 \arctan(a)}{\pi \sqrt a}$,
whose maximum is $1.02288$ for $a\simeq 1.39175$. The second constant
becomes $\frac{4\sqrt{a}}{\pi(1+a^2)}$, whose maximum is $0.72559$ and is obtained
for $a\simeq 0.57735$.

\begin{remark}
The case $a=1/2$ in Example~3 corresponds to $f''(x)\equiv \textendash1$ for which
the estimate from Lemma~\ref{lemma:est-diagonal} is optimal (see Figure~\ref{fig:sqDiag}).
However, this example does not saturate the inequalities~\eqref{main_estim}-\eqref{main_estim_bis}.
\end{remark}

\section{Index of notations}\label{app:index}
We provide here a short index of our notations. By default, we use the American standard names, notations and spellings.

{\small
\paragraph{Numbers}
\begin{description}[itemsep=-0.3ex, labelsep=0pt, align=left]
\item $z=a+ib\in\C$ : complex numbers (with $a,b\in\R$).
\item $\lfloor\cdot\rfloor$ and $\lceil\cdot\rceil$ : resp.~floor and ceiling functions (round down/up to the next integer).
\item $\ln x$ :  natural logarithm.
\item $\log_b x = \frac{\ln x}{\ln b}$ : base-$b$ logarithm (for complexity, the default base is $b=2$). 
\item $x\simeq y$ : the numbers $x$ and $y$ have a similar order of magnitude (used colloquially).
\item $y\lesssim y$ : the order of magnitude of $x$ is smaller than or equal to that of $ys$ (used colloquially).
\end{description}

\paragraph{Asymptotic estimates} The asymptotic parameter $\sigma\to\sigma^\ast$ can be continuous or discrete and is given from context;
the signs (or complex phase) of $A,B$ are irrelevant.
\begin{description}[itemsep=-0.3ex, labelsep=0pt, align=left]
\item $A=O(B)$ : there exists a bounded function $C(\sigma)$ such that  $A(\sigma)=C(\sigma)B(\sigma)$.
\item $A\ll B$ : there exists a function $\varepsilon(\sigma)$ that tends to zero, such that $A(\sigma)=\varepsilon(\sigma) B(\sigma)$.
\end{description}

\paragraph{Sets}
\begin{description}[itemsep=-0.3ex, labelsep=0pt, align=left]
\item $\ii{m}{n}=\{m,m+1,\ldots,n-1,n\}$ : integer interval $[m,n]\cap\Z$.
\item $[a,b)$ : real-line interval, semi-open on the right side.
\item $\CC=\C\cup\{\infty\}$ : Riemann Sphere.
\item $\RR=\R\cup\{\infty\}$ : compaction of $\R$ into a circle.
\item $\# E$ : cardinal of a finite set.
\item $|E|$ : Lebesgue measure of a measurable set $E\subset\R^n$.
\end{description}

\paragraph{Polynomials}
\begin{description}[itemsep=-0.3ex, labelsep=0pt, align=left]
\item $\mathbb{K}[X]$ : set of polynomials with coefficients in the field $\mathbb{K}$ (typically $\R$ or $\C$).
\item $\mathbb{K}[[X]]$ : set of formal series with coefficients in the field $\mathbb{K}$.
\end{description}

\paragraph{Complexity} (see page~\pageref{def:complexity})
\begin{description}[itemsep=-0.3ex, labelsep=0pt, align=left]
\item $\eval_d$ : arithmetic complexity of evaluating a polynomial of degree $d$.
\item $\eval_d(k)$ : arithmetic complexity of $k$ polynomial evaluations of degree $d$.
\item $\eval_d(k,p)$ : bit complexity of evaluating a polynomial of degree $d$ on $k$ evaluation points with a fixed precision of $p$ bits
for all intermediary computations.
\item $M(p)$ : bit complexity of one multiply-add of two floating-point numbers with precision $p$.
\end{description}

\paragraph{Floating-point numbers} (see page~\pageref{equ:prec})
\begin{description}[itemsep=-0.3ex, labelsep=0pt, align=left]
\item $\xi = \pm 2^{n} \times 0.1\xi_1\xi_2 \ldots \xi_p$ : bit presentation of a floating point number.
\item $\ulp(\xi)$ : unit in the last place (smallest increment possible of the $p$-bit number $|\xi|$).
\end{description}

\paragraph{Floating-point representations of real and complex numbers} (see Section~\ref{par:equivModP})
\begin{description}[itemsep=-0.3ex, labelsep=0pt, align=left]
\item $s(z)$ : scale of a complex number (page~\pageref{equ:dsz}).
\item $x =_p y$ : $x,y\in\R$ have the same floating-point representation with precision $p \in \N^*$.
\item $x \simeq_p y$ : the $p$-bit floating-point representations of $x,y\in\R$ are identical or adjacent.
\item $z \approx_p z'$ : $z,z'\in\C$ have similar $p$-bits representations,
reduced by phase-shift invariance; see~\eqref{equ:eqpComplex}.
\end{description}

\paragraph{Concave geometry} (see Section~\ref{sect:analyse})
\begin{description}[itemsep=-0.3ex, labelsep=0pt, align=left]
\item subgraph : for a concave function $f$, region of the $(x,y)$-plane such that $y\leq f(x)$.
\item $S(f,\delta)$ : edge of the subgraph of $f$ of (vertical) thickness $\delta$.
\item $L(f, \delta, \theta)$ : maximal length of a segment of slope $\tan \theta $ contained in the strip $S(f, \delta)$.
\item Equation \eqref{eq:def_f_from_Ep} : definition of $f$ and $\delta$ in the FPE Algorithm application case.
\end{description}

\paragraph{FPE Algorithm} (see Section~\ref{sect:algo})
\begin{description}[itemsep=-0.3ex, labelsep=0pt, align=left]
\item $E_P:\ii{0}{d} \rightarrow \Z \cup \{-\infty\}$ : scales of the coefficients of the polynomial $P$.
\item $\ep:[0,d] \rightarrow \R \cup \{-\infty\}$ : concave cover of~$E_P$.
\item $\lambda = \log_2 |z| = -\tan \theta$ : dyadic scale of the evaluation point $z$.
\item $G_p$, $B_p$ : list of a-priori good (resp. ignored) coefficients for a given precision $p$.
\item $\ell$, $r$ : left/right edges to further reduce $G_p$ for a given $\lambda$.
\item $Q_\lambda(z)$ : reduced polynomial produced by the FPE algorithm.
\item $\A_p$ : new algorithm proposed in this article, for computations with a fixed precision $p$.
\item $\avg_{\CC}$ : average operator for $z$ uniformly distributed over $\CC$.
\item $\avg_{\RR}$ : average operator for $z$ uniformly distributed over $\RR$.
\item $\avg_{D(0,1)}$ : average operator for $z$ uniformly distributed over the unit complex disk.
\end{description}

}

\section{Listing of tasks implemented in~\cite{FPELib}}\label{FPETasks}

In our implementation~\cite{FPELib}, the tasks listed in this section are called in the command line with \texttt{FastPolyEval -task [arguments]}.
The first argument is systematically the precision of the computation, in bits.
Use \texttt{-task -help} for more detailed informations.

{\small\noindent%
\begin{longtable}{p{.20\textwidth} p{.80\textwidth}} 

    \multicolumn{2}{l}{\bf Tools for generating and handling polynomials}\\
   
    \tt -sum     &       computes the sum of two polynomials and writes the result to a CSV file\\
    \tt -diff        &   computes the difference of two polynomials\\
    \tt -prod     &      computes the product of two polynomials\\
    \tt -der        &    computes the derivative of a polynomial\\

    \tt -roots        &  computes the polynomial with a given list of roots\\
    \tt -Chebyshev &     writes the coefficients of the Chebyshev polynomial\\
    \tt -Legendre     &  writes the coefficients of the Legendre polynomial\\
    \tt -Hermite      &  writes the coefficients of the Hermite polynomial\\
   \tt  -Laguerre    &   writes the coefficients of the Laguerre polynomial\\
  \tt   -hyperbolic   &  writes the coefficients of the hyperbolic polynomial\\
    \\
    \multicolumn{2}{l}{\bf Tools for generating and handling sets of complex numbers}\\

    \tt -cat        &    concatenates two CSV files containing complex numbers\\
   \tt  -re      &       writes the real part of the list of complex numbers\\
   \tt  -im        &     writes the imaginary part of the list of complex numbers\\
  \tt   -conj        &   writes the conjugates of the list of complex numbers\\
  \tt   -join        &   joins the real part of two sequences into one sequence of complex numbers\\
 \tt    -tensor      &   computes the tensorial product of the two lists of numbers ($c_i = a_i \times b_i$)\\
  \tt   -grid        &   computes the set product of the real parts of two sequences\\
    \tt -exp        &    computes the complex exponential of a list of points\\
   \tt  -rot      &      maps complex numbers $(a, b)$ to $a*exp(ib)$\\

  \tt   -unif       &    writes real numbers in arithmetic progression\\
  \tt   -rand        &   writes real random numbers uniformly distributed in an interval\\
  \tt   -normal       &  writes real random numbers with Gaussian distribution\\
  \tt   -sphere  &       writes polar coordinates approximating a uniform distribution on the sphere\\
 \tt    -polar       &   computes the points given by polar coordinates on the sphere\\

   \tt  -comp    &       compares two lists of points\\
   \\ 
   \multicolumn{2}{l}{\bf Fast Polynomial Evaluator algorithm for production use and benchmarking}\\

   \tt  -eval   &        quickly evaluates a polynomial on a set of points\\
   \tt  -evalD   &       quickly evaluates the derivative of a polynomial on a set of points\\
   \tt  -evalN     &     quickly evaluates one Newton step of a polynomial on a set of points\\
  \tt   -iterN       &   quickly iterates the Newton method (partial search of roots of the polynomial)\\
  \tt -analyse  & computes the concave cover and the intervals of $|z|$ for which the evaluation strategy changes
\end{longtable}
} 

\bibliographystyle{alpha}\small
\bibliography{references}

\vspace*{3em}\noindent
$^{\small 1}$ \authorramona\\[1ex]
$^{\small 2}$ \authornicu\\[1ex]
$^{\small 3}$ \authorfv

\end{document}